\documentclass[reqno]{amsart}

\usepackage{amsmath}
\usepackage{amsfonts}
\usepackage{amsthm}
\usepackage{amssymb}
\usepackage{graphicx}
\usepackage[rightcaption]{sidecap}
\usepackage{bm}
\usepackage{dsfont}
\usepackage{color}
\usepackage{hyperref}
\usepackage{enumerate}
\usepackage{comment}
\usepackage[font=footnotesize]{caption}
\usepackage[all]{xy}
\usepackage{xy}
\usepackage{tikz}
\usepackage{paralist}
\allowdisplaybreaks
\usepackage{pgfplots}
\pgfplotsset{%
   every tick label/.append style = {font=\tiny},
   every axis label/.append style = {font=\scriptsize}
}

\usepackage{a4wide}

\numberwithin{equation}{section}

\newtheorem{thm}{Theorem}[section]
\newtheorem{mainthm}{Theorem}
\newtheorem{maincor}[mainthm]{Corollary}

\newtheorem*{conj*}{Conjecture}

\newtheorem{lem}[thm]{Lemma}

\theoremstyle{definition}

\newtheorem{dfn}[thm]{Definition}
\newtheorem{rmk}[thm]{Remark}

\begin{document}

\title[Periodic magnetic geodesics with low energy]{Periodic Magnetic Geodesics with Every Low Energy:\\ Existence and Localization}

\author[V.~Assenza]{Valerio Assenza}
\address{IMPA, Estrada Dona Castorina 110, Rio de Janeiro, 22460-320, Brazil}
\email{valerio.assenza@impa.br}
\author[G. Benedetti]{Gabriele Benedetti}
\address{SISSA, Mathematics Area, Via Bonomea 265, Trieste 34136, Italy}
\email{gbenedet@sissa.it}
\author[L.~Macarini]{Leonardo Macarini}
\address{IMPA, Estrada Dona Castorina 110, Rio de Janeiro, 22460-320, Brazil}
\email{leonardo@impa.br}

\date{\today}

\begin{abstract}
Let $(Q,g)$ be a Riemannian manifold equipped with a non-identically zero magnetic field represented by a closed $2$-form $\beta$. We allow $Q$ to be non-compact and do not assume that the metric $g$ is complete. We prove that if the magnetic strength, defined as the pointwise norm of $\beta$, attains a strict local maximum on a non-empty compact set $K$, then every sufficiently low energy level carries a contractible periodic magnetic geodesic of the pair $(g,\beta)$ localized near $K$. More precisely, such orbits exist in every neighborhood of $K$, and their lengths converge to zero with the energy. In particular, if $Q$ is compact, then every sufficiently small energy level carries a contractible periodic magnetic geodesic. We also show that, in general, neither the non-emptiness nor the compactness of $K$ can be omitted.

Our proof relies on the calculus of variations of the Lagrangian action functional, and uses several new ideas. More precisely, we overcome: (i) the non-exactness of $\beta$ by restricting the minimax to the set of short loops; (ii) the non-completeness of $g$ by combining a compactification of $Q$ with Thom's Jet Transversality and a blow-up for sequences of magnetic geodesics with energy tending to zero; (iii) the possible non-compactness of Palais--Smale sequences by the positivity of the Ricci magnetic curvature for low energy established by the first-named author. Unlike previous work, Struwe's monotonicity argument cannot be used for our purposes, and we rely on a two-Lyapunov-function argument due to Abbondandolo and Majer.
\end{abstract}

\maketitle

\tableofcontents

\section{Magnetic Geodesics in Classical Mechanics}\label{s:one}

This work studies the existence of periodic orbits of a charged particle in a stationary magnetic field. We find such periodic orbits for low energy near compact regions where the strength of the magnetic field has a strict local maximum, but we do not impose any compactness or geometric condition on the whole configuration space. If the configuration space is compact and the magnetic field is not identically zero, we deduce the existence of such orbits near the global maximum set of the strength.

We work here in the realm of classical mechanics. We refer the reader interested in relativistic or semiclassical mechanics to \cite{Caponio,Morin} and references therein. We start by describing the setup. 

\subsection{The Setting}Let $Q$ be a smooth finite-dimensional manifold (possibly non-compact), and let $g$ be a smooth Riemannian metric on $Q$ (possibly non-complete). Let $\beta$ be a smooth, closed two-form on $Q$, which we refer to as the magnetic form. Let $B$ be the Lorentz force endomorphism representing the bilinear form $\beta$ with respect to the metric $g$, that is,
\begin{equation}\label{e:endo}
\beta_q(u,v)=g_q(B_qu,v),\qquad \forall\,q\in Q,\ \forall\,u,v\in T_qQ.
\end{equation}
A magnetic geodesic $q\colon I\to Q$, defined on an interval $I$, is a non-constant solution of the second-order differential equation
\begin{equation}\label{e:mag}
D_t\dot q=B_q\dot q,
\end{equation}
where $\dot q$ denotes the tangent vector to $q$ and $D_t$ denotes the Levi-Civita covariant derivative along $q$. The name ``magnetic geodesic" comes from the fact that, when $Q$ has dimension two or three, these curves describe the motion of a charged particle subject to the law of inertia given by $g$ and the magnetic field given by $\beta$, see \cite{ARNOLD1,Kozlov,Burns}. Magnetic geodesics also arise when describing the dynamics of systems with rotational symmetry, such as spinning tops \cite{Abraham,Har,Novikov82} or the three-body problem \cite{Montgomery}, after quotienting out the symmetry. 

Equation \eqref{e:mag} admits a local Lagrangian formulation. More precisely, in every open set of $U\subset Q$ where $\beta=d\alpha$ is exact, the magnetic geodesics are the solutions of the Euler--Lagrange equation of the Lagrangian
\begin{equation}
L_\alpha(q,v)=\tfrac{1}{2}|v|^2_q-\alpha_q[v].
\end{equation}
Here we denoted by $|\cdot|_q$ the norm associated with $g_q$ for every $q\in Q$. By Legendre duality, we get a local Hamiltonian formulation determined by the Hamiltonian function 
\begin{equation}
H_\alpha(q,p)=\tfrac12|p+\alpha_q|^2_q
\end{equation}
and the standard symplectic form $\omega=\mathrm{d} p\wedge \mathrm{d}q$. By shifting the momenta by $\alpha$, the local Hamiltonian formulation can be globally described by the Hamiltonian function $H$ with respect to the symplectic form $\omega_\beta$, where
\[
H(q,p)=\tfrac12|p|_q^2
\qquad\mathrm{and}\qquad
\omega_\beta=\mathrm{d}p\wedge \mathrm{d}q-\pi^*\beta,
\]
see \cite{Souriau,Sternberg,Givental}. The Hamiltonian nature of the system, or simply the antisymmetry of the Lorentz force $B$, implies that the kinetic energy
\[
E(q,v)=\tfrac12|v|_q^2
\]
is constant along magnetic geodesics. We are therefore naturally led to study magnetic geodesics on a prescribed energy level $E=e=\tfrac12m^2$, where the positive parameter $m$ represents the speed of the curve.

Observe that, in the absence of a magnetic force, namely when $\beta=0$, equation~\eqref{e:mag} becomes homogeneous and reduces to the standard geodesic equation. The literature on this subject is extensive; we refer the reader, for instance, to \cite{GeodesicFlows_Paternain,Lecturesonclosedgeodesics}. Here, by contrast, we are interested in genuinely magnetic dynamics where the magnetic form $\beta$ does not vanish identically.

\subsection{High and Low Speed}The lack of homogeneity of equation~\eqref{e:mag} under rescaling implies that the magnetic dynamics may change drastically as the speed $m$ varies. Indeed, the Riemannian term on the left-hand side of~\eqref{e:mag} scales quadratically with $m$, whereas the magnetic term on the right-hand side scales linearly. This gives rise to a fundamental distinction between the low- and high-energy regimes: at high speed, the effect of the magnetic field becomes comparatively weak, whereas at low speed it strongly deflects the trajectories. Therefore, one expects a transition, as the speed decreases, from essentially free geodesic motion \cite{Siburg} to motion confined in the directions transverse to the kernel of the Lorentz force. We refer to  \cite{Arnold63,Braun,Arnold97,Castilho,San0,San1,AB2022_normalform} and references therein for extensive literature on the confinement problem for magnetic geodesics. 

When $Q$ is an oriented surface, magnetic dynamics and its dependence on the speed have a geometric description that has been popularized by Arnold since the 1960s \cite{ARNOLD1,Anosov}. In this case, \eqref{e:mag} is equivalent to the prescribed geodesic curvature equation
\begin{equation}
\kappa_q=\frac{b(q)}{m},
\end{equation}
where $\kappa_q$ is the geodesic curvature of $q$ with respect to $g$, and $b\colon Q\to\mathbb R$ is the function defined through
\begin{equation}\label{e:strength}
\beta=b\mu_g,
\end{equation}
where $\mu_g$ is the area form of $g$. In particular, if $Q$ is closed, $g$ is a hyperbolic metric, and $b=1$, then by \cite{Ginzburg96b}, we distinguish three cases:
\begin{enumerate}
    \item For $m>1$, the magnetic geodesic flow at speed $m$ is conjugated to the geodesic flow of $g$, up to a constant time reparameterization. In particular, there is exactly one periodic magnetic geodesic of speed $m$ in every non-trivial free homotopy class.
    \item For $m=1$, the magnetic geodesic flow at speed $m$ is the horocycle flow. Hence, the flow is minimal and every orbit is dense in $Q$. 
    \item For $0<m<1$, the magnetic geodesic flow at speed $m$ is a free circle action. In particular, all magnetic geodesics at speed $m$ are periodic and contractible.
\end{enumerate} 
\subsection{Periodic Magnetic Geodesics}
According to this example, one way to interpret the dichotomy between the high- and low-energy regimes is through the problem of finding periodic magnetic geodesics, also known as \textit{closed magnetic geodesics} in the literature, with prescribed speed $m$ \cite{Ginzburg96b}, a line of research that has been brought forward by Novikov in the 1980s \cite{Novikov82}. Intuitively, as $m$ decreases, there is a transition from topologically nontrivial periodic orbits, which arise naturally in standard geodesic flows when $Q$ is closed and not simply connected, to short periodic contractible orbits in the low-speed regime.

Correspondingly, the Lagrangian and Hamiltonian approach to the study of standard periodic geodesics generalize to high-energy magnetic geodesics. We refer, for example, to \cite{CIPP2000,Hofer-Viterbo88,Abbondandolo1} for the case of exact magnetic fields on closed manifolds, and we also point at \cite{Koh} for a heat-flow approach. For general magnetic fields, there exists a number $m(g,\beta)\in[0,\infty]$ called \textit{Mañé critical value} such that if $Q$ is compact, then there is a periodic magnetic geodesic for every speed $m>m(g,\beta)$ \cite{Osuna2005,Merry2010}. If $Q$ is not simply connected, such an orbit can be found in every prescribed nontrivial free-homotopy class of loops in $Q$, and carries geometric information in the Anosov case on surfaces \cite{ADMT2026}.

We give a precise definition of $m(g,\beta)$ in \eqref{e:mane}, but for now it is enough to mention that $m(g,\beta)=0$ if and only if $\beta=0$ and that $m(g,\beta)$ is finite if and only if the lift of $\beta$ to the universal cover admits a bounded primitive 1-form. In the hyperbolic surface example given above, we have $m(g,\beta)=1$. For the relevance of the Mañé critical value in symplectic geometry and Lagrangian mechanics, we refer to \cite{Mane0,Mane1,PP,CIPP1998,CIPP2000,Pat06,Pat09,CFP2010,Sorrentino}.

\section{Periodic Magnetic Geodesics with Low Energy: State of the Art}
In this work, instead, we study the existence of contractible periodic magnetic geodesics with speed $m<m(g,\beta)$, focusing on values of $m$ close to zero. In the subcritical regime, the study is considerably more subtle compared to the supercritical regime. Indeed, for $m>m(g,\beta)$ magnetic geodesics are reparametrizations of standard geodesics for a globally defined Finsler metric of Randers-type (upon, possibly, lifting the system to the universal cover of $Q$), while for $m<m(g,\beta)$, the Finsler interpretation is valid only locally \cite{CIPP1998}. 

We present our results in Section \ref{s:results}. To put them into context, in this section we recall the state of the art when $Q$ is compact.

\subsection{The Case of Surfaces}When $Q$ is a compact surface, the first results about periodic magnetic geodesics were obtained in the 1980s. Using the Hamiltonian formalism, Arnold \cite{Arnold86} and Ginzburg \cite{Ginzburg87} showed that if $\beta$ is nowhere vanishing, then there are short contractible magnetic geodesics for every $m$ small enough (more precisely, at least $2\,\mathrm{genus}(Q)+2$ of them under non-degeneracy assumptions and $\mathrm{cuplength}(Q)+1$ in general). Using the Lagrangian formalism, Novikov \cite{Novikov82} and Taimanov \cite{Taimanov92A,Taimanov92B} established the existence of periodic, embedded magnetic geodesics with low speed $m$ for oscillating magnetic fields, a class that includes exact magnetic fields. These periodic magnetic geodesics might not be contractible \cite{Tai15}, but they collectively bound an open region of $Q$. Contreras, Macarini, and G.~Paternain \cite{CMP} subsequently refined this argument showing the existence of this collection of periodic magnetic geodesics for every $m<m(g,\beta)$ when the magnetic form is exact, see also \cite{AM2019}. 

Several other authors established the existence of periodic magnetic geodesics at low speeds on surfaces providing additional information (such as embeddedness and lower bounds on the number of orbits), using a wide range of techniques. The papers \cite{Ketover,ChengZhou1,SarnataroStryker2026} use ideas from minimal surfaces. The papers \cite{Schneider1,ResenbergSchneider} use degree theory for immersed curves. The papers \cite{Ben16,GGM} use Floer theory. The papers \cite{AMP,AMMP,AB2015,AB2017,AABMT} use the Lagrangian action functional. The papers \cite{Castilho,AB2022} use the Birkhoff normal form. Finally, the existence of so-called Zoll magnetic systems where every magnetic geodesic of a given speed is periodic is discussed in \cite{AL,AB2022_normalform,ABB2024,BP} using different approaches such as the Nash--Moser implicit function theorem and twistors.

\subsection{Higher Dimension}
For compact manifolds $Q$ of any dimension, the existence of contractible periodic magnetic geodesics with low speed is a challenging problem, since the usual compactness arguments that hold for $m>m(g,\beta)$ break down. We refer here to the contact-type condition in the Hamiltonian case, and to the Palais--Smale property, which we discuss in more detail below, in the Lagrangian case \cite{CIPP2000}. These difficulties have stimulated the development of new ideas.

In \cite{Schlenk} Schlenk used the Hamiltonian approach and Floer theory to show that if $\beta$ is non-identically zero, then contractible periodic magnetic geodesics exist for almost every small $m$ (see also \cite{Polterovich,Macarini2004,FS} for earlier partial results in this direction). More precisely, this means that there exists $m'(g,\beta)$ in $(0,m(g,\beta)]$ and a subset of $(0,m'(g,\beta))$ with full Lebesgue measure such that for all $m$ in this subset there exists a contractible closed magnetic geodesic with speed $m$. 

Using Lagrangian variational methods, Contreras showed that one can take $m'(g,\beta)=m(g,\beta)$ when the magnetic field is exact \cite{Contreras2006,Tai2010,Abbondandolo1}, a restriction that was later dropped in \cite{Osuna2005, Merry2010,AsselleBenedetti2016,Groman}. The existence for almost every speed in a certain interval in the Hamiltonian and Lagrangian setting is an incarnation of a celebrated argument due to Struwe, which applies to variational settings that depend monotonically on a real parameter (in this case the speed $m$) \cite{Struwe88,Struwe90,Jean}.

Whether every subcritical value of $m$ carries a closed magnetic geodesic, either contractible or noncontractible, remains an open problem, but partial progress has been made for sufficiently low speed under additional assumptions on $\beta$. Using Floer theory, Ginzburg and G\"urel \cite{GG2004,GG2009} and Usher \cite{Usher} proved the existence of a contractible periodic magnetic geodesic for every sufficiently small $m$ under the assumption that $\beta$ is symplectic. We refer to \cite{Kerman1,KermanGinzburg} and \cite{CGK2004} for earlier results in this direction and to \cite{BBS} for a normal form in this setting.

More recently, the first-named author used the Lagrangian action functional to prove the existence of a contractible periodic magnetic geodesic for every sufficiently small $m$ when $\beta$ is merely nowhere vanishing \cite{Assenza2024}, a type of closed two-forms that exists on every manifold $Q$ of dimension at least $3$ by Thom Jet Transversality \cite[Chapter 3, Theorem 2.4]{Hirsch1976}. Both Hamiltonian and Lagrangian results fix the lack of compactness by showing that periodic orbits of given Conley--Zehnder index, respectively, Morse index have bounded period.

In the present paper, building on \cite{Assenza2024}, we will drop the nowhere-vanishing condition on $\beta$ and investigate the case of non-compact manifolds $Q$ as well. Let us present our results.

\section{Main Results of the Paper}\label{s:results}
The main result of this work establishes the existence of contractible periodic magnetic geodesics localized near a non-empty, compact local maximum of the magnetic strength for every low speed, see Definition \ref{d:strict}. Our argument allows $Q$ to be non-compact, a case where very little is known, see \cite{Zuddas0,Gong2023,Caldiroli2024,CoraMusina} for partial results. However, contrary to previous work, we do not need $g$ to be a complete metric. Our results are also new when $Q$ is compact, since they yield the existence of a contractible periodic magnetic geodesic for every low speed $m$ and for every magnetic form $\beta$ that does not vanish identically.

\subsection{Magnetic Strength and Strict Local Maxima}To precisely state the results, we introduce the following notions. Here and in the following, we will drop the subscript indicating the base point of the tangent space where a norm is taken. For instance, we will write $|v|:=|v|_q$ for all $q\in Q$ and $v\in T_qQ$.

\begin{dfn}
Let $g$ be a Riemannian metric and $\beta$ a magnetic form on $Q$ with associated Lorentz force $B$. We define the magnetic strength  at $q\in Q$ by
\begin{equation}
\label{eq:strength}
|\beta_q|:=|B_q|,
\end{equation}
where $|B_q|$ denotes the operator norm
\begin{equation}
|B_q|:=\sup_{|v|=1}|B_qv|.
\end{equation}
We denote by
\begin{equation}
\Vert\beta\Vert_\infty:=\sup_{q\in Q}|\beta_q|
\end{equation}
the supremum of the magnetic strength.
\end{dfn}

\begin{dfn}\label{d:strict}
We say that the magnetic strength attains a strict local maximum on $K\subset Q$ if there exists $b_K>0$ such that
\begin{itemize}
    \item $|\beta_q|=b_K$ for every $q\in K$;
    \item there exists an open neighborhood $V$ of $K$ such that $|\beta_q|<b_K$ for every $q\in V\setminus K$.
\end{itemize}
\end{dfn}
\subsection{The Statements}We are ready to state the first main theorem. 
\begin{mainthm}\label{t:A}
Let $(Q,g)$ be a Riemannian manifold and let $\beta$ be a magnetic field on $Q$ whose magnetic strength attains a strict local maximum $b_K>0$ on a nonempty compact set $K\subset Q$. Then, for every real number $\varepsilon>0$ and every neighborhood $V$ of $K$, there exists $m_{\varepsilon,V}>0$ such that, for every $m\in(0,m_{\varepsilon,V})$, there is a contractible periodic magnetic geodesic of speed $m$ and length at most $\frac{2\pi m}{b_K}(1+\varepsilon)$ contained in $V$.
\end{mainthm}
By continuity of the magnetic strength, if $Q$ is compact and $\beta$ is not identically zero, then its global maximum set is always non-empty and satisfies Definition~\ref{d:strict}. Therefore, Theorem~\ref{t:A} implies the following result.
\begin{maincor}\label{c:B}
Let $(Q,g)$ be a closed Riemannian manifold and $\beta$ a magnetic field on $Q$ such that $\beta$ is not identically zero. Then, for every real number $\varepsilon>0$ and every neighborhood $V$ of the set $\{q\in Q\ |\ |\beta_q|=\Vert\beta\Vert_\infty\}$ in $Q$, there exists $m_{\varepsilon,V}>0$ such that, for every $m\in(0,m_{\varepsilon,V})$, there is a contractible periodic magnetic geodesic of speed $m$ and length at most $\frac{2\pi m}{\Vert\beta\Vert_\infty}(1+\varepsilon)$ contained in $V$.
\end{maincor}
Corollary \ref{c:B} is sharp. Indeed, if $\beta$ is identically zero, then magnetic geodesics are standard geodesics, and there exist metrics (such as flat metrics on tori) with no contractible, non-constant periodic geodesics. Moreover, even if a metric admits contractible, non-constant periodic geodesics, their length is bounded away from zero and therefore it cannot converge to zero as $m$ goes to zero.

When $Q$ is a surface, the localization result of Theorem~\ref{t:A} was obtained in \cite[Lemma 5.2]{AB2022} using a local surface of section, and, when the local compact maximum is a non-degenerate critical point of the magnetic strength, in \cite{EFP2024} by means of the implicit function theorem. To the best of the authors' knowledge, the localization in arbitrary dimension and in this general setting is new.

Before outlining the proof of Theorem~\ref{t:A} and discussing the techniques involved, we further illustrate its generality by showing that the assumptions on the local maximum set cannot, in general, be dropped. Indeed, \cite{KirschLaurain2010} proves that if $Q=\mathbb R^2$ and
\begin{equation}
g=\mathrm{d}x^2+\mathrm{d}y^2,\qquad \beta=b\mathrm{d}x\wedge \mathrm{d}y,\qquad \frac{\partial b}{\partial x}>0,
\end{equation}
then there are no periodic magnetic geodesics at any speed. In this case, $|\beta_q|=|b(q)|$ and therefore the only possible local maximum is the empty set. In the next result, we also give a counterexample having a noncompact local maximum set.
\begin{mainthm}\label{t:B}
There exists a Riemannian metric $g$ on $\mathbb R^2$ such that, if $\beta=\mu_g$ is the area form of $g$, then there are no contractible periodic magnetic geodesics at any sufficiently small speed.
\end{mainthm}
The metric $g$ constructed in Theorem~\ref{t:B} is invariant under vertical translations and, hence, so is its area form $\mu_g$. Thus, the magnetic geodesic flow induced by $g$ and $\beta=\mu_g$ is integrable. To show that this example has the required property, we use the fact that the dynamics drifts in the direction of the Hamiltonian vector field on $\mathbb R^2$ with Hamiltonian function given by the Gaussian curvature $K=K(x)$ of $g$ \cite{Arnold97,AB2022_normalform}. We therefore construct $g$ so that $\frac{\mathrm{d}K}{\mathrm{d}x}>0$. Notice that in this example $Q=\mathbb R^2$ is the only nonempty local maximum set since the magnetic strength is constant equal to $1$. 

\subsection{From Noncompact to Compact Manifolds}
The main ingredient in the proof of Theorem \ref{t:A} is a weaker version of Corollary \ref{c:B} that upgrades the existence result of \cite{Assenza2024} for nowhere vanishing magnetic fields on compact manifolds to a localization result near the global maximum set of the magnetic strength.
\begin{mainthm}\label{t:C}
Let $(Q,g)$ be a closed Riemannian manifold and $\beta$ a magnetic field on $Q$ such that $\beta$ is nowhere zero. Then, for every real number $\varepsilon>0$ and every neighborhood $V$ of the set $\{q\in Q\ |\ |\beta_q|=\Vert\beta\Vert_\infty\}$ in $Q$, there exists $m_{\varepsilon,V}>0$ such that, for every $m\in(0,m_{\varepsilon,V})$, there is a contractible periodic magnetic geodesic of speed $m$ and length at most $\frac{2\pi m}{\Vert\beta\Vert_\infty}(1+\varepsilon)$ contained in $V$.
\end{mainthm}
In the next section, we sketch the proof of Theorem \ref{t:C}. In the remainder of this section, we explain how to deduce Theorem~\ref{t:A} from Theorem~\ref{t:C}. The key step is to realize a non-empty, compact strict local maximum as a global maximum for an auxiliary system whose magnetic form is nowhere vanishing and whose base manifold is compact. More precisely, suppose that the system $(g,\beta)$ has a nonempty compact strict local maximum set $K$. We proceed as follows.

Let $\widehat Q$ be the closed manifold obtained by doubling a neighborhood $V$ of $K$. We construct a system $(\hat g,\hat\beta)$ on $\widehat Q$ satisfying the following properties:
\begin{enumerate}[(i)]
    \item $(\hat g,\hat\beta)$ coincides with $(g,\beta)$ on a suitable neighborhood $V'\subset V$ of $K$;
    \item $|\hat\beta|$ attains its unique strict global maximum on $K$;
    \item $\hat\beta$ is nowhere vanishing.
\end{enumerate}
Properties~(i) and~(ii) are achieved by rescaling the metric away from $K$. Property~(iii) is obtained by applying the Thom Jet Transversality Theorem \cite[Chapter 3, Theorem 2.4]{Hirsch1976} to the closed two-form $\beta$ near the boundary of the neighborhood of $K$. Transversality for two-forms requires that $Q$ has dimension at least four, so that
\[
\binom{\dim Q}{2}>\dim Q,
\]
where on the left we have the codimension of the zero section in the total space of the bundle $\Lambda^2T^*\widehat Q\to \widehat Q$ of $2$-forms over $\widehat Q$ (equivalently, the rank of the bundle) and on the right we have the dimension of the zero section of the bundle.

When $Q$ has dimension three, we stabilize the system by lifting it to $\widehat Q\times S^1$ in such a way that contractible magnetic geodesics on $\widehat Q\times S^1$ project to contractible magnetic geodesics on $\widehat Q$ with the same speed \cite{Macarini2004}, and the global strict local maximum on $\widehat Q\times S^1$ projects to the strict global maximum on $\widehat Q$. Applying the localization Theorem~\ref{t:C} to $(\hat g,\hat\beta)$ on $\widehat Q$ and observing that the dynamics of $(g,\beta)$ and $(\hat g,\hat\beta)$ agree on $V'$, completes the proof of Theorem \ref{t:A}. 

\section{Theorem \ref{t:C}: Setting and Sketch of Proof}
In this section, we give a sketch of the proof of Theorem~\ref{t:C}. Our approach relies on the minimax method for the Lagrangian action functional for subcritical energies as described by Contreras in \cite{Contreras2006}, and building on the principle of throwing-out-cycles by Novikov and Taimanov \cite{Novikov82,Taimanov92A}. Thus, we start by recalling the setting and its difficulties, how several authors tackled these difficulties, and which new ideas we employed in the current paper. Throughout this section, $Q$ is assumed to be compact.
\subsection{The Lagrangian Action Functional and the Mañé Critical Value}
Our starting point to study periodic magnetic geodesics with energy 
\begin{equation}
e=\tfrac12m^2    
\end{equation}
is to identify them with the critical points of the free-period action functional $S_e$ on a connected component of the free-loop space of $Q$. For example, the action on the space $\Lambda_0 Q$ of contractible loops on $Q$, whose elements are loops $q\colon\mathbb R/T\mathbb Z\to Q$ bounding a disk $\hat q\colon\mathbb D\to Q$, is given by
\begin{equation}\label{e:sm1}
S_e(q):=\int_{\mathbb R/T\mathbb Z}\Big(\tfrac{1}{2}|\dot q(t)|^2_{q(t)}+e\Big)\mathrm{d}t-\int_{\mathbb D}\hat q^*\beta.
\end{equation}
We note that the period $T$ is not fixed a priori and that, although $dS_e$ is well-defined, $S_e$ may be multivalued, since its value can depend on the homotopy class of $\hat q$ relative to the boundary. In fact, $S_e$ is single-valued on $\Lambda_0 Q$ if and only if $\beta$ is weakly exact, that is, $\beta|_{\pi_2(Q)}=0$.

The action functional provides a precise way to distinguish between the high- and low-energy regimes through Mañé's critical value
\begin{equation}\label{e:mane}
e(g,\beta)=\tfrac12m(g,\beta)^2,\qquad e(g,\beta):=\inf\Big\{e\in\mathbb R\ \Big|\ S_e(q)\geq 0,\ \forall\,q\in\Lambda_0 Q\Big\}\in[0,\infty].
\end{equation}
One can show that if $e(g,\beta)<\infty$, then $S_e$ is single-valued in every free-homotopy class of loops upon substituting the integral over the disc in \eqref{e:sm1} with the integral over a cylinder starting at a reference loop \cite{Merry2010}. For $e>e(g,\beta)$, the existence theory for periodic magnetic geodesics resembles that of standard geodesics of Finsler metrics and, as observed above, the system is virtually Finsler. This means that the lifts of magnetic geodesics to the universal cover of $Q$ are the geodesics of a suitable Finsler metric, up to time reparametrization. From a variational point of view, periodic magnetic geodesics of every energy $e>e(g,\beta)$ can be found by looking for minimizers of the action functional $S_e$ in every non-trivial free-homotopy class of loops \cite{CIPP2000,Merry2010}. Indeed, for $e\geq e(g,\beta)$, $S_e$ is bounded from below on every free-homotopy class of loops, while for $e>e(g,\beta)$ it also satisfies the so-called Palais--Smale compactness condition \cite{PS}.

\subsection{The Palais--Smale Condition}The Palais--Smale condition is a crucial tool for establishing the existence of critical points of functionals on infinite-dimensional spaces by min-max methods. It guarantees that if $(q_k)$ is a sequence such that, for some $c_*\in\mathbb R$,
\begin{equation}\label{e:PS}
\mathrm{d}_{q_k}S_e\to0,\qquad S_e(q_k)\to c_*,
\end{equation}
then there exists a subsequence $(q_{k_j})$ converging to a critical point $q_*$ of $S_e$ with $S_e(q_*)=c_*$; see, for instance, \cite{Abbondandolo1}. For magnetic geodesics, the crucial step in establishing the Palais--Smale condition is to show that the sequence of periods $(T_k)$ of the loops $(q_k)$ is bounded. From \eqref{e:sm1}, the period $T$ of a loop $q$ can be recovered from the difference quotient of the action at two energies $e_1<e_2$:
\begin{equation}\label{e:Lip}
T=\frac{S_{e_2}(q)-S_{e_1}(q)}{e_2-e_1}.
\end{equation}
In the supercritical regime, one takes
\[
e_1:=e(g,\beta),\qquad e_2:=e>e(g,\beta).
\]
The upper bound on $T_k$ then follows from the upper bound on $S_{e_2}(q_k)$ and the fact that $S_{e_1}(q)\geq 0$ for every $q\in\Lambda_0 Q$.

For $e\in(0,e(g,\beta))$, the functional $S_e$ is no longer bounded from below on any connected component of the loop space. On the other hand, the fact that $\{S_e<0\}$ is a non-empty subset of the space of contractible loops $\Lambda_0 Q$ implies that $S_e$ has a \textit{mountain-pass geometry} on $\Lambda_0 Q$. Indeed, $0$ is a local infimum of $S_e$, which is approached on the subset of constant loops $Q_T\subset\Lambda_0 Q$ of period $T$, as the period $T$ converges to $0$. Therefore, considering paths of loops connecting the set of constant loops $Q_T$ with small $T$ to the sublevel set $\{S_e<0\}$ yields a corresponding \textit{minimax scheme} and \textit{minimax value} $c(e)$. However, it is currently unknown whether $S_e$ satisfies the Palais--Smale condition for $e\in(0,e(g,\beta))$. The difficulty is that, even if $S_e(q_k)$ converges, the kinetic and magnetic terms in \eqref{e:sm1} may both diverge. Thus, the problem of establishing the compactness needed for the minimax scheme to produce a critical point, formulated in the 1980s by Novikov and Kozlov \cite{Novikov82,Kozlov}, remained open for many years.

A breakthrough came in 2006, when Contreras \cite{Contreras2006} used Struwe's monotonicity argument \cite{Struwe90} to show that, when $\beta$ is exact, the minimax scheme produces a critical point of $S_e$ with action $c(e)$ at every point $e$ where the minimax function $e\mapsto c(e)$ is locally Lipschitz. The proof again uses \eqref{e:Lip}, now with \begin{equation*}
e_1=e,\qquad e_2=e+\varepsilon,    
\end{equation*}
where $\varepsilon>0$ depends on the Lipschitz constant at $e$. Since the family of functionals $e\mapsto S_e$ is increasing, the same is true for the family of minimax values $e\mapsto c(e)$. Hence, almost every $e\in(0,e(g,\beta))$, with respect to the Lebesgue measure, is a Lipschitzianity point of $c$.

Contreras' argument was later generalized from exact to weakly exact magnetic forms by Merry \cite{Merry2010}, and then to arbitrary magnetic forms in \cite{AsselleBenedetti2016}. Even in this general case, the minimax value $c(e)$ is well-defined since $S_e$ is still single-valued on the set of sufficiently short loops in $\Lambda_0 Q$ (since $Q$ is locally contractible), although it may be multivalued on the whole of $\Lambda_0 Q$.

\subsection{The Magnetic Curvature}
In view of Contreras' result, it is very natural to ask whether contractible magnetic geodesics exist for \emph{every} $e\in(0,e(g,\beta))$, or at least for every $e$ in an interval $(0,e'(g,\beta))$, for some $e'(g,\beta)\in(0,e(g,\beta)]$, and how one can overcome the possible failure of the Palais--Smale condition for $S_e$.

In this direction, a crucial idea was introduced by Bahri and Taimanov in 1998 \cite{BahriTaimanov}, inspired by a scheme developed by Sacks and Uhlenbeck for minimal surfaces \cite{SacksUhlenback}. The idea is that, although critical points $q$ of $S_e$ with prescribed action may have an arbitrarily large period $T$, the Morse index $\mathrm{ind}_e(q)$ of $q$ gives an upper bound for $T$, provided a suitable Ricci-type curvature associated with the system admits a positive lower bound $r$ which is uniform in a neighborhood of $e$. More precisely, the inequality
\begin{equation}\label{e:periodindex}
T\leq C_r\mathrm{ind}_e(q),
\end{equation}
holds for some constant $C_r>0$ depending on $r$. In Riemannian geometry, for example, an estimate of this type lies at the heart of the Bonnet--Myers theorem, which asserts that a complete Riemannian manifold with a positive lower bound on its Ricci curvature is bounded and hence compact \cite[Theorem 19.4]{Milnor}. In our setting, since the minimax is taken over a family of sets of dimension $1$ connecting the two valleys of $S_e$, one obtains an estimate $\mathrm{ind}_e(q)\leq1$, which, combined with \eqref{e:periodindex}, yields a bound on the period for the critical points.

Unfortunately, the Ricci-curvature-type quantity introduced by Bahri and Taimanov is typically not positive at low speeds. For example, it is not positive at low speed if the Ricci curvature of $g$ is somewhere negative. When $Q$ is a surface, it is positive at low speed if and only if the function $b$ is constant and the Gaussian curvature of $g$ is positive. In any case, this Ricci-curvature-type quantity does not have a positive lower bound as the energy goes to zero.

A Bahri--Taimanov-type idea was used a few years later by Ginzburg--Gürel \cite{GG2004,GG2009} and Usher \cite{Usher} in the Hamiltonian setting to obtain an estimate for the Conley--Zehder index analogous to \eqref{e:periodindex}. They used a sectional-curvature-type quantity that, with the additional assumption that $\beta$ is symplectic, is positive when the energy is low. In this way, they obtained the existence of contractible periodic magnetic geodesics at all sufficiently low energies when $\beta$ is symplectic.

Recently, the first-named author introduced in \cite{Assenza2024} a notion of magnetic curvature whose positivity is easier to study compared to the one of \cite{BahriTaimanov}. This notion is very natural since it governs magnetic Jacobi fields \cite{Assenza2025}, has applications to the Brunn--Minkowski inequality \cite{Assouline2,Assouline1}, connection with the hyperbolicity of the flow \cite{Gouda,Grognet,Gro99,Wojtkowski,Hasselblatt}, can be computed for left-invariant structures \cite{RTW}, and is a special case of a more abstract construction for Lagrangian systems \cite{Agrachev,Jakub,AssenzaTestolina2026} and frame flows \cite{Brahim2026}. When $Q$ is an orientable surface, the magnetic curvature was first introduced by Gabriel and Miguel Paternain in the 1990s \cite{PP96,Pat06}. In this case, it is represented by a function $K_{g,\beta,m}\colon SQ\to \mathbb R$ on the unit tangent bundle given by the formula
\begin{equation}\label{e:Kmag}
K_{g,\beta,m}(q,v)=m^2K_g(q)-m\,\mathrm{d}_qb[v^\perp]+b(q)^2,
\end{equation}
where $K_g$ is the Gaussian curvature of $g$, $v^\perp$ is the rotation of $v$ by ninety degrees, and $b$ is the function determined by $\beta=b\mu_g$, see \eqref{e:strength}. 

In \cite{Assenza2024}, the first-named author showed that in every dimension the magnetic Ricci tensor has a structure similar to \eqref{e:Kmag}: for every $q\in Q$, it is homogeneous of degree $2$ in the pair of variables $m$ and $B(q)$, and the part that is quadratic in $B(q)$ vanishes if and only if $B(q)=0$. This implies that the magnetic Ricci curvature is positive at all sufficiently low energies when $\beta$ is nowhere vanishing. Consequently, \cite{Assenza2024} extends the existence of contractible periodic magnetic geodesics at low energies when $\beta$ is nowhere vanishing.

\subsection{The Localization Problem} From the previous subsection, we conclude that when $Q$ is compact and $\beta$ is nowhere vanishing, the minimax method of \cite{Assenza2024} produces contractible periodic magnetic geodesics $q_m$ at speed $m$ whose periods $T_m$ have a uniform upper bound $T_+$ for $m$ in a neighborhood of $0$. Consequently, one obtains the corresponding upper bound
\begin{equation}
\ell(q_m)\leq T_+m
\end{equation}
for the lengths $\ell(q_m)$ of the orbits. Since $Q$ is compact, it follows that, as $m$ tends to zero, the orbits $q_m$ converge, after passing to a subsequence, to a point $q_*\in Q$.

This naturally leads to the question of how to characterize the point $q_*$, or, equivalently, how to localize the orbits $q_m$ for small $m$. This is a crucial question. Indeed, understanding this localization phenomenon in Theorem \ref{t:C} allowed us to extend in Theorem \ref{t:A} the existence of contractible periodic magnetic geodesics at low energies to magnetic fields that are allowed to vanish and to non-compact manifolds, settings in which very little was known.

The clue to this localization phenomenon came from the two-dimensional case. Indeed, if $Q$ is a surface and $q_*\in Q$ is a strict local maximum or minimum of the magnetic strength $b$ defined by \eqref{e:strength}, then for every neighborhood $V$ of $q_*$ there exists $m_V>0$ such that $V$ contains contractible magnetic geodesics of every speed $m\in(0,m_V)$ \cite[Lemma 5.2]{AB2022}. More generally, the same conclusion holds when the point $q_*$ is replaced by a compact strict local maximum or minimum $K\subset Q$ of the magnetic function $b$. 

Let us now present the five new ingredients that allowed us to prove Theorem \ref{t:C}, that is, the localization of the orbits $q_m$ at the global maximum $\{q\in Q\ |\ |\beta_q|=\Vert\beta\Vert_\infty\}$ when $Q$ is compact and $\beta$ is nowhere vanishing.

\subsection{First New Ingredient: Localization via the Action}
The crucial ingredient in the proof of Theorem \ref{t:C} is to localize the orbits $q_m$ obtained by the minimax construction at energy $e=\tfrac12m^2$ using their action value
\begin{equation}
\mathfrak c(m):=c(e)=S_e(q_m).
\end{equation}
In Lemma \ref{l:special}, by choosing a path
\begin{equation}\label{e:pathqr}
r\mapsto q_{m,r},\qquad r\in[0,1],    
\end{equation}
in the minimax class originating from a point $q_*\in Q$ with $|\beta_{q_*}|=\Vert\beta\Vert_\infty$, we show that
\begin{equation}\label{e:mini}
\mathfrak c(m)\leq \frac{\pi m^2}{\Vert\beta\Vert_\infty}+o(m^2),\ \ \text{as }m\to0
\end{equation}
in Lemma \ref{l:upperc}. However, since the action $S_e$ is multivalued on $\Lambda_0 Q$, in order to get meaningful information on $q_m$ from their action, we first need to find a way to restrict the minimax construction to the set of short loops, where the action is single-valued.

\subsection{Second New Ingredient: The Restricted Action Functional}
To work on the set of short loops, we first eliminate the period variable $T$ restricting to the set of loops $q$ such that 
\begin{equation}\label{e:tme}
T=\tfrac1m\mathcal E(x),
\end{equation}
where $x$ is the 1-periodic reparametrization of $q$ and $\mathcal E$ is the $L^2$-norm of the velocity vector of $x$. Linking $T$ to $\mathcal E(x)$ is convenient since there is a constant $d_g>0$ such that
\begin{equation}
\mathcal E(x)<d_g\quad \Longrightarrow\quad q \text{ is a short loop}.  
\end{equation}
We then work with the restricted action
\begin{equation}\label{e:smse}
\mathcal S_m(x):=S_e(q)=m\mathcal E(x)-\int_{\mathbb D}\hat x^*\beta,
\end{equation}
where $\hat x\colon\mathbb D\to Q$ is a capping disk for $x$. Since \eqref{e:tme} is equivalent to $\tfrac{\partial S_e}{\partial T}=0$ and $S_e$ is strictly convex in the variable $T$, working with $\mathcal S_m$ is the same as working with $S_e$, see Lemma \ref{l:crit}. Moreover, since the path $r\mapsto x_{r,m}$ corresponding to $r\mapsto q_{r,m}$ in \eqref{e:pathqr} satisfies
\begin{equation}\label{e:om}
\mathcal E(x_{r,m})=O(m),
\end{equation}
the condition $\mathcal E(x)<d_g$ will not affect the minimax value $\mathfrak c(m)$, when $m$ is small enough.

On the other hand, \eqref{e:smse} now yields, for $m'>m$,
\begin{equation}\label{e:Lip2}
\mathcal E(x)=\frac{\mathcal S_{m'}(x)-\mathcal S_{m}(x)}{m'-m}.
\end{equation}
Therefore, if we apply Struwe's monotonicity argument, we get an upper bound on $\mathcal E(x)$ along a Palais--Smale sequence that depends on the local Lipschitz constant of $\mathfrak c(m)$ at $m$. However, we do not have an estimate for this constant that forces $\mathcal E(x)<d_g$ for all $m$ small enough. Indeed, the constant is obtained from the abstract fact that every monotone function is locally Lipschitz almost everywhere. Therefore, Struwe's idea cannot be used for our purposes.

\subsection{Third New Ingredient: Two Lyapunov Functions}
Our approach is instead based on a two-Lyapunov-function argument. The possible relevance of this argument for the existence of periodic orbits of autonomous Tonelli Lagrangian systems was observed by Abbondandolo in \cite[Section 8.3]{Abbondandolo1}, who had previously used it with Majer in the study of geodesics in Lorentzian geometry \cite{AbbondandoloMajer2008}.

The idea is to deform the paths in the minimax family using a vector field $W_m$ that is a perturbation of the negative gradient of $\mathcal S_m$. We now sketch the construction of $W_m$ after two preliminary definitions. The details are contained in Lemma \ref{l:wm} and Lemma \ref{l:phi}.

First, by the positivity of the magnetic Ricci curvature, there is $T_0>0$ such that for all $\mu>0$ small enough the critical points $x$ of $\mathcal S_\mu$ with period $T=\tfrac{1}{\mu}\mathcal E(x)$ satisfy
\begin{equation}\label{e:lowerT}
T\geq T_0\qquad \Longrightarrow \qquad \mathrm{ind}_\mu(x)\geq 3,
\end{equation}
see \eqref{e:periodindex} and Lemma \ref{l:index}. In principle, we could have taken any lower bound in \eqref{e:lowerT} instead of $3$. However, as will be clear in the next subsection, $3$ is the smallest lower bound that makes our argument work.

Second, we fix $m'>m$ and let $\mathcal R_0$ and $\mathcal R_1$ be the regions where $\mathcal E(x)> T_0m'$ and $\mathcal E(x)>T_1m'$, respectively, where $T_1>T_0$ is arbitrary. By shrinking $m$ (and $m'$), we can assume that $T_1m'<d_g$. 

We are ready to define the vector field $W_m$. On the complement of $\mathcal R_0$, the vector field $W_m$ coincides with the negative gradient of $\mathcal S_m$, while on $\mathcal R_0$ we interpolate between the negative gradients of $\mathcal S_m$ and $\mathcal S_{m'}$. We construct this interpolation in such a way that
\begin{equation}\label{e:dsm0}
\mathrm{d}\mathcal S_{m'}[W_m]\leq0\qquad \text{on }\mathcal R_1.
\end{equation}
In other words, $\mathcal S_{m'}$ does not increase along the flow lines of $W_m$ in $\mathcal R_1$. 

Thanks to \eqref{e:om} and \eqref{e:smse}, we can now restrict the minimax geometry to the set
\begin{equation}\label{e:set}
\big\{0\leq\mathcal S_m(x)\leq\mathcal S_{m'}(x)< \tfrac13D(mm'+m^2)\big\}, 
\end{equation}
for some $C>0$. We choose $m'=am$ with $a>1$. In the paper and in the discussion below, we take $a=2$,  for simplicity. In this case, \eqref{e:Lip2} implies that on the set given in \eqref{e:set} we have
\begin{equation}
\mathcal E(x)< D m,
\end{equation}
see Lemma \ref{l:inclusion}. By enlarging $C$ and shrinking $m$, we ensure in Lemma \ref{l:D} that
\begin{equation}
\mathcal E(x)\leq T_1m'\ \ \Longrightarrow\ \ \mathcal S_{m'}(x)\leq \tfrac12Dm^2,\qquad Dm<d_g. 
\end{equation}
Under these hypotheses, the flow of $W_m$ is forward-complete on the set given in \eqref{e:set}, a necessary condition for running the minimax scheme in that subset of short loops.

\subsection{Fourth New Ingredient: Slice Deformations}
The two-Lyapunov-function argument shows that we can perform the minimax argument in the set of short loops by replacing the negative gradient of $\mathcal S_m$ by the perturbed vector field $W_m$. The price to pay for this is the inequality \eqref{e:dsm0}, which implies that the forward-flow of $W_m$ might produce a path in the minimax class that also hangs on critical points of $\mathcal S_\mu$, for some $\mu\in[m,m']$, lying at the level $\mathfrak c(m)$ in the region $\mathcal R_0$. To overcome this difficulty, we observe that the restriction of $\mathcal S_\mu$ to a hypersurface $\{\mathcal E(x)=e_*\}$ in $\mathcal R_0$ has index at least $3-1=2$. Therefore, by transversality \cite{Abraham0}, a generic path in the minimax class, whose dimension is $1<2$, can be further deformed to decrease the action $\mathcal S_\mu$ while keeping $\mathcal E$ fixed, see Lemma \ref{l:local} and Lemma \ref{l:sigma}. By \eqref{e:smse}, this means that $\mathcal S_m$ and $\mathcal S_{m'}$ also decrease. By combining these slice deformations with the flow of $W_m$ in Lemma \ref{l:tildesigma}, we obtain paths in the minimax class that can hang only at critical points of $\mathcal S_m$ at the level $\mathfrak c(m)$, see Theorem \ref{t:D}.

\subsection{Fifth New Ingredient: The Blow-up Analysis}
In summary, as output of the minimax construction performed above, for every $e=\tfrac12m^2$ small enough, we get periodic magnetic geodesics $q_m$ with speed $m$, and satisfying
\begin{equation}\label{e:sl}
S_e(q_m)\leq \frac{\pi m^2}{\Vert \beta\Vert_\infty}+o(m^2),\qquad \ell(q_m)=m T(q_m)=O(m).
\end{equation}
Since $Q$ is compact and the lengths of the orbits $q_m$ tend to zero, after passing to a subsequence, the orbits $q_m$ converge to a point $q_*\in Q$ as $m\to0$. Moreover, since their periods are bounded from above, and can be shown to be bounded away from zero, the rescaled tangent vectors $\tfrac1m\dot q_m$ converge to a periodic solution $v$ of the linearized system on $T_{q_*}Q$ given by
\begin{equation}\label{e:linear0}
\dot v=B_{q_*}v.
\end{equation}
The periods and actions of the periodic solutions of \eqref{e:linear0} can easily be computed, see Section \ref{s:linear}. By the analysis in Section \ref{s:local}, they imply that
\begin{equation}
T(q_m)\geq\frac{2\pi}{|\beta_{q_*}|}+o(1),\qquad S_e(q_m)= \tfrac12m^2 T(q_m)+o(m^2).
\end{equation}
Combined with the action estimate in \eqref{e:sl}, they yield
\[
|\beta_{q_*}|\geq \Vert\beta\Vert_\infty,\qquad \ell(q_m)=mT(q_m)\leq \frac{2\pi m}{\Vert\beta\Vert_\infty}+o(m).
\]
This completes the proof of Theorem \ref{t:C}.

\section*{Structure of the Paper and Acknowledgments}
\subsection*{Structure of the paper}
The remainder of this paper is organized as follows.

In Section~\ref{s:from3to4}, we show how to reduce Theorem~\ref{t:A} from three to four dimensions. Then, we use Thom's Jet Transversality Theorem in dimension four to derive Theorem~\ref{t:A} from Theorem~\ref{t:C}. 

In Section~\ref{s:linear}, we give a complete description of the linearized magnetic dynamics on tangent spaces and estimate the period and action of their periodic orbits. These properties will be used in Section~\ref{s:local} to localize the orbits obtained with the minimax scheme.

In Section~\ref{s:Contractible Magnetic Geodesics as Critical Points}, we introduce the variational framework. More precisely, on the space of short loops, we define a single-valued Lagrangian action functional whose critical points are contractible periodic magnetic geodesics with prescribed energy. We also construct local slice deformations around critical points of large Morse index, keeping the energy fixed and decreasing the action.

In Section~\ref{s:The Minimax Geometry}, we formulate Theorem~\ref{t:D}, which will be instrumental to prove Theorem~\ref{t:C}, about the existence of critical points via a minimax scheme at low energy, whose admissible classes consist of paths joining loops near the constants to loops with zero action.

In Section~\ref{s:Deforming the Paths in the Minimax Family}, using a two-Lyapunov-function argument, we construct a suitable pseudogradient to deform the minimax classes. Combined with the local deformations near critical points of large Morse index introduced in Section~\ref{s:Contractible Magnetic Geodesics as Critical Points}, this allows us to prove Theorem~\ref{t:D}.

In Section~\ref{s:local}, using the estimates established in Section~\ref{s:linear}, we analyze the limit of the period and the action of the contractible periodic magnetic geodesic given by Theorem \ref{t:D}, as the energy goes to zero. These estimates enables us to localize the magnetic geodesics and bound their length, thus completing the proof of Theorem~\ref{t:C}.

Finally, in Section~\ref{A Non-compact Example without Periodic Magnetic Geodesics at Low Speed}, we construct a magnetic system whose global maximum set is noncompact and which admits no periodic magnetic geodesics at sufficiently small energies. This shows that the compactness assumption on the local maximum set in Theorem~\ref{t:A} cannot, in general, be omitted.

\subsection*{Acknowledgements}The authors are indebted to SISSA and to Paolo Caldiroli at Università di Torino for providing the financial means and the stimulating atmosphere for the research stay in May and June 2026 of the first- and third-named author in Italy, where part of this work has been completed. The authors thank Luca Asselle for inspiring discussions. A large language model was used for preliminary literature search assistance.

\section{From Theorem \ref{t:C} to Theorem \ref{t:A}}\label{s:from3to4}
\subsection{From Dimension Three to Dimension Four} We already know that Theorem \ref{t:A} holds when $Q$ is a surface by \cite[Lemma 5.2]{AB2022}. We now want to show that knowing Theorem \ref{t:A} for four-dimensional manifolds implies it for three-dimensional manifolds using an idea that goes back to \cite{Macarini2004}. In the next subsection, we will show how Theorem \ref{t:C} implies Theorem \ref{t:A} for manifolds of dimension at least four.

 Thus, let $Q$ be a manifold of dimension three and consider the stabilized manifold $\tilde Q:=S^1\times Q$. Denote by $\theta$ the angular coordinate on $S^1$ and by $\pi\colon \tilde Q\to Q$ the projection onto the second factor. We define
 \begin{equation}\label{e:decouple}
    \tilde g:=\mathrm{d}\theta^2+\pi^* g,\qquad \tilde\beta:=\pi^*\beta.
 \end{equation}
 Assume that $\tilde q\colon I\to\tilde Q$ is a magnetic geodesic with speed $\tilde m>0$ on $\tilde Q$. Since the magnetic geodesic equation \eqref{e:mag} associated with \eqref{e:decouple} decouples, we get for some $a\in\mathbb R$ the formula \begin{equation}\label{e:lifted}
 \tilde q(t):=(\theta_0+at,q(t)),    
 \end{equation}
 where $q:=\pi(\tilde q)$ is a magnetic geodesic on $Q$ with speed $m$ such that
 \begin{equation}
     \tilde m^2=m^2+a^2.
 \end{equation}
 If $\tilde q$ is a contractible periodic magnetic geodesic, then \eqref{e:lifted} forces $a=0$ and therefore, $\tilde m=m$. Since $|\tilde\beta_{(\theta,q)}|=|\beta_q|$ for all $(\theta,q)\in \tilde Q$, it follows that if $K$ is a strict local maximum for $\beta$ then $\tilde K:=S^1\times K$ is a strict local maximum for $\tilde \beta$ with $b_{\tilde K}=b_K$.
 
 Let $\varepsilon>0$ and let $V\subset Q$ be an open neighborhood of a non-empty, compact, strict local maximum $K$ for $\beta$. Since we are assuming that Theorem \ref{t:A} holds for four-manifolds, we can apply it to $\tilde Q$, $\varepsilon$ and $\tilde V:=S^1\times V$. We get $\tilde m_{\varepsilon,\tilde V}>0$ such that for every $m\in(0,\tilde m_{\varepsilon,\tilde V})$, there is a contractible periodic magnetic geodesic $\tilde q$ with speed $\tilde m$ and length at most $\frac{2\pi\tilde m}{b_{\tilde K}}(1+\varepsilon)$ contained in $\tilde V$. By the argument above, we conclude that $q=\pi(\tilde q)$ is a contractible magnetic geodesic contained in $V$ with speed $m:=\tilde m$ and length at most $\frac{2\pi\tilde m}{b_{K}}(1+\varepsilon)$.   
\subsection{Compactification of the Ambient Space}
In this subsection, we show how Theorem \ref{t:C} implies Theorem \ref{t:A} when $Q$ has dimension at least four.
Let $K$ be a non-empty, compact, strict local maximum of the magnetic strength of $Q$. Let $\varepsilon>0$ and $V$ be any open neighborhood of $K$. By the definition of strict local maximum, up to shrinking $V$, we can assume that
\begin{equation}\label{e:strict}
|\beta_q|=b_K,\quad\forall\,q\in K,\qquad|\beta_q|<b_K,\quad \forall\,q\in V\setminus K,
\end{equation}
for some $b_K>0$, see Definition \ref{d:strict}.

Let $W$ be an open neighborhood of $K$, whose closure is compact, contained in $V$, and such that
\begin{equation}\label{e:12bk}
\inf_{q\in W}|\beta_q|\geq \tfrac12b_K.
\end{equation}
We can also assume that the boundary of $W$ is smooth. Indeed, $K$ is the zero set of a smooth non-negative and coercive function $f\colon V\to[0,\infty)$. By Sard's theorem, we can let $W$ to be a sublevel set of $f$ with corresponding level that is regular.

Let
\begin{equation}
F\colon [-r_0,r_0]\times \partial W\to V    
\end{equation}
be a parameterization of a neighborhood of $\partial W$ disjoint from $K$ and with the property that the collar $[-r_0,0]\times\partial W$ is mapped into $W$ and $\{0\}\times\partial W$ is mapped onto $\partial W$. We write \begin{equation}\label{e:U}
W':=W\setminus F\big([-r_0,0]\times\partial W\big).
\end{equation}
Thus, $W'$ is an open neighborhood of $K$ with compact closure contained in the interior of $W$. Let $V'$ be any open neighborhood of $K$ in $Q$ such that the closure of $V'$ is contained in $W'$.

Let us denote by $\pi\colon (-r_0,r_0)\times \partial W\to\partial W$ the projection onto the second factor. Since $\pi$ is a deformation retraction, on $(-r_0,r_0)\times \partial W$ the magnetic form $\beta$ can be written as
\begin{equation}\label{e:fbeta}
F^*\beta=\pi^*\eta+\mathrm{d}\alpha,
\end{equation}
where $\eta$ is a closed two-form on $\partial W$ and $\alpha$ is a one-form on $(-r_0,r_0)\times \partial W$. Given $0<r_2<r_1<r_0$, we consider a cut-off function $\chi\colon(-r_0,0)\to\mathbb R$ such that
\begin{enumerate}[(i)]
    \item $\chi(r)=1$ for all $r\in(-r_0,-r_1)$;
    \item $\chi(r)=0$ for all $r\in(-r_2,0)$.
\end{enumerate}
Define a new closed two-form $\beta'$ on $W$ by
\begin{equation}\label{e:beta'}
\beta':=\beta\ \ \text{on }W',\qquad F^*\beta':=\pi^*\eta+\mathrm{d}(\chi\alpha)\ \ \text{on }(-r_0,0]\times\partial W.
\end{equation}
The form $\beta'$ is smooth by \eqref{e:fbeta} and the definition of $\chi$.

Let $\widehat Q$ be the double of the manifold $W$, which is obtained by
\begin{equation}
\widehat Q:= W\sqcup_\varphi W, 
\end{equation}
where
\begin{equation}
\varphi\colon(-r_2,r_2)\times \partial W\to(-r_2,r_2)\times \partial W,\qquad \varphi(r,x):=(-r,x).
\end{equation}
We have $\pi\circ\varphi=\pi$ and $F^*\beta'=\pi^*\eta$ on $(-r_2,r_2)\times\partial W$, by \eqref{e:beta'} and the definition of $\chi$. Therefore, 
\begin{equation}
\varphi^*(F^*\beta')=\varphi^*(\pi^*\eta)=(\pi\circ\varphi)^*\eta=\pi^*\eta=F^*\beta'.
\end{equation}
Therefore, $\beta'$ extends to a closed two-form $\hat\beta'$ on $\widehat Q$ that coincides with $\beta$ on $W'$ inside the first copy of $W$.

Let $\Lambda^1\widehat Q\to \widehat Q$ be the cotangent bundle of $\widehat Q$, and denote by $\mathrm{Sec}(\Lambda^1\widehat Q)$ the space of its smooth sections, that is, smooth differential one-forms. Consider the first jet bundle $J^1(\Lambda^1 \widehat Q)$ of $\Lambda^1\widehat Q\to\widehat Q$ \cite[Section 1.4]{CEM}. If we fix an affine connection $\nabla$ on $\widehat Q$, then $J^1(\Lambda^1\widehat Q)\to \Lambda^1\widehat Q$ is an affine bundle, whose fiber $J^1(\Lambda^1\widehat Q)_{(q,\eta)}$ at $(q,\eta)\in \Lambda^1\widehat Q$ is given by the set of all $\zeta\in\mathrm{Sec}(\Lambda^1\widehat Q)$ such that $\zeta_q=\eta$ quotiented by the equivalence relation $\sim_q$ given by
\begin{equation}
\zeta\sim_q\zeta'\quad\Longleftrightarrow\quad (\nabla\zeta)_q=(\nabla\zeta')_q.
\end{equation}
Denoting by $[\zeta]_q$ the equivalence class of any such $\zeta$, we have an isomorphism
\begin{equation}
J^1(\Lambda^1\widehat Q)_{(q,\eta)}\to \mathrm{Hom}(T_q\widehat Q,T^*_q\widehat Q)\cong T_q^*\widehat Q\otimes T^*_q\widehat Q,\qquad [\zeta]_q\mapsto (\nabla\zeta)_q.
\end{equation}
Moreover, if $\mathrm{Sec}(J^1(\Lambda^1\widehat Q))$ is the space of smooth sections of $J^1(\Lambda^1\widehat Q)\to\widehat Q$, there is a jet map
\begin{equation}
j^1\colon \mathrm{Sec}(\Lambda^1\widehat Q)\to \mathrm{Sec}(J^1(\Lambda^1\widehat Q)),\quad j^1(\zeta)_q:=(\zeta_q,[\zeta]_q),\quad\forall\,q\in\widehat Q,
\end{equation}
whose image is the space of holonomic sections.

Let us now suppose that $\nabla$ is a symmetric connection and consider the affine submersion
\begin{equation}
\mathcal A\colon J^1(\Lambda^1 \widehat Q)\to \Lambda^2 \widehat Q,\qquad \mathcal A(q,\eta,[\zeta]_q):=(q,\mathcal A_q(\nabla\zeta)_q),
\end{equation}
where $\mathcal A_q$ takes the antisymmetric part of $(\nabla\alpha)_q\in T_q^*\widehat Q\otimes T^*_q\widehat Q$, see \cite[Section 4.7.A]{CEM} Since $\nabla$ is symmetric, applying $\mathcal A$ on holonomic sections, we get the exterior differential
\begin{equation}\label{e:ext}
\mathcal A\big(j^1(\zeta)\big)_q=\mathrm{d}_q\zeta,\qquad\forall\,\zeta\in \mathrm{Sec}(\Lambda^1\widehat Q),\ \ \forall\,q\in\widehat Q.
\end{equation}

Let $\Gamma_{\hat\beta'} \subset \Lambda^2 \widehat Q$ be the graph of $\hat \beta'\in\mathrm{Sec}(\Lambda^2\widehat Q)$ and consider
\begin{equation}\label{e:A}
\mathcal X:=\mathcal A^{-1}(\Gamma_{\hat\beta'})\subset J^1(\Lambda^1\widehat Q). 
\end{equation}
Since $\mathcal A$ is a submersion, the fiberwise codimension of $\mathcal X$ is the same as the fiberwise codimension of $\Gamma_{\hat\beta'}$, which is $\binom{n}{2}$, where $n:=\dim\widehat Q$. Since $n=\dim \widehat Q=\dim Q\geq 4$, we have $n<\binom{n}{2}$ and therefore, for all $\zeta\in \mathrm{Sec}(\Lambda^1\widehat Q)$, we get 
\begin{equation}
j^1(\zeta)\colon \widehat Q\to J^1(\Lambda^1\widehat Q) \text{ is transverse to $\mathcal X$}\quad\Longleftrightarrow\quad j^1(\zeta)\colon \widehat Q\to J^1(\Lambda^1\widehat Q) \text{ does not intersect $\mathcal X$}.   
\end{equation}
On the other hand, given the open set $W'$ of $\widehat Q$ defined in \eqref{e:U}, by \eqref{e:12bk} we have
\begin{equation}
\hat\beta'_q=\beta_q\neq0_q,\qquad \forall\,q\in W'.
\end{equation}
Thus, $0\in \mathrm{Sec}(\Lambda^1\widehat Q)$ is transverse to $\mathcal X$ on $W'$. Applying Thom Jet Transversality \cite[Chapter 3, Theorem 2.4]{Hirsch1976} relative to $W'$ (see also \cite[Section 4.7.B]{CEM}), we get a one-form $\zeta$ on $\widehat Q$ such that
\begin{enumerate}[(i)]
    \item $\mathrm{d}_q\zeta=\mathcal A\big(j^1(\zeta)\big)_q\neq \hat\beta'_q$ for all $q\in \widehat Q$,
    \item $\zeta_q=0$ for all $q\in W'$,   
\end{enumerate}
where we used equation \eqref{e:ext} and the definition \eqref{e:A} of $\mathcal X$. We define the two-form on $\widehat Q$
\begin{equation}
\hat\beta:=\hat\beta'-\mathrm{d}\zeta,    
\end{equation}
which, by the above discussion, has the properties
	\begin{enumerate}[(i)]
		\item $\hat\beta_q\neq0$ for all $q\in \widehat Q$;
		\item $\hat\beta_q=\beta_q$ for all $q\in W'$;
		\item $\hat\beta$ is closed.
	\end{enumerate}
    
Let now $\hat g'$ be any Riemannian metric on $\widehat Q$ such that
\begin{equation}\label{e:hatg'}
\hat g'=g\quad \text{on }W'.    
\end{equation}
Consider a positive function $f\colon\widehat Q\to(0,\infty)$, which we specify below, and define
\begin{equation}
\hat g:=f\hat g'.
\end{equation}
Then,
\begin{equation}
|\hat\beta_q|_{\hat g}=\frac{1}{f(q)}|\hat\beta_q|_{\hat g'},\qquad \forall\,q\in \widehat Q.
\end{equation}
Since $V'$ has compact closure in $W'$, we can choose $f$ so that 
\begin{enumerate}[(i)]
    \item $f(q)=1$ for all $q\in V'$;
    \item $f(q)\geq 1$ for all $q\in W'$;
    \item $f(q)\geq (\tfrac12b_K)^{-1}\Vert \hat\beta\Vert_{\infty,\hat g'}$ for all $q\notin W'$.
\end{enumerate}
Keeping \eqref{e:hatg'} in mind, we deduce
\begin{enumerate}[(i)]
\item $\hat g=g$ on $V'$;
\item $|\hat\beta_q|\leq|\beta_q|$ for all $q\in W'$;
\item $|\hat\beta_q|\leq \tfrac12 b_K$ for all $q\notin W'$.
\end{enumerate}
It follows that $\Vert \hat\beta\Vert_\infty:=\Vert \hat\beta\Vert_{\infty,\hat g}=b_K$.

Given $\varepsilon>0$ and $V'$, we apply Theorem \ref{t:C} to the compact manifold $\widehat Q$ with Riemannian metric $\hat g$ and nowhere vanishing magnetic form $\hat\beta$. There exists $m_{\varepsilon,V'}>0$ such that for every $m\in(0,m_{\varepsilon,V'})$ there exists a contractible periodic magnetic geodesic $\hat q$ for $\hat g$ and $\hat \beta$ on $\hat Q$ with speed $m$ and length at most $\tfrac{2\pi m}{\Vert\hat\beta\Vert_\infty}(1+\varepsilon)=\tfrac{2\pi m}{b_K}(1+\varepsilon)$ contained in $V'$. Since $V'\subset V$, and $\hat g=g$, $\hat\beta=\beta$ on $V'$, the curve $q:=\hat q$ is a periodic magnetic geodesic of $g$ and $\beta$ with the properties asserted in Theorem \ref{t:A}. The argument is thus completed.

\section{The Lorentz Endomorphism and the Linearized Flow}\label{s:linear}
In this section, we study the action and period of magnetic geodesics of constant Riemannian metric and magnetic form on a linear space. This study will be used in Section \ref{s:local} for localizing a sequence of minimax magnetic geodesics with energy converging to zero.

\subsection{Magnetic Geodesics for Linear Systems}Let $Q$ be any manifold endowed with a Riemannian metric $g$ and a magnetic form $\beta$. Fix $q\in Q$. Recall that $B_q\colon T_qQ\to T_qQ$ is the Lorentz endomorphism, which represents $\beta_q$ with respect to $g_q$ according to equation \eqref{e:endo}. Let $SQ\subset TQ$ denote the unit tangent bundle of $g$. We also recall that
\begin{equation}
|B_q|:=\sup_{v\in S_qQ}|B_qv|,\qquad |\beta_q|:=|B_q|.
\end{equation}

Since $B_q$ is antisymmetric with respect to $g$, the spectral theorem tells us that there are integers $j,k$, real numbers $b_1\geq \ldots\geq b_j>0$, and an orthonormal basis $u_1,v_1,u_2,v_2,\ldots,u_j,v_j,w_1,\ldots,w_k$ of $T_qM$ such that
\begin{equation}
B_qu_h=b_hv_h, \ \ B_qv_h=-b_hu_h,\ \ \forall\,h=1,\ldots,j,\qquad B_q w_h=0,\ \ \forall\,h=1,\ldots k.
\end{equation} 
In particular, $u_1,v_1,\ldots,u_j,v_j$ is an orthonormal basis of $(\ker B_q)^\perp$, and $w_1,\ldots,w_k$ of $\ker B_{q}$. Moreover, we have
\begin{equation}
b_1=|B_q|=|\beta_q|.
\end{equation}
Let us fix $q_*\in Q$ and consider the solutions $v\colon\mathbb R\to T_{q_*}Q$ of the linear system
\begin{equation}\label{e:linear}
\dot v= B_{q_*}v.
\end{equation}
Using the orthonormal basis, we identify 
\begin{equation}\label{e:identify}
T_{q_*}Q\cong (\ker B_{q_*})^\perp\times\ker B_{q_*}\cong\mathbb C^j\times\mathbb R^k.
\end{equation}
We compute the solutions of \eqref{e:linear} with initial condition $v_0:=(z,w)\in \mathbb C^j\times \mathbb R^k$ as
\begin{equation}\label{e:linearv}
v(t)=(e^{ib_1t}z_1,\ldots, e^{ib_jt}z_j,w),
\end{equation}
where $z=(z_1,\ldots,z_j)$.\bigskip

Assume that $z\neq 0$. Then, there exists $h\in\{1,\ldots,j\}$ such that $z_h\neq 0$. If we also assume that $v$ is periodic with period $T$, then there exists $\nu_h\in\mathbb N$ such that  
\begin{equation}
b_hT=2\pi \nu_h.
\end{equation}
We deduce that
\begin{equation}\label{e:linearperiod}
T=\frac{2\pi\nu_h}{b_h}\geq \frac{2\pi}{|\beta_{q_*}|}.
\end{equation}
Finally, given \eqref{e:linear}, we can solve the equation \eqref{e:mag} of magnetic geodesics $q\colon\mathbb R\to T_{q_*}Q$ on the linear space $T_{q_*}Q$ with respect to the constant metric $g_{q_*}$ and the constant magnetic field $\beta_{q_*}$:
\begin{equation}\label{e:qt}
q(t)=q_0+\Big((e^{ib_1t}-1)\frac{z_1}{ib_1},\ldots, (e^{ib_jt}-1)\frac{z_j}{ib_j},tw\Big).
\end{equation}

\subsection{Action and Period Estimates}After this preparation, we can state and prove the next lemma, which will be crucial for the localization in Theorem \ref{t:C}. 

\begin{lem}\label{l:boundslinear}
Let $q$ be a non-constant, periodic magnetic geodesic of the pair $(g_{q_*},\beta_{q_*})$ on the linear space $T_{q_*}Q$. Then, the period $T$ of $q$ satisfies the lower bound
\begin{equation}\label{e:linearperiod2}
T\geq \frac{2\pi}{|\beta_{q_*}|}.
\end{equation}
If $q$ has energy $e>0$, then the action $S_e^{q_*}(q)$ of $q$ with respect to the constant metric $g_{q_*}$ and the constant magnetic field $\beta_{q_*}$ on $T_{q_*}Q$ (see \eqref{e:sm1}) satisfies the equation
\begin{equation}\label{e:linearaction}
S_e^{q_*}(q)=eT.
\end{equation}    
\end{lem}
\begin{proof}
Since $q$ is periodic and given by \eqref{e:qt}, we get $w=0$ and therefore $\tfrac12|z|^2=e>0$. Thus, $\dot q$ is periodic and satisfies \eqref{e:linear}, and, therefore, \eqref{e:linearperiod2} follows from \eqref{e:linearperiod}.

To compute the action, we recall the formula introduced in \eqref{e:sm1} that consists of a kinetic piece and a magnetic piece. The kinetic piece reads
\begin{equation}
\int_{\mathbb R/T\mathbb Z}\Big(\tfrac{1}{2}|\dot q(t)|^2_{q(t)}+e\Big)\mathrm{d}t=2eT.
\end{equation}
To compute the magnetic piece, we observe the following three facts:
\begin{enumerate}[(i)]
    \item The integral is invariant under translations of $q$. Thus, we can simply assume
    \begin{equation}
q(t)=\Big(e^{ib_1t}\frac{z_1}{ib_1},\ldots, e^{ib_jt}\frac{z_j}{ib_j},0\Big).
    \end{equation}
    \item The magnetic form $\beta_{q_*}$ is equal to $b_h \mathrm{d}x_h\wedge \mathrm{d}y_h$ in the $h$-th complex coordinate $q_h=x_h+iy_h$. Therefore, the integral is the sum of the $j$ integrals in each complex coordinate of $\mathbb C^j$.
    \item The two-form $b_h\mathrm{d}x_h\wedge \mathrm{d}y_h$ has the primitive one-form $b_h\alpha$, where $\alpha_{q_h}[u_h]=\tfrac12(iq_h)\cdot u_h$, which we can use to compute the integrals using Stokes' Theorem.
\end{enumerate}
Therefore, we obtain
\begin{equation}
\begin{aligned}
\int_{\mathbb D}\hat q^*\beta=\sum_{h=1}^j\int_{\mathbb D}\hat q_h^*(b_h\mathrm{d}x_h\wedge \mathrm{d}y_h)=\sum_{h=1}^j\int_{\mathbb R/T\mathbb Z}b_hq_h^*\alpha&=\sum_{h=1}^j\int_{\mathbb R/T\mathbb Z}b_h\frac12 \frac{ie^{ib_h t}z_h}{ib_h}\cdot e^{ib_ht}z_h \mathrm{d}t\\
&=\sum_{h=1}^j\int_{\mathbb R/T\mathbb Z}\frac12 |z_h|^2 \mathrm{d}t\\
&=eT.
\end{aligned}
\end{equation}
Putting the kinetic and the magnetic pieces together, we get
\begin{equation}
S_e^{q_*}(q)=2eT-eT=eT.
\end{equation}
\end{proof}
\section{Contractible Closed Magnetic Geodesics as Critical Points}\label{s:Contractible Magnetic Geodesics as Critical Points}
We now start to set up the variational scheme that we will use to prove Theorem \ref{t:C}. 

\subsection{The Free-period Lagrangian Action Functional}Let $Q$ be a compact manifold endowed with a Riemannian metric $g$ and a magnetic form $\beta$ such that
\begin{equation}\label{e:nowhere}
\beta_q\neq0,\qquad\forall\,q\in Q.    
\end{equation}
Let $d_g>0$ be a number that is strictly less than twice the injectivity radius of $g$. Let $q\colon\mathbb R/T\mathbb Z\to Q$ be any absolutely continuous loop with period $T>0$ such that
\begin{equation}\label{e:small}
\int_{\mathbb R/T\mathbb Z}|\dot q(t)|^2_{q(t)}\mathrm{d}t<\infty,\qquad \ell(q):=\int_{\mathbb R/T\mathbb Z}|\dot q(t)|_{q(t)}\mathrm{d}t<d_g.
\end{equation}
For any such loop and for every energy $e=\tfrac12m^2>0$, we define the action $S_e$ given in \eqref{e:sm1} by choosing any capping disk $\hat q\colon \mathbb D\to Q$ for $q$ entirely contained in the injectivity ball of $g$ centered at $q(0)$. Such a capping disk exists by \eqref{e:small}. We can decouple the dependence of $S_e$ from the period $T$ by considering the bijection
\begin{equation}
q\longleftrightarrow (x,T),\qquad  \text{where }x\colon \mathbb R/\mathbb Z\to Q,\  x(s):=q(sT),\ \forall\,s\in\mathbb R/\mathbb Z.
\end{equation}
Under this identification, we obtain
\begin{equation}
S_e(x,T)=\tfrac12\mathcal E(x)^2T^{-1}+eT-\mathcal B(x),
\end{equation}
where 
\begin{equation}
\mathcal E(x):=\Big(\int_{\mathbb R/\mathbb Z}|\dot x(s)|^2\mathrm{d}s\Big)^{\frac12}
\end{equation}
is the energy of $x$ and 
\begin{equation}
\mathcal B(x):=\int_{\mathbb D}\hat x^*\beta
\end{equation}
is the magnetic flux through the capping disk $\hat x\colon \mathbb D\to Q$ that corresponds to the disk $\hat q$.

We let 
\begin{equation}
\Omega:=\Big\{x\colon \mathbb R/\mathbb Z\to Q \text{ absolutely continuous}\ \Big|\ 0<\mathcal E(x)<d_g\Big\},
\end{equation}
which has the structure of a Hilbert manifold and we denote the induced distance function by
\begin{equation}\label{e:distance}
\mathrm{dist}\colon \Omega\times\Omega\to [0,\infty).
\end{equation}
More precisely, $\Omega$ is an open subset of $\Lambda_0Q$ of contractible loops, whose Hilbert-manifold structure is described, for example, in \cite{Lecturesonclosedgeodesics,Abbondandolo1}. In particular, the tangent space $T_x\Omega$ is the set of absolutely continuous vector fields $\xi\colon \mathbb R/\mathbb Z\to x^*TQ$ along $x$ such that the square norm
\begin{equation}\label{e:norm}
\Vert\xi\Vert^2:=\int_{\mathbb R/\mathbb Z}|\xi(s)|^2\mathrm{d}s+\int_{\mathbb R/\mathbb Z}|D_s\xi(s)|^2\mathrm{d}s
\end{equation}
is finite. The inner product on $T_x\Omega$ is the one induced by this norm.

Endowing $\Omega\times(0,\infty)$ with the product Hilbert manifold structure, we get the free-period Lagrangian action functional 
\begin{equation}
S_e\colon\Omega\times(0,\infty)\to\mathbb R, \quad(x,T)\mapsto S_e(x,T),    
\end{equation}
which is
\begin{enumerate}[(i)]
\item well-defined, since
\begin{equation}
\ell(x)\leq \mathcal E(x)<d_g;
\end{equation}
\item smoothly Frèchet differentiable, since both $\mathcal E$ and $\mathcal B$ are smooth by \cite{Abbondandolo1}. 
\end{enumerate}
More precisely, the differentials of $\mathcal E$ and $\mathcal B$ satisfy some uniform estimates contained in the next lemma. In particular, the lower bound on $\Vert \mathrm{d}\mathcal E\Vert$ quantifies the fact that there are no critical points of $\mathcal E$ on $\Omega$. Indeed, such critical points would correspond to non-constant, standard periodic geodesics of the metric $g$ with length less than twice the injectivity radius, and such geodesics do not exist.
\begin{lem}\label{l:differentials}
There exists a constant $R_g>0$ depending on $g$ such that
\begin{equation}
R_g\leq\Vert \mathrm{d}_x\mathcal E\Vert\leq 1,\qquad \Vert \mathrm{d}_x\mathcal B\Vert\leq \Vert\beta\Vert_\infty\mathcal E(x)\leq \Vert \beta\Vert_\infty d_g,\qquad \forall\,x\in\Omega.
\end{equation}
\end{lem}
\begin{proof}
The lower bound on $\Vert \mathrm{d}_x\mathcal E\Vert$ was proven in \cite[Proposition 1.4.14]{Lecturesonclosedgeodesics}. To prove the upper bound, we observe that
\[
\mathrm{d}_x\mathcal E[\xi]
=\frac{1}{\mathcal E(x)}
\int_{\mathbb R/\mathbb Z}
g_{x(s)}\left(\dot x(s),D_s\xi(s)\right)\mathrm{d}s,
\qquad
\forall\,\xi\in T_x\Omega.
\]
Applying the Cauchy--Schwarz inequality, we obtain, recalling the definition of the norm \eqref{e:norm},
\[
\big|\mathrm{d}_x\mathcal E[\xi]\big|
\leq
\left(
\int_{\mathbb R/\mathbb Z}
\vert D_s\xi(s)\vert^2\mathrm{d}s
\right)^{\frac{1}{2}}\leq \Vert\xi\Vert.
\]
It follows that
\[
\Vert \mathrm{d}_x\mathcal E\Vert\leq1,
\qquad
\forall\, x\in\Omega.
\]
A further application of the Cauchy--Schwarz inequality yields, for every $x\in\Omega$ and every $\xi\in T_x\Omega$
\begin{align*}
\vert \mathrm{d}_x\mathcal B[\xi]\vert
&=
\left|\int_{\mathbb R/\mathbb Z}
\beta_{x(s)}
\big(\xi(s),\dot x(s)\big)\mathrm{d}s\right|\leq
\Vert\beta\Vert_\infty \Vert \xi\Vert
\mathcal E(x)
\leq \Vert\beta\Vert_\infty d_g\Vert\xi\Vert.
\end{align*}
Thus,
\begin{equation}
\Vert \mathrm{d}_x\mathcal B\Vert
\leq
\Vert\beta\Vert_\infty d_g, \qquad \forall\, x \in \Omega.
\end{equation}
\end{proof}

Finally, let us discuss the critical points of $S_e$. By \cite{Contreras2006}, for every $q=(x,T)\in\Omega\times(0,\infty)$, 
\begin{equation}
q\text{ is a magnetic geodesic with energy $e$}\quad\Longleftrightarrow\quad \text{$q$ is a critical point of $S_e$.}
\end{equation}
We denote the set of such critical points by $\mathrm{Crit}\,S_e$. If $q\in \mathrm{Crit }\,S_e$, we let $\mathrm{d}^2_qS_e$ denote the second Fr\`echet differential of $S_e$ at $q$, which is a symmetric bilinear form on $T_q \big(\Omega\times(0,\infty)\big)$. The Morse index $\mathrm{ind}_e q$ of $q$ is the maximal dimension of a subspace $\mathcal W\subset T_q\big(\Omega\times(0,\infty)\big)$, where $
\mathrm{d}^2_qS_e$ is negative definite, that is,
\begin{equation}
\mathrm{d}^2_qS_e[w,w]<0,\qquad \forall\,w\in\mathcal W\setminus\{0\}.
\end{equation}
As shown in \cite[Proposition~3.1]{AbboSchwarz}, the self-adjoint operator associated with $\mathrm{d}_q^2S_e$ is a compact perturbation of a positive Fredholm operator. Consequently, we have
\begin{equation*}
    \mathrm{ind}_e (q)<\infty,\quad\forall\,q\in\mathrm{Crit}\,S_e.
\end{equation*}
\subsection{The Restricted Action Functional}
For our purposes, we will need to eliminate the period variable $T$ from the variational setting. To this aim, we notice that
\begin{equation}
\frac{\partial S_e}{\partial T}(x,T)=e-\tfrac12\mathcal E(x)^2T^{-2},\qquad \frac{\partial^2 S_e}{\partial T^2}(x,T)=\mathcal E(x)^2T^{-3}.
\end{equation}
In particular, recalling that $e=\tfrac12m^2$,
\begin{equation}\label{e:graph}
\frac{\partial S_e}{\partial T}(x,T)=0,\qquad\Longleftrightarrow\qquad T=\frac{\mathcal E(x)}{m}.
\end{equation}
Thus,
\begin{equation}
\frac{\partial S_e}{\partial T}\big(x,\tfrac{1}{m}\mathcal E(x)\big)=0,\qquad \frac{\partial^2 S_e}{\partial T^2}\big(x,\tfrac{1}{m}\mathcal E(x)\big)>0.
\end{equation}
This means that the critical points of $S_e$ lie on the graph $\{T=\tfrac1m\mathcal E(x)\}\subset \Omega\times(0,\infty)$ and that restricting the functional to the graph we get the same critical points with the same Morse index, since $S_e$ is strictly convex in the variable $T$.

Thus, we will work with the restricted action functional
\begin{equation}
\mathcal S_m\colon \Omega\to\mathbb R,\qquad \mathcal S_m(x):=S_e\big(x,\tfrac1m\mathcal E(x)\big)=m\mathcal E(x)-\mathcal B(x).
\end{equation}
We can recover the period of the corresponding orbit $q$ from the function
\begin{equation}\label{e:Tm}
T_m\colon \Omega\to\mathbb R,\qquad T_m(x)=\frac{1}{m}\mathcal E(x).
\end{equation}
By the chain rule, the restricted functional $\mathcal S_m$ is smoothly Frèchet-differentiable since the same is true for $S_e$. In particular, by \eqref{e:graph} and \eqref{e:Tm} we get
\begin{equation}\label{e:eqdiff}
\mathrm{d}_x\mathcal S_m=\frac{\partial S_e}{\partial x}(x,T_m(x)),\qquad\forall\,x\in\Omega.
\end{equation}
Thus, the discussion above yields the following result about critical points of $\mathcal S_m$.
\begin{lem}\label{l:crit}
Let $m>0$ and $T>0$. Let $x\in\Omega$ and $q:=(x,T)\in\Omega\times(0,\infty)$. The following three statements are equivalent
\begin{enumerate}[(i)]
    \item $x$ is a critical point of $
\mathcal S_m$ with $T_m(x)=T$;
\item $q$ is a critical point of $S_e$ with $e=\tfrac12m^2$;
\item $q$ is a periodic magnetic geodesic with speed $m$ and period $T$.
\end{enumerate}
If any of the above statements hold, then $\mathrm{ind}_m(x)=\mathrm{ind}_e(q)$.\hfill\qed
\end{lem}

We will be interested in the critical points of $\mathcal S_m$ when $m$ is small. In this case, we can use Lemma \ref{l:differentials} to give a lower bound on their period.
\begin{lem}\label{l:lowerboundperiod}
For all $m>0$ and $x\in \mathrm{Crit}\,\mathcal S_m$, we have the lower bound
\begin{equation*}
    \frac{R_g}{\Vert\beta\Vert_\infty}\leq T_m(x),
\end{equation*}
where $R_g$ is the constant given in Lemma \ref{l:differentials}.
\end{lem}
\begin{proof}
If $x \in \mathrm{Crit}\,\mathcal S_m$, then
\begin{equation*}
     \mathrm{d}_x \mathcal{E} = \frac{1}{m}\mathrm{d}_x \mathcal{B}
\end{equation*}
Therefore, combining the lower bound on $\Vert \mathrm{d}_x\mathcal E\Vert$ with the upper bound on $\Vert \mathrm{d}_x\mathcal B\Vert$ in Lemma \ref{l:differentials}, we obtain
\begin{equation*}
R_g\leq  \Vert \mathrm{d}_x \mathcal E\Vert = \frac{1}{m} \Vert \mathrm{d}_x \mathcal{B} \Vert \leq \frac{1}{m}\Vert \beta \Vert_\infty \mathcal{E}(x) = \Vert \beta \Vert_\infty T_m(x).
 \end{equation*}
\end{proof}
When $\mu>0$ is small, the positivity of the magnetic Ricci curvature at speed $\mu$ introduced in \cite{Assenza2024} implies that a lower bound on the period gives a corresponding lower bound on the Morse index $\mathrm{ind}_\mu(x)$ of the critical points of $S_\mu$. For later purposes, we state this implication when the lower bound on the index is equal to $3$.
\begin{lem}\label{l:index}
There exist $m_0>0$ and $T_0>0$ such that 
\begin{equation}
\forall\,\mu\in(0,2m_0),\ \ \forall\,x\in\mathrm{Crit\,}S_\mu,\qquad \mathcal E(x)\geq T_0\mu\quad \Longrightarrow\quad\mathrm{ind}_\mu(x)\geq 3.
\end{equation}
\end{lem}
\begin{proof}
Let $\mu>0$. By Lemma \ref{l:crit}, there is a bijection between critical points $x\in\mathrm{Crit}\,\mathcal S_\mu$ and critical points $(x,T)\in \mathrm{Crit}\,S_e$, where $e=\tfrac12\mu^2$. This bijection preserves the Morse index, that is, $\mathrm{ind}_\mu(x)=\mathrm{ind}_e(x,T)$.  By \eqref{e:nowhere} and \cite[Proposition 7]{Assenza2024} and the compactness of the unit sphere bundle of $Q$, there exists $m_0>0$ such that the magnetic Ricci curvature at speed $\mu\in(0,2m_0)$ has a positive lower bound $\frac{1}{r^2}$. By Lemma \cite[Lemma 14]{Assenza2024}, this implies that 
\begin{equation}
T_\mu(x)\leq \pi r(\mathrm{ind}_e(x,T)+1). 
\end{equation}
Thus, to guarantee $\mathrm{ind}_\mu(x)\geq 3$ we require
\begin{equation}
4\leq \frac{T_\mu(x)}{\pi r}.
\end{equation}
Therefore, it is enough to put $T_0:=4\pi r$ to get $4\mu\pi r\leq \mathcal E(x)$. 
\end{proof}

We will construct critical points of $\mathcal S_m$ using Palais--Smale sequences. In the next result, we characterize the converging Palais--Smale sequences of $\mathcal S_m$ using the known characterization of the converging Palais--Smale sequences of $S_e$.
\begin{lem}\label{l:PS}
Let $m>0$ and let $(x_k)\subset \Omega$ be a sequence such that 
\begin{equation}
\Vert \mathrm{d}_{x_k}\mathcal S_m\Vert\to 0.
\end{equation}
If there exist $0<a<b<d_g$ with
\begin{equation}\label{e:bounde}
a\leq \mathcal E(x_k)\leq b,\qquad \forall\,k\in\mathbb N,
\end{equation}
then, up to taking a subsequence, $x_k\to x_*\in\mathrm{Crit}\,\mathcal S_m$. 
\end{lem}
\begin{proof}
Consider the corresponding sequence $q_k=(x_k,T_k)$, where $T_k=\tfrac1m\mathcal E(x_k)$. By \eqref{e:graph}, \eqref{e:eqdiff} and \eqref{e:bounde}, we get
\begin{equation}
\Vert \mathrm{d}_{q_k}S_e\Vert=\Vert \mathrm{d}_{x_k}\mathcal S_m\Vert\to 0,\qquad ma\leq T_k\leq mb.
\end{equation}
By \cite{AsselleBenedetti2016}, up to a subsequence, $q_k\to q_*\in\mathrm{Crit}\,S_m$. Writing $q_*=(x_*,T_*)$, we also have $x_k\to x_*$.
\end{proof}

\subsection{Local Slice Deformations at Critical Points of High Index}
We will leverage the lower bounds on the index of critical points $x_*$ of $\mathcal S_\mu$ in terms of the period by constructing local deformations that keep the energy $\mathcal E$ fixed and decrease the action $S_\mu$ for one, and hence for every, value of $\mu>0$. These deformations will take place in the complement of submanifolds of codimension $\mathrm{ind}_\mu(x_*)-1$ following an idea contained in \cite[Lemma~4]{Assenza2024}, see also \cite{BahriTaimanov}. Hence, if the critical points have high index, then these submanifolds have high codimension, a property that will be crucial for us. For our purposes, we write down the result for $\mathrm{ind}_\mu(x_*)=3$.
\begin{lem}\label{l:local}
Let $\mu>0$ be a real number and let $x_*\in\mathrm{Crit}\,\mathcal S_\mu$ be such that $\mathrm{ind}_\mu(x_*)\geq 3$. Then, for every $d_0>0$, there exist 
\begin{enumerate}[(a)]
    \item real numbers $C_{x_*}>0$ and $\gamma_{x_*}>0$,
    \item neighborhoods $U_{x_*}\subset V_{x_*}\subset \Omega$ of $x_*$,
    \item a subset $\hat L_{x_*}
    \subset V_{x_*}$,
    \item a family of smooth embeddings $\sigma^t_{x_*}\colon\Omega\setminus \hat L_{x_*}\to\Omega$ for $t\in[0,1]$,
\end{enumerate}
with the following properties
\begin{enumerate}[(i)]
    \item all points in $V_{x_*}$ are at a distance at most $d_0$ from $x_*$;
    \item $\hat L_{x_*}$ is a submanifold of $V_{x_*}$ of codimension $2$;
        \item $\sigma^t_{x_*}(x)=x$ if $x\in\Omega\setminus V_{x_*}$;
    \item $\mathrm{dist}(\sigma^t_{x_*}(x),x)\leq C_{x_*}t$ for all $x\in \Omega\setminus \hat L_{x_*}$ and all $t\in[0,1]$;
    \item $\mathcal E(\sigma^t_{x_*}(x))=\mathcal E(x)$ for all $x\in \Omega\setminus \hat L_{x_*}$;
    \item $\mathcal S_m(\sigma^t_{x_*}(x))\leq \mathcal S_m(x)$ for all $m>0$ and all $x\in\Omega\setminus \hat L_{x_*}$;
    \item $\mathcal S_m(\sigma^t_{x_*}(x))\leq \mathcal S_m(x)-t^2\gamma_{x_*}$ for all $m>0$ and all $x\in U_{x_*}\setminus \hat L_{x_*}$.
\end{enumerate}
\end{lem}
\begin{proof}
In what follows, $\mathcal{H}$ denotes a separable Hilbert space modeling the Hilbert manifold $\Omega$. We denote by $\Vert\cdot\Vert_*$ the norm on $\mathcal{H}$ and by $B_r^{\mathcal{H}}\subset\mathcal{H}$ the ball of radius $r>0$ centered at the origin. By \cite{Abbondandolo1}, for some $r>0$ there exists a bi-Lipschitz chart centered at $x_*$, denoted by
\begin{equation}
\Psi\colon V_{x_*}\to B_r^{\mathcal H},
\end{equation}
where on $V_{x_*}\subset \Omega$ we take the distance defined in \eqref{e:distance} and on $B_r^{\mathcal H}\subset\mathcal H$, we take the flat distance induced by the norm $\Vert\cdot\Vert_*$. Up to shrinking the chart, we can assume that property (i) holds.

Let $e_*:=\mathcal E(x_*)$. Since $d_{x_*}\mathcal E\neq0$ by Lemma \ref{l:differentials}, the set $\mathcal{E}^{-1}(e_*)$ is a smooth hypersurface passing through $x_*$. Hence, by the Implicit Function Theorem, we can assume that, for some $\varepsilon>0$, the chart $\Psi$ is of the type
\[
\Psi\colon V_{x_*}\to(-\varepsilon,\varepsilon)\times B_r^{\mathcal H}\subset \mathbb R\times\mathcal H,
\]
with coordinates $(b,y)$ centered at $x$, such that
\begin{equation}\label{e:lemmachart0}
    \mathcal{E}(b,y)=e_*+b.
\end{equation}

Since $\mathrm{ind}_\mu(x)\geq 3$ by assumption and $\mathcal H$ has codimension $1$ in the chart, there exists a subspace of $\mathcal{H}$ of dimension $2=3-1$ such that the quadratic form associated with $d^2_{x_*}\mathcal{S}_m$ is negative definite when restricted to it. We identify this subspace with $\mathbb R^2$ and by slightly abusing notation we replace $\mathcal H$ with $\mathbb R^2\times\mathcal H$. In this way, for some $\delta>0$ we may rewrite the chart as
\[
\Psi\colon V_{x_*}\to B_\delta^2\times(-\varepsilon,\varepsilon)\times B_r^\mathcal H\subset\mathbb R^2\times\mathbb R\times \mathcal H,
\]
with coordinates $(a,b,y)$. In these local coordinates, by~\eqref{e:lemmachart0} for every $m>0$ the action $\mathcal{S}_{m}$ takes the form
\begin{equation}\label{e:lemmachart1}
\mathcal{S}_{m}(a,b,y)
    =m(e_*+b)-\mathcal{B}(a,b,y).
\end{equation}
Consequently,
\begin{equation*}
    \frac{\partial \mathcal{S}_{m}}{\partial a}(a,b,y)=-\frac{\partial \mathcal B}{\partial a}(a,b,y)=\frac{\partial \mathcal{S}_\mu}{\partial a}(a,b,y),\qquad
\frac{\partial^2 \mathcal{S}_{m}}{\partial a^2}(a,b,y)=-\frac{\partial^2 \mathcal B}{\partial a^2}(a,b,y)=\frac{\partial^2 \mathcal{S}_{\mu}}{\partial a^2}(a,b,y).
\end{equation*}
Therefore, since $x_*$ is a critical point of $\mathcal S_\mu$ and has coordinates $(0,0,0)$, we have
\begin{equation}
\frac{\partial\mathcal B}{\partial a}(0,0,0)=0,\qquad -\frac{\partial^2\mathcal B}{\partial a^2}(0,0,0)[u,u]\leq -\eta|\alpha|^2,\quad\forall\,\alpha\in\mathbb R^2
\end{equation}
for some $\eta>0$. Since $\mathcal B$ is smoothly Frèchet differentiable, its second derivative is continuous. Thus, up to shrinking the neighborhood and $\eta$, we can assume that
\begin{equation}
\frac{\partial^2\mathcal B}{\partial a^2}(a,b,y)[\alpha,\alpha]\leq -\eta|\alpha|^2,\qquad\forall\,(a,b,y)\in B_\delta^2\times(-\varepsilon,\varepsilon)\times B_r^{\mathcal H},\ \ \forall\,\alpha\in\mathbb R^2.
\end{equation}

Consider the map
\begin{equation}
(a,b,y)\mapsto \frac{\partial\mathcal B}{\partial a}(a,b,y)\in\mathbb R^2.    
\end{equation}
By the Implicit Function Theorem, there exists a function $(b,y)\mapsto a(b,y)$ such that 
\begin{equation}
\frac{\partial\mathcal B}{\partial a}(a(b,y),b,y)=0,\qquad  a(0,0)=0.
\end{equation}
Hence, up to the bi-Lipschitz diffeomorphism $(a,b,y)\mapsto (a-a(b,y),b,y)$, we can assume that in the chart $\Psi$ we have
\begin{equation}\label{e:morse}
\frac{\partial\mathcal B}{\partial a}(0,b,y)=0,\qquad -\frac{\partial^2\mathcal B}{\partial a^2}(a,b,y)[\alpha,\alpha]\leq -\eta|\alpha|^2.
\end{equation}

We claim that for all $(a,b,y)\in B_\delta^2\times(-\varepsilon,\varepsilon)\times B_r^\mathcal H$ with $a\neq0$ and all $\lambda\in[0,\delta-|a|)$ we have
\begin{equation}\label{e:taylor}
-\mathcal B\Big(a+\lambda\frac{a}{|a|},b,y\Big)\leq -\mathcal B(a,b,y)-\frac12\eta\lambda^2.
\end{equation}
Indeed, observe that \eqref{e:morse} implies
\begin{equation}\label{e:lemmachart3}
-\frac{\partial\mathcal B}{\partial a}(a,b,y)[a]\leq0,
\qquad
\forall\,(a,b,y)\in B_\delta^2\times(-\varepsilon,\varepsilon)\times B_r^{\mathcal H}.
\end{equation}
Therefore, by Taylor's theorem, there exists $\lambda'\in(0,\delta-|a|)$ such that
\begin{equation}
\begin{aligned}
-\mathcal B\Big(a+\lambda\frac{a}{|a|},b,y\Big)&=-\mathcal B(a,b,y)-\frac{\partial\mathcal B}{\partial a}(a,b,y)\Bigg[\lambda\frac{a}{|a|}\Bigg]-\frac12\frac{\partial^2\mathcal B}{\partial a^2}\Big(a+\lambda'\frac{a}{|a|}\Big)\Bigg[\lambda\frac{a}{|a|},\lambda\frac{a}{|a|}\Bigg]\\
&\leq -\mathcal B(a,b,y)+0-\frac12\eta\lambda^2.
\end{aligned}
\end{equation}

We define
\begin{equation}
\hat L_{x_*}:=\Psi^{-1}\big(0\times(-\varepsilon,\varepsilon)\times B^{\mathcal H}_r\big),
\end{equation}
which readily satisfies property (ii) of the present lemma. Fix now
\begin{equation}
0<\delta''<\delta'<\delta,\qquad 0<\varepsilon''<\varepsilon'<\varepsilon,\qquad 0<r''<r'<r,
\end{equation}
and take a smooth bump function $\xi\colon (0,\delta)\times(-\varepsilon,\varepsilon)\times B^{\mathcal H}_r\to[0,1]$ such that
\begin{equation}\label{e:xi}
    \begin{aligned}
\xi(\rho,b,y)&=0,\qquad \forall\,(\rho,b,y)\notin (0,\delta')\times (-\varepsilon',\varepsilon')\times B^{\mathcal H}_{r'},\\ 
\xi(\rho,b,y)&=1,\qquad \forall\,(\rho,b,y)\in (0,\delta'')\times (-\varepsilon'',\varepsilon'')\times B^{\mathcal H}_{r''}.
    \end{aligned}
\end{equation}
We let
\begin{equation}\label{e:ustar}
U_{x_*}:=\Psi^{-1}\big(B_{\delta''}^2\times (-\varepsilon'',\varepsilon'')\times B^{\mathcal H}_{r''}\big).
\end{equation}
For every $u\in[0,\delta-\delta']$, we define
\begin{equation}
\begin{aligned}
\hat\sigma^{u}_{x_*}&\colon \big(B_{\delta}^2\setminus0\big)\times (-\varepsilon,\varepsilon)\times B^{\mathcal H}_r\to B_{\delta}^2\times (-\varepsilon,\varepsilon)\times B^{\mathcal H}_r,\\
\hat\sigma^{u}_{x_*}&(a,b,y):=\Big(a+u\xi(|a|,b,y)\frac{a}{|a|},b,y\Big).    \end{aligned}
\end{equation}
In other words, $\hat\sigma^u_{x_*}$ pushes radially out by a distance at most $u$ in the first coordinate and it leaves the second and third coordinate unchanged. We notice that $\hat\sigma^u_{x_*}$ is the identity outside $B_{\delta'}^2\times (-\varepsilon',\varepsilon')\times B^{\mathcal H}_{r'}$ and that its image is contained in $B_{\delta}^2\times (-\varepsilon,\varepsilon)\times B^{\mathcal H}_r$ since $u\in[0,\delta-\delta']$. 
If we let $(\rho,\theta)$ be the polar coordinates on $B^2_\delta\setminus 0$, then we have the expression
\begin{equation}
\hat\sigma^{u}_{x_*}(\rho,\theta,b,y)=(\rho+u\xi(\rho,b,y),\theta,b,y).
\end{equation}
If $u_0\in[0,\delta-\delta']$ is such that
\begin{equation}
0<u_0\left\Vert\frac{\partial\xi}{\partial\rho}\right\Vert_\infty<1,
\end{equation}
then $\hat\sigma^{u}_{x_*}$ is an embedding for all $u\in[0,u_0]$.

Therefore, we define for every $t\in[0,1]$
\begin{equation}\label{e:sigmat}
\sigma^t_{x_*}\colon \Omega\setminus\hat L_{x_*}\to\Omega,\qquad \sigma^t_{x_*}(x):=\begin{cases}
\Psi^{-1}\circ\hat\sigma^{t u_0}_{x_*}\circ\Psi(x),&\text{if }x\in V_{x_*},\\
x,&\text{if }x\notin V_{x_*}.
\end{cases}
\end{equation}
This is a family of smooth embeddings satisfying property (iii) of the present lemma, as follows from \eqref{e:xi}. The distance with respect to the norm $\Vert\cdot\Vert_*$ between the points $\sigma^t_{x_*}(x)$ and $x$ is
\begin{equation}
\big\Vert\hat\sigma^{t u_0}_{x_*}(a,b,y)-(a,b,y)\Vert_*=tu_0\xi(|a|,b,y)\leq tu_0.
\end{equation}
Since the chart $\Psi$ is bi-Lipschitz, formula \eqref{e:sigmat} implies that there exists a constant $C_{x_*}>0$ satisfying property (iv) of the present lemma. Since $\hat\sigma^{tu_0}_{x_*}$ preserves the coordinate $b$, we deduce that $\sigma^t_{x_*}$ preserves the function $\mathcal E$, see \eqref{e:lemmachart0}. This shows property (v) of the present lemma. From this fact, it also follows that it is enough to show properties (vi) and (vii) for $\hat\sigma^{tu_0}_{x_*}$ and the function $-\mathcal B$ instead of showing them for $\sigma^t_{x_*}$ and the function $\mathcal S_m$, see \eqref{e:lemmachart1}. From \eqref{e:taylor}, we get
\begin{equation}
\begin{aligned}
-\mathcal B\Big(\hat\sigma^{tu_0}_{x_*}(a,b,y)\Big)\leq -\mathcal B(a,b,y)-\frac12\eta\big(t u_0 \xi(a,b,y)\big)^2.    
\end{aligned}
\end{equation}
This shows property (vi). Since $\xi(a,b,y)=1$ if $(a,b,y)\in \Psi(U_{x_*})$, see \eqref{e:ustar}, we also deduce property (vii) with $\gamma_{x_*}:=\tfrac12 \eta u_0^2$.
\end{proof}

\section{The Minimax Argument}\label{s:The Minimax Geometry}
The main ingredient to prove Theorem \ref{t:C} will be the following result.
\begin{mainthm}\label{t:D}
There exists a constant $D>0$ such that for every $\varepsilon>0$ there exists $m_\varepsilon>0$ with the following property: for every $m\in(0,m_\varepsilon]$ there exists a critical point $x_m$ of $\mathcal S_m$ with
\begin{equation}
\mathcal E(x)\leq Dm,\qquad \mathcal S_m(x)\leq \frac{\pi m^2}{\Vert\beta\Vert_\infty}(1+\varepsilon).
\end{equation}
\end{mainthm}
In this and the next section, we prove Theorem \ref{t:D} using a minimax argument. In Section \ref{s:local}, we show how Theorem \ref{t:D} implies Theorem \ref{t:C}.
\subsection{Action Estimates}
In this subsection, we present three estimates on the action $\mathcal S_m$. To this purpose, we recall that by the compactness of $Q$, there exists a constant $A'>0$ such that for every $q_*\in Q$ and normal chart $U\to\mathbb R^n$ with radius $d_g/2$ around $q_*$, we have for the norm induced by $g$ at different points
\begin{equation}\label{e:equivalence}
|u|_{q_1}\leq A'|u|_{q_2},\qquad \forall\,q_1,q_2\in U,\quad \forall\,u\in\mathbb R^n.
\end{equation}
\begin{lem}\label{l:upperlower}
There exists a constant $A>0$ such that for all $m>0$
\begin{equation}
m\mathcal E(x)-\tfrac12A\mathcal E(x)^2\leq \mathcal S_m(x)\leq m\mathcal E(x)+\tfrac12A\mathcal E(x)^2,\qquad\forall\,x\in\Omega.
\end{equation}
\end{lem}
\begin{proof}
By the definition of $\mathcal S_m$, it is enough to show that
\begin{equation}
|\mathcal B(x)|\leq \tfrac12A\mathcal E(x)^2,\qquad\forall\,x\in\Omega.    
\end{equation}
If $x\in\Omega$, let $q_*:=x(0)$ and consider the capping disk $\hat x\colon \mathbb D\to U$ for $x$ given by the coordinate expression in a normal chart 
\begin{equation}
\hat x(r,s)=rx(s),\qquad \forall\,(r,s)\in[0,1]\times\mathbb R/\mathbb Z.
\end{equation}
Thus,
\begin{equation}
|x(s)|_{x(s)}\leq \tfrac12\ell(x)\leq \tfrac12\mathcal E(x),\qquad \forall\,s\in\mathbb R/\mathbb Z.
\end{equation}
We obtain 
\begin{equation}
\begin{aligned}
|\mathcal B(x)|=\left|\int_0^1\int_{\mathbb R/\mathbb Z}\beta_{rx(s)}(x(s),r\dot x(s))\mathrm{d}r\mathrm{d}s\right|&\leq \Vert\beta\Vert_\infty\int_0^1\int_{\mathbb R/\mathbb Z}|x(s)|_{rx(s)}|r\dot x(s)|_{rx(s)}\mathrm{d}r\mathrm{d}s\\
&\leq\Vert\beta\Vert_\infty\int_0^1r\int_{\mathbb R/\mathbb Z}|x(s)|_{x(s)}A'|\dot x(s)|_{x(s)}\mathrm{d}r\mathrm{d}s\\
&\leq \Vert\beta\Vert_\infty\tfrac14\mathcal E(x)A'\mathcal E(x),
\end{aligned}
\end{equation}
where in the second inequality we used that normal coordinates are a radial isometry and that \eqref{e:equivalence} holds. Thus, we let $A:=\tfrac12\Vert\beta\Vert_\infty A'$.
\end{proof}
For the two-Lyapunov-function argument in Section \ref{s:Deforming the Paths in the Minimax Family}, we will need estimates for both $\mathcal S_m$ and $\mathcal S_{2m}$. We now deal with this second function.
\begin{lem}\label{l:D}
Let $m_0$ and $T_0$ be given in Lemma \ref{l:index} and fix any real number $T_1$ such that
\begin{equation}\label{e:T1}
T_1>\max\big\{T_0,\frac{2\pi}{\Vert\beta\Vert_\infty}\big\}.
\end{equation}
There exist $m_1\in(0,m_0]$ and $D>0$ such that
\begin{enumerate}[(i)]
    \item for all $m\in(0,m_1]$ and all $x\in\Omega$ we have
\begin{equation}
\mathcal E(x)\leq 2T_1m\quad\Longrightarrow\quad \mathcal S_{2m}(x)\leq \tfrac12Dm^2;
\end{equation}
\item $Dm_1<d_g$.
\end{enumerate}
\end{lem}
\begin{proof}
For all $m>0$, by Lemma \ref{l:upperlower}
we have
\begin{equation}
\mathcal E(x)\leq 2T_1m\quad\Longrightarrow\quad \mathcal S_{2m}(x)\leq 2m(2T_1m)+\tfrac12A(2T_1m)^2.
\end{equation}
Therefore, we can take
\begin{equation}
D:=8T_1+4AT_1^2,\qquad m_1:=\min\Big\{m_0,\frac{d_g}{2D}\Big\}. 
\end{equation} \end{proof}

In addition to the estimates above that hold for a general $x\in\Omega$, we need a special estimate for a path of loops bifurcating from the maximum set of the magnetic strength.
\begin{lem}\label{l:special}
There exist $r_0>0$ and a continuous path $x\colon (0,r_0]\to \Omega$ such that
    \begin{equation}
        (i)\quad \mathcal E(x(r))=r,\ \ \forall\,r\in(0,r_0],\qquad (ii)\quad \mathcal B(x(r))=\frac{1}{4\pi}\Vert\beta\Vert_\infty r^2+o(r^2),\ \ \text{as }r\to0.
    \end{equation}
\end{lem}
\begin{proof}
Let $q_*\in Q$ be such that $|\beta_{q_*}|=\Vert\beta\Vert_\infty$. By Section \ref{s:linear}, we recall that there is an orthogonal splitting $T_{q_*}Q\cong \mathbb R^2\times \mathbb R^{n-2}$ such that \begin{equation}
g_{q_*}|_{\mathbb R^2\times 0}=\mathrm{d}x^2+\mathrm{d}y^2,\qquad \beta_{q_*}|_{\mathbb R^2\times 0}=\Vert\beta\Vert_\infty \mathrm{d}x\wedge \mathrm{d}y.
\end{equation} 
Let $U\to \mathbb R^2\times \mathbb R^{n-2}$ be normal coordinates centered at $q_*$ that respect the splitting. Using the normal chart $U$, we define \begin{equation}
x\colon(0,\tfrac12d_g)\to \Omega,\qquad x(d)(s):=de^{2\pi is},\qquad \forall\,d\in(0,\tfrac12d_g),\ \forall\,s\in\mathbb R/\mathbb Z.
\end{equation}
We compute for $d\to 0$
\begin{equation}\label{e:BB}
\mathcal B(x(d))=\Vert\beta\Vert_\infty\pi d^2 +o(d^2)
\end{equation}
and
\begin{equation}
\mathcal E(x(d))=2\pi d\Big(\int_{\mathbb R/\mathbb Z}|e^{2\pi i s}|^2_{x(d)(s)}\mathrm{d}s\Big)^\frac{1}{2}.
\end{equation}
The function $d\mapsto\mathcal E(x(d))$ extends to $0$ for $d=0$, it is differentiable, and $\mathcal E(x(0))'=2\pi>0$. Therefore, it admits an inverse 
\begin{equation}\label{e:dr}
r\mapsto d(r)=\frac{1}{2\pi}r+o(r).
\end{equation}
Reparameterizing $x(r):=x(d(r))$, we get 
\begin{equation}
\mathcal E(x(r))=\mathcal E(x(d(r))=d^{-1}(d(r))=r.
\end{equation}
This gives us property (i). Property (ii) follows by substituting \eqref{e:dr} into \eqref{e:BB}.
\end{proof}

\subsection{The Minimax Class}
In this subsection we will define the minimax class of paths, the corresponding minimax value, and give bounds for this minimax value that will be used in Section \ref{s:local}.
\begin{dfn}\label{d:T}
Fix any triple of real numbers $0<T_*<T_*'<T_*''$ such that
\begin{equation}
T_*m_1\leq\frac{4\pi}{\Vert\beta\Vert_\infty},\quad \delta:=\tau''-\tau'>0,\quad \text{where}\quad\tau':=T_*'+\tfrac12A(T_*')^2,\quad \tau'':=T_*''-\tfrac12A(T_*'')^2.
\end{equation}
For every $m\in(0,m_1]$, define the set
\begin{equation}
\Omega_m:=\left\{\ x\in\Omega\ \Big| \ 0\leq \mathcal S_m(x)\leq \mathcal S_{2m}(x)< Dm^2,\ T_*m\leq \mathcal E(x)\ \right\}.
\end{equation}
\end{dfn}
Having a control from below on the function $\mathcal S_m$ and from above on the function $\mathcal S_{2m}$ in $\Omega_m$, we can bound the energy $\mathcal E$ on the closure of $\Omega_m$.
\begin{lem}\label{l:inclusion}
For every $m\in(0,m_1]$, denoting by $\bar\Omega_m$ the closure of $\Omega_m$ in $\Omega$, we have the inclusion
\[
\bar\Omega_m\subset\Big\{\mathcal E(x)\leq Dm\Big\}.
\]
Moreover, 
\begin{equation}\label{e:partial}
\Omega_m\cap\partial\Omega_m\subset\Big\{T_*m=\mathcal E(x)\Big\}\cup\Big\{\mathcal S_m(x)=0\Big\}.
\end{equation}
\end{lem}
\begin{proof}
If $x\in \bar\Omega_m$, then $\mathcal S_{2m}(x)\leq Dm^2$ and $\mathcal S_m(x)\geq0$. Thus, we compute using \eqref{e:Lip2}
\begin{equation}
\mathcal E(x)=\frac{\mathcal S_{2m}(x)-\mathcal S_m(x)}{2m-m}\leq\frac{Dm^2-0}{m}=Dm.
\end{equation}
\end{proof}
\begin{dfn}
For every $m\in(0,m_1]$, define the minimax class
\begin{equation}
X_m:=\Big\{\  x\in C^0([0,1],\Omega_m)\ \Big|\ \mathcal E(x(0))=T_*m,\ \mathcal S_m(x(1))=0\ \Big\}
\end{equation}
and the minimax critical value
\begin{equation}\label{e:cm}
\mathfrak c(m):=\inf_{x\in X_m}\max_{r\in[0,1]}\mathcal S_m(x(r)).
\end{equation}
\end{dfn}
We proceed to give a lower bound and an upper bound on $\mathfrak c(m)$. They will be based on the following elementary observation.
\begin{lem}\label{r:quadratic}
For every $a>0$, the quadratic function \begin{equation}
(0,\infty)\to\mathbb R,\qquad r\mapsto r-\tfrac12ar^2,
\end{equation}
is positive for $r\in\big(0,\frac{2}{a}\big)$ and attains the maximum $\frac{1}{2a}$ at $r=\frac{1}{a}$.\hfill\qed
\end{lem}
We first establish the lower bound for $\mathfrak c(m)$.
\begin{lem}\label{l:lowerc}
For every $m\in(0,m_1]$, recalling Definition \ref{d:T}, we have
\begin{enumerate}[(i)]
    \item the implication
    \begin{equation}
\mathcal E(x)=T_*m\quad\Longrightarrow\quad \mathcal S_m(x)>0,\quad\forall\,x\in\Omega;
\end{equation}
    \item the lower bound
    \begin{equation}
\mathfrak c(m)\geq \tau''m^2>\delta m^2;
\end{equation}
    \item the inclusion
    \begin{equation}
\Big\{\mathcal S_m(x)\geq \mathfrak c(m)-\delta m^2\Big\}\subset\Big\{\mathcal E(x)\geq T_*'m\Big\}.
\end{equation}
\end{enumerate}
\end{lem}
\begin{proof}
By Lemma \ref{l:upperlower}, if $x\in\Omega$ and $r:=\mathcal E(x)$, then
\begin{equation}\label{e:double}
mr-\tfrac12Ar^2\leq \mathcal S_m(x)\leq mr+\tfrac12Ar^2.
\end{equation}
Since $\tau''>0$, applying Lemma \ref{r:quadratic} with $a:=A$, we see that if $r<T_*''m$, then $\mathcal S_m(x)>0$. Since $T_*<T_*''$, then (i) follows.

Moreover, if $x\in X_m$, then since $\mathcal S_m(x(1))=0$, we deduce that the path $x$ intersects the set $\{\mathcal E=T_*''m\}$. Hence, by the definition of $\tau''$ and of the minimax value $\mathfrak c(m)$, we conclude that $\mathfrak c(m)\geq \tau''m^2$ as required by (ii).

Finally, we establish the contrapositive inclusion in property (iii). Indeed, if $\mathcal E(x)<T_*'m$, then by \eqref{e:double}, we get
\begin{equation}
\mathcal S_m(x)< \tau'm^2= \tau''m^2-\delta m^2\leq \mathfrak c(m)-\delta m^2.
\end{equation}
\end{proof}
We now establish the upper bound for $\mathfrak c(m)$.
\begin{lem}\label{l:upperc}
For every $\varepsilon>0$, there exists $m_\varepsilon\in(0,m_1]$ such that 
\begin{equation}
\forall\,m\in(0,m_\varepsilon],\ \ \exists\,x_m\in X_m,\qquad \max_{r\in[0,1]}\mathcal S_m(x_m(r))\leq \frac{\pi m^2}{\Vert\beta\Vert_\infty}(1+\varepsilon).
    \end{equation}
    Therefore, we have the upper bound
\begin{equation}
\mathfrak c(m)\leq \frac{\pi m^2}{\Vert\beta\Vert_\infty}(1+\varepsilon).
\end{equation}
\end{lem}
\begin{proof}
Fix any $\varepsilon_0$. Without loss of generality, it is enough to show the lemma assuming $\varepsilon\leq \varepsilon_0$. Let $x\colon(0,r_0]\to\Omega$ be the continuous path provided by Lemma \ref{l:special}. Then, for all $m\in(0,m_1]$
\begin{equation}
\mathcal S_m(x(r))=mr-\frac{1}{4\pi}\Vert\beta\Vert_\infty r^2+o(r^2)
\end{equation}
as $r$ tends to zero. Therefore, there exists $r_\varepsilon\in(0,r_0]$ such that
\begin{equation}\label{e:upperepsilon}
\mathcal S_m(x(r))\leq mr-\frac{1}{4\pi}\frac{\Vert\beta\Vert_\infty}{(1+\varepsilon)} r^2,\qquad \forall\,r\in(0,r_\varepsilon].
\end{equation}
Let $m_\varepsilon\in(0,m_1]$ be such that
\begin{equation}
\frac{4\pi(1+\varepsilon_0)}{\Vert\beta\Vert_\infty}m_\varepsilon<r_\varepsilon.
\end{equation} 
Let $m\in(0,m_\varepsilon]$ and define
\begin{equation}
r_m:=\frac{4\pi (1+\varepsilon_0)}{\Vert\beta\Vert_\infty}m<r_\varepsilon.
\end{equation}
From Definition \ref{d:T}, we get
\begin{equation}\label{e:tstar}
T_*m<r_m.
\end{equation}
Moreover, using $\varepsilon\leq \varepsilon_0$ in \eqref{e:upperepsilon}, we get
\begin{equation}\label{e:epsilon0}
\mathcal S_m(x(r_m))\leq m r_m-\frac{1}{4\pi}\frac{\Vert\beta\Vert_\infty}{1+\varepsilon_0} r_m^2=0,
\end{equation}
where the last equality follows from Lemma \ref{r:quadratic}.

By \eqref{e:tstar} and \eqref{e:epsilon0} and property (i) in Lemma \ref{l:lowerc}, there exists an interval $[T_*m,r_m']\subset(0,r_m]$ such that
\begin{equation}
\mathcal S_m(x(r))\geq 0,\ \ \forall\,r\in [T_*m,r_m'],\qquad \mathcal S_m(x(r_m'))=0.
\end{equation}
By the definition of $T_1$ in Lemma \ref{l:D}, we can choose $\varepsilon_0$ such that
\begin{equation}
\frac{4\pi (1+\varepsilon_0)}{\Vert\beta\Vert_\infty}\leq 2T_1.
\end{equation}
By the defition of $r_m$, we get
\begin{equation}
\mathcal E(x(r))=r\leq r_m\leq 2T_1m,\qquad \forall\,r\in [T_*m,r_m'].
\end{equation}
By Lemma \ref{l:D}.(i), it follows that $x|_{[T_*m,r_m']}\in X_m$, up to reparameterizing $[T_*m,r_m']$ to $[0,1]$. Finally, applying Lemma \ref{r:quadratic} with $a:=\frac{\Vert\beta\Vert_\infty}{2\pi (1+\varepsilon)m}$ and using \eqref{e:upperepsilon}, we get 
\begin{equation}
\max_{r\in[0,1]}\mathcal S_m(x_m(r))\leq m\frac12\frac{2\pi (1+\varepsilon)m}{\Vert\beta\Vert_\infty}.
\end{equation}\end{proof}
\section{Deforming the Paths in the Minimax Family}\label{s:Deforming the Paths in the Minimax Family}
\subsection{Good and Bad Critical Points}
Now that we have defined the minimax class $X_m$, we need to construct a forward-complete vector field $W_m$ that will be used to deform the paths $x\in X_m$ and detect the critical points needed in Theorem \ref{t:D}. To this purpose, let us distinguish the sets of good and bad critical points.
\begin{dfn}
For every $m\in(0,m_1]$, let
\begin{equation}
\begin{aligned}
K_m&:=\Big\{\  x\in \Omega\cap \mathrm{Crit}\,\mathcal S_m\ \Big|\ T_*m\leq \mathcal E(x),\ \mathcal S_m(x)=\mathfrak c(m),\ \mathcal S_{2m}(x)\leq Dm^2 \ \Big\};\\
L_m&:=\bigcup_{\mu\in[m,2m]}\Big\{ \ x\in \Omega\cap \mathrm{Crit}\,\mathcal S_\mu\  \Big|\ 2 T_0 m\leq \mathcal E(x),\ \mathcal S_m(x)=\mathfrak c(m),\ \mathcal S_{2m}(x)\leq Dm^2 \Big\}.
\end{aligned}
\end{equation}
\end{dfn}
\begin{lem}\label{l:KL}
The sets $K_m$ and $L_m$ are compact and are contained in the closure $\bar\Omega_m$.
\end{lem}
\begin{proof}
The fact that $K_m$ and $L_m$ are contained in $\bar\Omega_m$ follows from the fact that $\mathfrak c(m)$ is positive, see Lemma \ref{l:lowerc}. Moreover, by Lemma \ref{l:inclusion} they are contained in the set 
\begin{equation}\label{e:loweruppere}
\{T_*m\leq\mathcal E(x)\leq Dm\}.
\end{equation}
By Lemma \ref{l:PS} we immediately see that $K_m$ is compact. Let now $(x_k)\subset L_m$, then there exists $(\mu_k)\subset[m,2m]$ such that $x_k\in\mathrm{Crit}\,\mathcal S_{\mu_k}$. Up to taking a subsequence, we can assume that $\mu_k\to\mu_*\in[m,2m]$. Therefore,
\begin{equation}
\Vert \mathrm{d}_{x_k}\mathcal S_{\mu_*}\Vert=\Vert \mathrm{d}_{x_k}(\mu_*-\mu_k)\mathcal E+\mathrm{d}_{x_k}\mathcal S_{\mu_k}\Vert=\Vert \mathrm{d}_{x_k}(\mu_*-\mu_k)\mathcal E\Vert=|\mu_*-\mu_k|\Vert \mathrm{d}_{x_k}\mathcal E\Vert.
\end{equation}
Since $\Vert \mathrm{d}\mathcal E\Vert$ is bounded on $\Omega$, we deduce that $\Vert \mathrm{d}_{x_k}\mathcal S_{\mu_k}\Vert\to0$. By \eqref{e:loweruppere} and Lemma \ref{l:PS}, we conclude that, up to a subsequence, $x_k\to x_*\in\mathrm{Crit}\,\mathcal S_{\mu_*}$ and the point $x_*\in L_m$.
\end{proof}
Thanks to Lemma \ref{l:KL}, Lemma \ref{l:upperc} and Lemma \ref{l:inclusion}, we have the immediate implication
\begin{equation}\label{e:implication}
\forall\,m\in(0,m_1],\ \ K_m\neq\varnothing\quad\Longrightarrow\quad \text{Theorem \ref{t:D}}.
\end{equation}
Eventually, our strategy will be to argue by contradiction to show that the left-hand side holds.

As a first step, in the next subsection we find a suitable local deformation of the action near $L_m$ using the local deformations of Lemma \ref{l:local}. 
\subsection{A Slice Deformation Near $L_m$} To construct the deformation, we will make use of the following elementary result.
\begin{lem}\label{l:system}
There exists a system of open neighborhoods $\{U_m^{(k)}\}_{k\in\mathbb N}$ of $L_m$ inside $\bar\Omega_m$ such that
\begin{equation}
\forall\,k\in\mathbb N,\quad U_m^{(k+1)}\ \text{has a positive distance from}\ \bar\Omega_m\setminus U_m^{(k)}.
\end{equation}
\end{lem}
\begin{proof}
Let $U_m^{(k)}$ be the set of points at a distance strictly less that $\tfrac1k$ from $L_m$. Let us show that this is a system of open neighborhoods of $L_m$. Let $U$ be any open neighborhood of $L_m$. By Lemma \ref{l:KL}, the set $L_m$ is compact. Therefore, it has a positive distance $d$ from the disjoint closed set $\bar\Omega_m\setminus U$. Then, $U_m^{(k)}\subset U$ if $\tfrac1k<d$.
\end{proof}
We are now ready to define the deformation near $L_m$ that will keep $\mathcal E$ fixed and decrease the actions. The proof is a streamlined version of \cite[Lemma 4]{Assenza2024}.
\begin{lem}\label{l:sigma}
For every $m\in(0,m_1]$, there exist
\begin{enumerate}[(a)]
\item a real number $\hat\gamma_m>0$,
\item an open neighborhood $U_m\subset\bar\Omega_m$ of $L_m$,
\item a set $\hat L_m\subset\bar\Omega_m$,
\item a smooth map $\hat\sigma_m\colon \Omega_m\setminus \hat L_m\to \Omega_m$, 
\end{enumerate}
with the following properties
\begin{enumerate}[(i)]
    \item the set $\hat L_m$ is a finite union of embedded submanifolds of $\bar\Omega_m$ of codimension $2$;
    \item the set $\hat L_m$ is at a positive distance from $\Omega_m\cap\partial\Omega_m$;
    \item the support of the map $\hat\sigma_m$ is at a positive distance from $\Omega_m\cap\partial\Omega_m$;
       \item $\mathcal S_m(\hat\sigma_m(x))\leq \mathcal S_m(x)$ for all $x\in \Omega_m\setminus \hat L_m$;
    \item $\mathcal S_m(\hat\sigma_m(x))\leq\mathcal S_m(x)-\hat\gamma_m$ for all $x\in (\Omega_m\cap U_m)\setminus \hat L_m$.
\end{enumerate}
\end{lem}
\begin{proof}
Since $L_m$ is compact and disjoint from the closed set 
\begin{equation}
\big\{T_*m=\mathcal E(x)\big\}\cup\big\{\mathcal S_m(x)=0\big\}\supset\Omega_m\cap\partial\Omega_m,    
\end{equation}
see Lemma \ref{l:inclusion}, the sets $L_m$ and $\Omega_m\cap\partial\Omega_m$ lie at a positive distance $d_1$ from each other. Let $x_*\in L_m$ be arbitrary. Then $x_*\in \mathcal S_\mu$ for some $\mu\in[m,2m]$ and $2T_0m\leq\mathcal E(x_*)$. Therefore, $T_0\mu\leq \mathcal E(x_*)$, $\mu<2m_0$ and we conclude that $\mathrm{ind}_\mu(x)\geq 3$ by Lemma \ref{l:index}. Thus, we can apply Lemma \ref{l:local} with $d_0:=\tfrac12d_1$ and get objects
\begin{equation}
C_{x_*},\qquad \gamma_{x_*},\qquad U_{x_*},\qquad V_{x_*},\qquad \hat L_{x_*},\qquad \{\sigma^t_{x_*}\colon\Omega\setminus \hat L_{x_*}\to\Omega\}_{t\in[0,1]}.   
\end{equation}
Finally, let $U_{x_*}^0$ be a neighborhood of $x_*$ that is at a positive distance from $\bar\Omega_m\setminus U_{x_*}$.

Since $L_m$ is compact by Lemma \ref{l:KL}, there exist $x_1,\ldots,x_k\in L_m$ such that 
\begin{equation}
L_m\subset \bigcup_{j=1}^kU_{x_j}^0.
\end{equation}
and there exists $d>0$ such that
\begin{equation}\label{e:distance2}
\frac{1}{k-j}\mathrm{dist}(U_{x_j}^0,\bar\Omega_m\setminus U_{x_j})\geq d,\qquad \forall\,j=1\ldots,k-1.
\end{equation}
Let $t_0\in(0,1]$ be any real number such that
\begin{equation}\label{e:tcd}
t_0\cdot\max\{C_{x_1},\ldots C_{x_k}\}< d.
\end{equation}

Define the positive real number
\begin{equation}\label{e:gammam}
\hat\gamma_m:=t_0^2\cdot \min\{\gamma_{x_1},\ldots\gamma_{x_k}\}
\end{equation}
and the open neighborhood of $L_m$ in $\bar\Omega_m$
\begin{equation}
U_m:=\bigcup_{j=1}^kU_{x_j}^0\cap\bar\Omega_m.
\end{equation}

Define the set $\hat L_m$ iteratively,
\begin{equation}\label{e:hatL}
\begin{aligned}
\hat L_m^1&:=\hat L_{x_1},\\
\hat L_m^{j+1}&:=\hat L_{x_{j+1}}\cup (\sigma^{t_0}_{x_{j+1}})^{-1}(\hat L_m^j),\quad \forall\,j=1,\ldots,k-1,\\
\hat L_m&:=\hat L_m^k.
\end{aligned}
\end{equation}
By induction, $\hat L_m$ is a finite union of submanifolds of codimension $2$. Thus, property (i) of the present lemma is satisfied.

Define again iteratively
\begin{equation}
\begin{aligned}
\sigma_m^1&:=\sigma^{t_0}_{x_1}\colon \Omega\setminus\hat L_{x_1}\to\Omega,\\
\sigma_m^{j+1}&:=\sigma_m^j\circ\sigma^{t_0}_{x_{j+1}}\colon\Omega\setminus\hat L_m^{j+1}\to \Omega,\quad \forall\,j=1,\ldots,k-1,\\ \sigma_m'&:=\sigma_m^k\colon\Omega\setminus\hat L_m\to\Omega.
\end{aligned}
\end{equation}
We claim that the map $\sigma_m'$ is well-defined. Indeed, by the general rule for the domain of a composition of two functions $f_1$ and $f_0$, we have
\begin{equation}\label{e:domain}
\mathrm{domain}(f_1\circ f_0)=f^{-1}_0\big(\mathrm{domain}(f_1)\big)
\end{equation}
and therefore by induction
\begin{equation}
\begin{aligned}
\mathrm{domain}(\sigma_m^{j+1})=\mathrm{domain}(\sigma_m^j\circ\sigma_{x_{j+1}}^{t_0})&=(\sigma_{x_{j+1}}^{t_0})^{-1}\big(\mathrm{domain}(\sigma_m^j))\\
&=(\sigma_{x_{j+1}}^{t_0})^{-1}\big(\Omega\setminus \hat L_m^j\big)\\
&=(\Omega\setminus \hat L_{x_{j+1}})\setminus (\sigma_{x_{j+1}}^{t_0})^{-1}(\hat L^j_m)\\
&=\Omega\setminus \hat L^{j+1}_m.
\end{aligned}
\end{equation}
Moreover, $\sigma_m'$ preserves the function $\mathcal E$ as all $\sigma_{x_j}^{t_0}$ do, according to Lemma \ref{l:local}.(v). Furthermore, since all maps $\sigma_{x_j}^{t_0}$ and all sets $\hat L_{x_j}$ are supported in the set $V_{x_1}\cup\ldots\cup V_{x_k}$, which is at a distance at least $d_1-d_0=\tfrac12d_1$ from $\Omega_m\cap\partial\Omega_m$, it follows that $\hat L_m$ and the support of $\sigma_m'$ are at a positive distance from $\Omega_m\cap\partial\Omega_m$. This shows property (ii) of the present lemma. Finally, since $\mathcal S_m$ does not increase under $\sigma^{t_0}_{x_j}$ for all $j=1,\ldots,k$, the same is true for their composition $\sigma_m'$.

Therefore, $\sigma_m'$ leaves $\Omega_m$ invariant, and we define
\begin{equation}
\hat\sigma_m:=\sigma_m'|_{\Omega_m\setminus\hat L_m}\colon\Omega_m\setminus \hat L_m\to\Omega_m,
\end{equation}
which satisfies properties (ii) and (iv) of the present lemma.

Let us take $x\in (\Omega_m\cap U_m)\setminus\hat L_m$. Therefore, $x\in U_{x_j}^0\cap\Omega_m$ for some $j=1,\ldots,k$. We have
\begin{equation}
\hat\sigma_m(x)=\sigma^{j-1}_m\circ \sigma^{t_0}_{x_j}\circ \breve\sigma^j_m(x),\qquad \breve\sigma^j_m:=\sigma^{t_0}_{x_{j+1}}\circ \ldots\circ\sigma^{t_0}_{x_k}.
\end{equation}
First, note that by \eqref{e:tcd} and Lemma \ref{l:local}.(iii), we have
\begin{equation}
\mathrm{dist}(x,\breve{\sigma}^j_m(x))\leq C_{x_k}t_0+\ldots C_{x_{j+1}}t_0< (k-j) d
\end{equation}
by the definition of $t_0$. Hence, by the definition of $d$ in \eqref{e:distance2}, we have
\begin{equation}
y:=\breve\sigma^j_m(x)\in U_{x_j}.
\end{equation}
Therefore, using Lemma \ref{l:local}, and more precisely property (vi), then property (vii) and then property (vi) again, we get
\begin{equation}
\mathcal S_m(\hat\sigma_m(x))=\mathcal S_m\Big(\sigma^{j-1}_m\big(\sigma^{t_0}_{x_j}(y)\big)\Big)\leq \mathcal S_m\big(\sigma^{t_0}_{x_j}(y)\big)\leq \mathcal S_m(y)-t_0^2\gamma_{x_j}\leq \mathcal S_m(x)-t_0^2\gamma_{x_j}.
\end{equation}
Thus, property (v) of the present lemma follows from the definition of $\hat\gamma_m$ given in \eqref{e:gammam}.
\end{proof}

\subsection{Gradient Bounds from the Palais--Smale Condition}
Now that we have the neighborhood $U_m$ from Lemma \ref{l:sigma}, we can use it to find lower bounds on the norm of the gradient of the action exploiting the Palais--Smale condition proved in Lemma \ref{l:PS}, under the assumption that $K_m=\varnothing$. To this aim, we need the following refinement of the equality case of the Cauchy--Schwarz inequality, which is reminiscent of \cite[Lemma 5.3]{AbbondandoloMajer2008}.
\begin{lem}\label{l:CS}
Let $\mathcal H$ be a real vector space endowed with an inner product. Let $a\in[0,1]$ and assume that $u_0,u_1\in \mathcal H$ are two vectors such that
\begin{equation}
u_0\cdot u_1\leq -(1-a)\Vert u_0\Vert\Vert u_1\Vert.
\end{equation}
Then there exists $\lambda\in[0,1]$ such that
\begin{equation}
\Vert u_\lambda\Vert\leq \sqrt{2a}\min\big\{\Vert u_0\Vert,\Vert u_1\Vert\big\},\qquad u_\lambda:=(1-\lambda)u_0+\lambda u_1.
\end{equation}
\end{lem}
\begin{proof}
If $u_0=u_1=0$, then we can take $u_\lambda=0$. Let us assume that $\Vert u_0\Vert\leq \Vert u_1\Vert$ with $u_1\neq0$. Taking the orthogonal projection of $u_0$ onto the line generated by $u_1$, we get
\begin{equation}
\Big\Vert u_0+\frac{-u_0\cdot u_1}{\Vert u_1\Vert ^2}u_1\Big\Vert^2=\Vert u_0\Vert^2-\frac{(u_0\cdot u_1)^2}{\Vert u_1\Vert^2}\leq \Vert u_0\Vert^2-(1-a)^2\Vert u_0\Vert^2=a(2-a)\Vert u_0\Vert^2.
\end{equation}
Since $c:=\frac{-u_0\cdot u_1}{\Vert u_1\Vert^2}\geq 0$, we can divide both sides of the inequality by $(1+c)^2$ and get
\begin{equation}
\Vert u_\lambda\Vert^2\leq \frac{a(2-a)\Vert u_0\Vert^2}{(1+c)^2}\leq \frac{2a\Vert u_0\Vert^2}{1},\qquad \lambda:=\frac{c}{1+c}.
\end{equation}
\end{proof}
Before establishing the lower bounds, we introduce the following notation.
\begin{dfn}
For every positive real number $\eta_m$, let
\begin{equation}
\begin{aligned}
\Omega_m^{\eta_m}&:=\Big\{\ x\in\Omega_m\ \Big|\ |\mathcal S_m(x)-\mathfrak c(m)|< \eta_m\ \Big\},\\
\widetilde\Omega_m^{\eta_m}&:=\Big\{\ x\in\Omega_m\ \Big|\ |\mathcal S_m(x)-\mathfrak c(m)|< \eta_m,\ 2T_0m< \mathcal E(x) \ \Big\}.\\
\end{aligned}
\end{equation}
Moreover, let $U_m''\subset U_m'$ be any open neighborhoods of $L_m$ inside $\bar\Omega_m$ such that $U_m'$ has a positive distance from $\bar\Omega_m\setminus U_m$, and $U_m''$ has a positive distance from $\bar\Omega_m\setminus U_m'$, see Lemma \ref{l:system}.
\end{dfn}
\begin{lem}\label{l:PS2}
For every $m\in(0,m_1]$ with $K_m=\varnothing$, there exists $\eta_m$ with $0<\eta_m<\min\{\delta m^2,1\}$ such that 
\begin{equation}
\begin{aligned}
(i)&\ \ \forall\,x\in\Omega_m^{\eta_m},\qquad\qquad\qquad\qquad\  \Vert\nabla \mathcal S_m(x)\Vert> \eta_m,\\
(ii)&\ \ \forall\,x\in\widetilde\Omega_m^{\eta_m}\setminus \bar U_m'',\qquad\nabla\mathcal S_m(x)\cdot\nabla\mathcal S_{2m}(x)>-(1-\eta_m)\Vert\nabla \mathcal S_m(x)\Vert\Vert\nabla\mathcal S_{2m}(x)\Vert.
\end{aligned}
\end{equation}
\end{lem}
\begin{proof}
Suppose by contradiction that (i) is false. Then, there exists a sequence $(x_k)\subset\Omega_m$ such that
\begin{equation}
\mathcal S_m(x_k)\to \mathfrak c(m),\qquad\Vert \mathrm{d}_{x_k}\mathcal S_m\Vert\to 0.
\end{equation}
By Lemma \ref{l:PS}, up to a subsequence, $x_k\to x_*\in\mathrm{Crit}\,\mathcal S_m$ with $\mathcal S_m(x_*)=\mathfrak c(m)$. Thus, $x_*\in K_m$, which is a contradiction. 

Suppose by contradiction that (ii) is false. Then, there exist sequences $(a_k)\subset(0,\min\{\delta m^2,1\})$ and $(x_k)\subset\Omega_m\setminus \bar U_m''$ such that 
\begin{equation}
\begin{aligned}
(i)&\ \ \mathcal S_m(x_k)\to \mathfrak c(m),\\
(ii)&\ \ 2T_0m<\mathcal E(x_k),\\
(iii)& \ \ a_k\to 0,\\
(iv)&\ \ \nabla\mathcal S_m(x_k)\cdot\nabla\mathcal S_{2m}(x_k)\leq -(1-a_k)\Vert\nabla \mathcal S_m(x_k)\Vert\Vert\nabla\mathcal S_{2m}(x_k)\Vert.
\end{aligned}
\end{equation}
Applying Lemma \ref{l:CS} with $u_0=\nabla\mathcal S_m(x_k)$ and $u_1=\nabla\mathcal S_{2m}(x_k)$ and $a=a_k$, there exists a sequence $(\lambda_k)\in[0,1]$ such that 
\begin{equation}
\Vert\nabla\mathcal S_{\mu_k}(x_k)\Vert\leq \sqrt{2a_k}\min\big\{\Vert\nabla\mathcal S_m(x_k)\Vert,\Vert\nabla\mathcal S_{2m}(x_k)\Vert\big\},
\end{equation}
where
\begin{equation}
\mu_k:=(1-\lambda_k)m+\lambda_k(2m)\in[m,2m],
\end{equation}
and we have used that
\begin{equation}
(1-\lambda_k)\nabla \mathcal S_m(x_k)+\lambda_k\nabla\mathcal S_{2m}(x_k)=\nabla\mathcal S_{(1-\lambda_k)m+\lambda_k(2m)}(x_k)=\nabla\mathcal S_{\mu_k}(x_k).
\end{equation}
Since $\nabla\mathcal E$ and $\nabla\mathcal B$ are uniformly bounded on $\Omega_m$ due to Lemma \ref{l:differentials}, we deduce that, up to a subsequence, \begin{equation}
\mu_k\to\mu_*\in[m,2m],\qquad \Vert\nabla\mathcal S_{\mu_k}(x_k)\Vert\to 0.    
\end{equation}
Using again that $\nabla\mathcal E$ is bounded, we deduce that
\begin{equation}
\Vert\nabla\mathcal S_{\mu_*}(x_k)\Vert\to 0.
\end{equation}
By Lemma \ref{l:PS}, we conclude that, up to a subsequence,
\begin{equation}
x_k\to x_*\in(\mathrm{Crit}\,\mathcal S_{\mu_*})\setminus U_m'',\qquad \mathcal S_m(x_*)=\mathfrak c(m),\qquad 2T_0m\leq\mathcal E(x_*).
\end{equation}
Thus, $x_*\in L_m\setminus U_m''=\varnothing$, and we have reached a contradiction.
\end{proof}
\subsection{The Pseudogradient and its Flow}
Using the bounds on the gradients of $\mathcal S_m$ and $\mathcal S_{2m}$, we obtained in the last subsection from the Palais--Smale condition, we can build a pseudogradient for both $\mathcal S_m$ and $\mathcal S_{2m}$ by interpolation.
\begin{lem}\label{l:wm}
If $m\in(0,m_1]$ and $K_m=\varnothing$, then there exist a real number $\theta_m\in(0,\delta m^2)$ and a smooth and bounded vector field $W_m$ on $\Omega_m$ such that
\begin{enumerate}[(i)]
\item $\mathrm{d}_x\mathcal S_m[W_m]\leq 0$ for all $x\in \Omega_m$;
\item $\mathrm{d}_x\mathcal S_m[W_m]\leq -\theta_m$ for all $x\in\Omega_m^{\theta_m}\setminus U_m'$;
\item $\mathrm{d}_x\mathcal S_{2m}[W_m]\leq 0$ for all $x\in\Omega_m$ with $2T_1m\leq \mathcal E(x)$, and the inequality is strict if $W_m(x)\neq0$;
\item the flow of $W_m$ is forward-complete and its support is at a positive distance from $\Omega_m\cap\partial\Omega_m$.
\end{enumerate}
\end{lem}
\begin{proof}
Let $\eta_m\in\big(0,\min\{\delta m^2,1\}\big)$ be given by Lemma \ref{l:PS}. Take any real numbers
\begin{equation}
\eta_m'\in(0,\eta_m),\qquad T_2\in(T_0,T_1).
\end{equation}
Let $\chi\colon\Omega_m\to[0,1]$ be any smooth function that satisfies for all $x\in\Omega_m$
\begin{equation}\label{e:chi}
\begin{aligned}
\mathcal E(x)\leq 2T_2m\quad&\Longrightarrow\quad \chi(x)=0,\\
2T_1m\leq \mathcal E(x)\quad& \Longrightarrow\quad\chi(x)> 1-\eta_m.
\end{aligned}
\end{equation}
This function exists since $1-\eta_m\in(0,1)$. Let $\rho\colon \Omega_m\to[0,1]$ be any smooth function that satisfies
\begin{equation}\label{e:rho}
\begin{aligned}
|\mathcal S_m(x)-\mathfrak c(m)|\geq\eta_m\quad&\Longrightarrow\quad \rho(x)=0,\\
x\in \bar U_m''\quad&\Longrightarrow\quad \rho(x)=0,\\
x\in\Omega^{\eta_m'}\setminus U_m'\quad& \Longrightarrow\quad\rho(x)=1.
\end{aligned}
\end{equation}
This function exists since $\bar U_m''\cup\{|\mathcal S_m(x)-\mathfrak c(m)|\geq\eta_m\}$ is disjoint from the closure of $\Omega^{\eta_m'}\setminus U_m'$.

For all $x\in\Omega_m$, we define
\begin{equation}
W_m(x):=\begin{cases}
-\rho(x)\Big(\nabla \mathcal S_m(x)+\chi(x)\frac{\Vert\nabla\mathcal S_m(x)\Vert}{\Vert\nabla\mathcal S_{2m}(x)\Vert}\nabla \mathcal S_{2m}(x)\Big),&\text{if }\nabla \mathcal S_{2m}(x)\neq0,\\
-\rho(x)\nabla \mathcal S_m(x),& \text{if }\mathcal E(x)<2T_2m,\\
0,&\text{if }x\in U_m''.
\end{cases}
\end{equation}
This is a good definition. Indeed, 
\begin{itemize}
    \item the three conditions cover the whole $\Omega_m$ since $T_0<T_2$ and $U_m''$ is a neighborhood of $L_m$;
    \item the three definitions coincide in the intersection of the three cases by \eqref{e:chi} and \eqref{e:rho}.
\end{itemize}
Moreover, since the three conditions are open, the vector field is smooth. Finally, the vector field is bounded by Lemma \ref{l:differentials}.

Let $x\in\Omega_m$ and let us estimate $\mathrm{d}_x\mathcal S_m[-W_m]$ from below. If $\rho(x)=0$, we have
\begin{equation}
\mathrm{d}_x\mathcal S_m[-W_m]=0.
\end{equation}
If $\rho(x)\neq0$, we distinguish two cases that cover all possibilities
\begin{equation}
\text{Case 1:}\quad  \mathcal E(x)<2T_2m;\qquad\qquad \text{Case 2:}\quad  x\in\widetilde\Omega^{\eta_m}\setminus \bar U''_m.
\end{equation}

In Case 1, Lemma \ref{l:PS2}.(i) implies
\begin{equation}\label{e:w1}
\mathrm{d}_x\mathcal S_m[-W_m]=\rho(x)\Vert\nabla\mathcal S_m(x)\Vert^2\geq \rho(x)\eta_m^2.
\end{equation}
In particular, if $x\in\Omega^{\eta_m'}\setminus U_m'$ also holds, then we get by the definition of $\rho$ in \eqref{e:rho}
\begin{equation}\label{e:w2}
\mathrm{d}_x\mathcal S_m[-W_m]\geq\eta_m^2.    
\end{equation}

In Case 2, we have $\nabla\mathcal S_{2m}(x)\neq0$ since $L_m\subset U_m''$. Then, Lemma \ref{l:PS2}.(i)-(ii) implies
\begin{equation}\label{e:w3}
\begin{aligned}
\mathrm{d}_x\mathcal S_m[-W_m]&=\rho(x)\Big(\Vert \nabla \mathcal S_m(x)\Vert^2+\chi(x)\frac{\Vert\nabla\mathcal S_m(x)\Vert}{\Vert\nabla\mathcal S_{2m}(x)\Vert}\nabla\mathcal S_m(x)\cdot\nabla \mathcal S_{2m}(x)\Big)\\
&\geq \rho(x)\big(1-(1-\eta_m)\chi(x)\big)\Vert\nabla\mathcal S_m(x)\Vert^2\\
&\geq \rho(x)\big(1-(1-\eta_m)\cdot 1\big)\eta_m^2\\
&\geq \rho(x)\eta_m^3.
\end{aligned}
\end{equation}
In particular, if $x\in\Omega^{\eta_m'}\setminus U_m'$ also holds, then by the definition of $\rho$ in \eqref{e:rho}
\begin{equation}\label{e:w4}
\mathrm{d}_x\mathcal S_m[-W_m]\geq \eta_m^3.
\end{equation}
We deduce that $W_m$ satisfies property (i) of the current lemma. Moreover, letting
\begin{equation}
\theta_m:=\min\big\{\eta_m',\eta_m^3\big\}
\end{equation}
the vector field $W_m$ also satisfies property (ii) of the current lemma. This follows since $\rho(x)\neq0$ if $x\in\Omega^{\theta_m}\setminus U_m'$ and, therefore, we can use \eqref{e:w2} and \eqref{e:w4}.

Let us now estimate $\mathrm{d}_x\mathcal S_{2m}[-W_m]$ from below for all $x\in\Omega_m$ such that $2T_1m\leq\mathcal E(x)$. If $\rho(x)=0$, then
\begin{equation}
\mathrm{d}_x\mathcal S_{2m}[-W_m]=0.
\end{equation}
If $\rho(x)\neq0$, it follows that $x\in\widetilde\Omega^{\eta_m}\setminus\bar U''_m$ since $T_0\leq T_1$. Thus $\nabla \mathcal S_m(x)\neq0$ and $\nabla\mathcal S_{2m}(x)\neq0$ and 
\begin{equation}\label{e:w5}
\begin{aligned}
d_x\mathcal S_{2m}[-W_m]&=\rho(x)\Big(\nabla\mathcal S_{2m}(x)\cdot\nabla \mathcal S_m(x)+\chi(x)\Vert\nabla\mathcal S_m(x)\Vert\Vert\nabla\mathcal S_{2m}(x)\Vert\Big)\\
&> \rho(x)\big(\chi(x)-(1-\eta_m)\big)\Vert\nabla\mathcal S_m(x)\Vert \Vert\nabla\mathcal S_{2m}(x)\Vert\\
&>0,
\end{aligned}
\end{equation}
where we used Lemma \ref{l:PS2}.(ii) and \eqref{e:chi}. Thus, we also get property (iii) of the current lemma.

It remains to show property (iv) about the support of $W_m$ and of the forward completeness of its flow. To this end, we preliminarily observe that since $\eta<\delta m^2$, we have the implication
\begin{equation}
\mathcal S_m(x)<\mathfrak c(m)-\delta m^2\quad \Longrightarrow\quad \rho(x)=0.
\end{equation}
The set $\big\{\mathcal S_m(x)\geq \mathfrak c(m)-\delta m^2\big\}$ is at a positive distance from $\Omega_m\cap\partial\Omega_m$ by \eqref{e:partial}, Lemma \ref{l:lowerc}.(ii)-(iii), and the fact that the gradients of $\mathcal S_m$ and $\mathcal E$ are bounded, see Lemma \ref{l:differentials}. Therefore, we conclude that the support of $W_m$ is at a positive distance from $\Omega_m\cap\partial\Omega_m$. 

In particular, if $x_0\in\Omega_m$ and $x\colon [0,r_+)\to \Omega$ is the maximal flow line of $W_m$ such that $x(0)=x_0$, then $r_+\in(0,+\infty]$. We need to show that $r_+=\infty$. Suppose, by contradiction, that $r_+<\infty$. In this case, $x$ is not constant. Moreover, since $\bar\Omega_m$ is a complete metric space and $W_m$ is bounded, it follows that 
\begin{equation}
\exists\, x_*\in\partial\Omega_m\setminus\Omega_m,\quad \lim_{r\to r_+}x(r)=x_*.   
\end{equation}
By \eqref{e:partial}, we deduce that $\mathcal S_{2m}(x_*)=Dm^2$. Therefore, by Lemma \ref{l:D}.(i)
\begin{equation}
\exists\,r_-\in[0,r_+),\quad 2T_1 m< \mathcal E(x(r)),\quad \forall\,r\in(r_-,r_+).     
\end{equation}
By property (iii) of the present lemma, we deduce that $r\mapsto \mathcal S_{2m}(x(r))$ is strictly decreasing for $r\in(r_-,r_+)$. This is a contradiction, since $\mathcal S_{2m}(x)<Dm^2$ for all $x\in\Omega_m$. The proof of property (iv) is complete.
\end{proof}
Now that we have constructed the pseudogradient $W_m$, we use its flow to decrease the action.
\begin{lem}\label{l:phi}
For every $m\in(0,m_1]$ with $K_m=\varnothing$, there exist
\begin{enumerate}[(a)]
\item a real number $\hat{\theta}_m\in(0,\theta_m]$,
\item a smooth diffeomorphism $\varphi_m\colon \Omega_m\to \Omega_m$, 
\end{enumerate}
with the following properties
\begin{enumerate}[(i)]
\item the support of $\varphi_m$ is at a positive distance from $\Omega_m\cap\partial\Omega_m$;
\item $\mathcal S_m(\varphi_m(x))\leq \mathcal S_m(x)$ for all $x\in\Omega_m$;
    \item $\mathcal S_m(\varphi_m(x))<\mathfrak c(m)$ for all $x\in \Omega_m\setminus \varphi_m^{-1}(U_m)$ with $\mathcal S_m(x)<\mathfrak c(m)+\hat\theta_m$.
\end{enumerate}
\end{lem}
\begin{proof}
Denote by $\{\Phi_t\}_{t\geq 0}$ the flow of $W_m$. Since $U_m'$ has a positive distance from $\bar\Omega_m\setminus U_m$ and the vector field $W_m$ is bounded by Lemma \ref{l:wm}, we deduce that 
\begin{equation}\label{e:t0}
\exists\,t_1\in(0,1],\quad \Phi_t(U_m')\subset U_m,\quad \forall\,t\in[0,t_1].
\end{equation}
By Lemma \ref{l:wm}.(iv), $W_m$ is forward-complete on $\Omega_m$ and its support has positive distance from $\Omega_m\cap\partial\Omega_m$. Hence, the diffeomorphism $\Phi_{t_1}\colon \Omega_m\to\Omega_m$ is well-defined and its support has a positive distance from $\Omega_m\cap\partial\Omega_m$.

We let
\begin{equation}
\begin{aligned}
\hat\theta_m:=\theta_m t_1,\qquad \varphi_m:=\Phi_{t_1}\colon \Omega_m\to\Omega_m.    
\end{aligned}
\end{equation}
Thus, property (i) of the present lemma is satisfied. Moreover, by Lemma \ref{l:wm}.(i), property (ii) of the present lemma is satisfied as well.

We are left with proving property (iii). Let $x\in\Omega_m\setminus \varphi_m^{-1}(U_m)$ with $\mathcal S_m(x)<\mathfrak c(m)+\hat\theta_m$. Suppose by contradiction that $\mathcal S_m(\varphi_m(x))\geq \mathfrak c(m)$. By Lemma \ref{l:wm}.(i), it follows that
\begin{equation}\label{e:window}
\mathcal S_m(\Phi_t(x))\in[\mathfrak c(m),\mathfrak c(m)+\hat\theta_m),\qquad\forall\,t\in[0,t_1].
\end{equation}
Since $\hat\theta_m\leq\theta_m$, we get
\begin{equation}\label{e:flowline1}
\Phi_t(x)\in\Omega_m^{\theta_m},\qquad\forall\,t\in[0,t_1].
\end{equation}
Moreover,
\begin{equation}\label{e:flowline2}
\Phi_t(x)\notin \Omega_m\cap U_m',\qquad \forall\, t\in[0,t_1]. 
\end{equation}
Indeed, if, by contradiction, $\Phi_{t_*}(x)\in\Omega_m\cap U_m'$ for some $t_*\in[0,t_1]$, then
\begin{equation}
\varphi_m(x)=\Phi_{t_1}(x)=\Phi_{t_1-t_*}(\Phi_{t_*}(x))\in \Omega_m\cap U_m
\end{equation}
by the flow property and letting $t=t_1-t_*$ in \eqref{e:t0}. This contradicts the fact that $x\notin\varphi_m^{-1}(U_m)$ and establishes \eqref{e:flowline2}. Using \eqref{e:flowline1} and \eqref{e:flowline2} together with property (ii) of Lemma \ref{l:wm}, we get
\begin{equation}
\mathcal S_m(\varphi_m(x))=\mathcal S_m(\Phi_{t_1}(x))=\mathcal S_m(x)+\int_0^{t_1}\mathrm{d}_{\Phi_t(x)}\mathcal S_m[W_m]\mathrm{d}t<\big(\mathfrak c(m)+\hat\theta_m\big)-\theta_mt_1=\mathfrak c(m).
\end{equation}
which contradicts \eqref{e:window} for $t=t_1$.
\end{proof}
\subsection{Combining the Local and Global Deformations: End of the Proof of Theorem \ref{t:D}}
Combining the local deformation $\hat\sigma_m$ from Lemma \ref{l:sigma} and the global deformation $\varphi_m$ from Lemma \ref{l:phi}, we let the action decrease in a complement of a subset of codimension $2$.
\begin{lem}\label{l:tildesigma}
For every $m\in(0,m_1]$ with $K_m=\varnothing$, there exist
\begin{enumerate}[(a)]
\item a real number $\tilde{\gamma}_m>0$,
\item a set $\tilde L_m\subset\bar\Omega_m$,
\item a smooth map $\tilde\sigma_m\colon \Omega_m\setminus \tilde L_m\to \Omega_m$, 
\end{enumerate}
with the following properties
\begin{enumerate}[(i)]
    \item the set $\tilde L_m$ is a finite union of embedded submanifolds of $\bar\Omega_m$ of codimension $2$;
    \item the set $\tilde L_m$ is at a positive distance from $\Omega_m\cap\partial\Omega_m$;
    \item the support of the map $\tilde\sigma_m$ is at a positive distance from $\Omega_m\cap\partial\Omega_m$;
    \item $\mathcal S_m(\tilde\sigma_m(x))<\mathfrak c(m)$ for all $x\in \Omega_m\setminus \tilde L_m$ with $\mathcal S_m(x)<\mathfrak c(m)+\tilde\gamma_m$. \end{enumerate}
\end{lem}
\begin{proof}
We let
\begin{equation}\label{e:tildes}
\begin{aligned}
\tilde\gamma_m&:=\min\{\hat\gamma_m,\hat\theta_m\},\\
\tilde L_m&:=\varphi_m^{-1}(\hat L_m),\\ \tilde\sigma_m&:=\hat\sigma_m\circ\varphi_m\colon \Omega_m\setminus \tilde L_m\to\Omega_m.    
\end{aligned}
\end{equation}
By the definition of $\hat \sigma_m$ in Lemma \ref{l:sigma} and identity \eqref{e:domain}, the map $\tilde\sigma_m$ is well-defined.
By Lemma \ref{l:sigma}, the set $\tilde L_m$ is a finite union of embedded submanifolds of $\bar\Omega_m$ of codimension $2$, thus proving property (i) of the present lemma. Since $\hat L_m$ and the supports of $\varphi_m$ and $\hat\sigma_m$ are at a positive distance from $\Omega_m\cap\partial\Omega_m$, properties (ii) and (iii) of the present lemma follow.

We are left with proving property (iv). Let $x\in\Omega_m\setminus \tilde L_m$ be such that 
\begin{equation}
\mathcal S_m(x)<\mathfrak c(m)+\tilde\gamma_m.
\end{equation}
We have two cases. In the first case, we assume that $x\in\Omega_m\setminus \varphi_m^{-1}(U_m)$. In this case, since $\tilde\gamma_m\leq\hat\theta_m$, we can apply Lemma \ref{l:phi}.(iv) and get
\begin{equation}
\mathcal S_m(\tilde\sigma_m(x))=\mathcal S_m(\hat\sigma_m(\varphi_m(x)))\leq \mathcal S_m(\varphi_m(x))<\mathfrak c(m), 
\end{equation}
where we have used Lemma \ref{l:sigma}.(iv). 
In the second case, we assume that
$\varphi_m(x)\in (\Omega_m\cap U_m)$. By the definition \eqref{e:tildes} of $\tilde L_m$, it also follows that $\varphi_m(x)\notin\hat L_m$. Then by property (v) in Lemma \ref{l:sigma}, we get
\begin{equation}
\mathcal S_m(\tilde\sigma_m(x))=\mathcal S_m\big(\hat\sigma_m(\varphi_m(x))\big)\leq \mathcal S_m(\varphi_m(x))-\hat\gamma_m<\mathfrak c(m)+\tilde\gamma_m-\hat\gamma_m\leq \mathfrak c(m),
\end{equation}
where the second inequality follows from property (i) in Lemma \ref{l:wm}, and the last inequality from $\tilde\gamma_m\leq\hat\gamma_m$, see definition \eqref{e:tildes} of $\tilde\gamma_m$.
\end{proof}
We now have all the ingredients to prove Theorem \ref{t:D}.
\begin{proof}[Proof of Theorem \ref{t:D}]
By \eqref{e:implication}, it is enough to show that for all $m\in(0,m_1]$, the set $K_m$ is non-empty. Assume, by contradiction, that $K_m=\varnothing$ for some $m\in(0,m_1]$. By the definition of the minimax value $\mathfrak c(m)$ given in \eqref{e:cm}, there exists $x\subset X_m$ such that 
\begin{equation}\label{e:smyr}
\mathcal S_m(x(r))< \mathfrak c(m)+\tilde\gamma_m,\qquad\forall\,r\in[0,1].
\end{equation}
By properties (i) and (ii) in Lemma \ref{l:tildesigma}, $\tilde L_m\subset \bar\Omega_m$ is a finite union of embedded submanifolds of codimension two and $\tilde L_m$ lies at a positive distance from the set $\Omega_m\cap\partial\Omega_m$, see \eqref{e:partial}. Since the domain of $x$ is the interval $[0,1]$, which has dimension $1<2$, by the Transversality Theorem \cite{Abraham0} we can perturb $x$ keeping its endpoints in $\Omega_m\cap\partial\Omega_m$ fixed to get $y\in X_m$ such that
\begin{equation}\label{e:final}
\forall\,r\in[0,1],\qquad (i)\ \ y(r)\notin \tilde L_m,\qquad (ii)\ \ \mathcal S_m(y(r))< \mathfrak c(m)+\tilde\gamma_m.
\end{equation}
By (i) in \eqref{e:final}, the path $\tilde y:=\tilde\sigma_m\circ y\colon[0,1]\to\Omega_m$ is well defined. By Lemma \ref{l:tildesigma}.(iii), $\tilde y\in X_m$. By (ii) in \eqref{e:final} and Lemma \ref{l:tildesigma}.(iv), we finally deduce that 
\begin{equation}
\mathcal S_m(\tilde y(r))<\mathfrak c(m),
\end{equation}   
which contradicts the definition of the minimax given in \eqref{e:cm} and completes the proof. 
\end{proof}

\section{Blow-up Analysis and Localization}\label{s:local}
In this section, we deduce Theorem \ref{t:C} from Theorem \ref{t:D}. Thus, let $Q$ be a compact manifold endowed with a Riemannian metric $g$ and a nowhere vanishing magnetic field $\beta$. 

\subsection{The Blow-Up Analysis}Let $D>0$ be given by Theorem \ref{t:D} and let $\varepsilon'>0$ be arbitrary. Consider any sequence of speeds $(m_k)$ converging to $0$ and let $(x_k)$ be any sequence of loops in $Q$ such that for all $k$, we get
\begin{equation}\label{e:blowupbounds}
x_k\in\mathrm{Crit}\,\mathcal S_{m_k},\qquad \mathcal E(x_k)\leq Dm_k,\qquad \mathcal S_{m_k}(x_k)\leq \frac{\pi m_k^2}{\Vert\beta\Vert_\infty}(1+\varepsilon').
\end{equation}
We define $T_k:=T_{m_k}(x_k)=\tfrac{1}{m_k}\mathcal E(x_k)$. Therefore,
\begin{equation}\label{e:blowupba}
C\leq T_k\leq D,\qquad \forall\,k\in\mathbb N,
\end{equation}
where the first inequality follows from Lemma \ref{l:lowerboundperiod}. By Lemma \ref{l:crit}, the reparameterized loops $q_k$, where $q_k(t):=x_k(t/T_k)$ for all $t\in\mathbb R$, are contractible periodic magnetic geodesics with speed $m_k$, period $T_k$, and length
\begin{equation}
\ell(q_k)=\mathcal E(x_k)=m_kT_k.
\end{equation}

Let $SQ\subset TQ$ be the unit tangent bundle of $g$, and normalize the tangent vectors as
\begin{equation}
u_k:=\Big(q_k,v_k:=\frac{1}{m_k}\dot q_k\Big)\colon\mathbb R\to SQ.
\end{equation}
Since $SQ$ is compact, up to taking a subsequence, there exists $(q_*,v_*)\in SQ$ such that
\begin{equation}\label{e:limitinitial}
(q_k(0),v_k(0))\to (q_*,v_*).
\end{equation}
Moreover, since $\mathcal E(x_k)$ is converging to zero by \eqref{e:blowupbounds}, we also deduce that $q_k$ is converging uniformly to the constant $q_*\in Q$ in the $C^0$-topology. Finally, we observe that $u_k=(q_k,v_k)$ satisfies the differential equation
\begin{equation}\label{e:limitODE}
D_tv=B_{q}v,
\end{equation}
where $D_t$ denotes the covariant derivative. Equation \eqref{e:limitODE} can be seen as a first-order differential equation on the manifold $SQ$ with vector field $Y_{m_k}$, where
\begin{equation}\label{e:Ym}
Y_m(q,v):=mG(q,v)+B(q,v).
\end{equation}
Here, $G$ is the geodesic vector field of $g$ and 
\begin{equation}
B(q,v):=(B_qv)^{\mathrm{ver}},\qquad \forall\,(q,v)\in SQ
\end{equation}
is the vertical lift of the Lorentz force endomorphism. We studied the solutions to \eqref{e:limitODE} when $q=q_*$ is constant in Section \ref{s:linear}. The discussion contained in that section yields the following result.

\begin{lem}\label{l:blowupbb}
The vector $v_*$ in \eqref{e:limitinitial} belongs to $(\ker B_{q_*})^\perp$.
\end{lem}
\begin{proof}
We assume by contradiction that $v_*\notin(\ker B_{q_*})^\perp$. Therefore, in the splitting given by \eqref{e:identify}, we have
\begin{equation}
v_*=(z_*,w_*)\in (\ker B_{q_*})^\perp\times\ker B_{q_*},\qquad w_*\neq0.
\end{equation}
We consider coordinates on $Q$ centered at $q_*$ denoted by
\begin{equation}
(\zeta,\omega)\in\mathbb C^j\times\mathbb R^k\cong (\ker B_{q_*})^\perp\times\ker B_{q_*}
\end{equation}
such that the coordinate vectors at the origin yield an orthonormal basis of $(\ker B_{q_*})^\perp$ and $\ker B_{q_*}$. In these coordinates, we write
\begin{equation}
q_k=(\zeta_k,\omega_k).    
\end{equation}
By the continuous dependence of the solution of ordinary differential equations on initial data applied to \eqref{e:limitODE}, we deduce that 
\begin{equation}
u_k=(\zeta_k,\omega_k,z_k,w_k),\qquad \text{where}\ \ z_k:=\tfrac{1}{m_k}\dot\zeta_k,\quad w_k:=\tfrac{1}{m_k}\dot\omega_k
\end{equation}
converges uniformly on compact intervals of time to \begin{equation}
u=(0,0,z,w_*),\qquad \text{for some }z\colon \mathbb R\to (\ker B_{q_*})^\perp,\quad z(0)=z_*.    
\end{equation}
In particular, $w_k=\tfrac{1}{m_k}\dot \omega_k$ converges uniformly to the constant $w_*$ on compact intervals of time. Integrating $\tfrac{1}{m_k}\dot \omega_k$ over the interval $[0,T_k]\subset [0,D]$, it follows that in the coordinates around $q_*$
\begin{equation}
0=\frac{1}{m_k}\Big(\omega_k(T_k)-\omega_k(0)\Big)=T_kw_*+o(1),
\end{equation}
which is a contradiction since $T_k\geq C$ for all $k$ and $w_*\neq0$.
\end{proof}
By \eqref{e:blowupba}, up to taking a subsequence, we can assume that $T_k\to T>0$. Therefore,
\begin{equation}
u_k=(q_*,v)+o(1)
\end{equation}
where $(q_*,v)$ is an orbit of $Y_0$ given in \eqref{e:Ym} with period $T$, the orbit $v$ lies in $\ker B_{q_*}\cap S_{q_*}Q$, and $o(1)$ is small in the $C^0$-topology, uniformly on compact time intervals. It follows that in normal coordinates centered at $q_*$ and modeled on $T_{q_*}Q$, we have
\begin{equation}\label{e:omk}
q_k=\bar q_k+o(m_k),
\end{equation}
where $\bar q_k$ is a magnetic geodesic for the constant metric $g_{q_*}$ and the constant magnetic field $\beta_{q_*}$ with period $T$ and speed $m_k$, and $o(m_k)$ is a remainder which is small in the $C^1$-topology, uniformly on compact time intervals. By inequality \eqref{e:linearperiod2} in Lemma \ref{l:boundslinear}, we have
\begin{equation}\label{e:lowerT2}
T\geq \frac{2\pi}{|\beta_{q_*}|}.
\end{equation}
Since $q_k=q_*+o(1)$, we see that 
\begin{equation}\label{e:o1o1}
g_{q_k}=g_{q_*}+o(1),\qquad \beta_{q_k}=\beta_{q_*}+o(1).
\end{equation}
Therefore, we deduce, letting $e_k:=\tfrac12m_k^2$, that
\begin{equation}\label{e:expansion}
S_{e_k}(q_k)=S_{e_k}^{q_*}(q_k)+o(m_k^2)=\Big(S_{e_k}^{q_*}(\bar q_k)+o(m_k^2)\Big)+o(m_k^2)=\frac{1}{2}m_k^2T+o(m_k^2),
\end{equation}
where 
\begin{itemize}
    \item in the first equality we used \eqref{e:o1o1} and $|\dot q_k|=m_k$,
    \item in the second equality we used \eqref{e:omk},
    \item in the third equality we used \eqref{e:linearaction} of Lemma \ref{l:boundslinear}.
\end{itemize}

On the other hand, by \eqref{e:blowupbounds}, we have 
\begin{equation}
S_{e_k}(q_k)=\mathcal S_{m_k}(x_k)\leq\frac{\pi m_k^2}{\Vert\beta\Vert_\infty}(1+\varepsilon').
\end{equation} 
Plugging this inequality into \eqref{e:expansion} and passing to the limit in $k$, we get the upper bound 
\begin{equation}
T\leq \frac{2\pi}{\Vert\beta\Vert_\infty}(1+\varepsilon').
\end{equation}
Combining this inequality with the lower bound \eqref{e:lowerT2}, we get
\begin{equation}
|\beta_{q_*}|\geq \frac{\Vert\beta\Vert_\infty}{1+\varepsilon'}.
\end{equation}
Let us take any $\varepsilon>\varepsilon'$. Since $q_k$ is uniformly converging to $q_*$, and $T_k\to T$, we conclude for all $k$ large enough
\begin{equation}\label{e:betabeta}
\forall\,t\in\mathbb R,\ \ |\beta_{q_k(t)}|\geq \frac{\Vert\beta\Vert_\infty}{1+\varepsilon},\qquad 
\ell(q_k)=m_kT_k\leq \frac{2\pi m_k}{\Vert\beta\Vert_\infty}(1+\varepsilon).
\end{equation}
\subsection{The Proof of Theorem \ref{t:C}}
Given the blow-up analysis of the previous subsection, the proof of Theorem \ref{t:C} readily follows from the classical subsequence argument.
\begin{proof}[Proof of Theorem \ref{t:C}]
Assume by contradiction that for some open neighborhood $V\subset Q$ of the set $\{q\in Q\ |\ |\beta_q|=\Vert\beta\Vert_\infty\}$ there is a sequence $m_k\to 0$ such that for all sequences $(q_k)$ of contractible periodic magnetic geodesics with speed $m_k$, 
\begin{equation}
q_k\not\subset V,\qquad \forall\,k\in\mathbb N.
\end{equation}
Since $Q$ is compact, there exists $\varepsilon>0$ such that
\begin{equation}\label{e:epsilonV}
\Big\{q\in Q\ \Big|\ |\beta_q|\geq \frac{\Vert\beta\Vert_\infty}{1+\varepsilon}\Big\}\subset V.
\end{equation}
Applying Theorem \ref{t:D} with $\varepsilon'\in(0,\varepsilon)$, given the sequence $(m_k)$ there is a sequence $(x_k)$ satisfying \eqref{e:blowupbounds} for $\varepsilon'\in(0,\varepsilon)$ and all $k$ large enough. Thus, by \eqref{e:betabeta} and \eqref{e:epsilonV}, up to taking a subsequence, the reparameterized subsequence $(q_k)$ is contained in $V$ for $k$ large enough, which is a contradiction.

Similarly, assume that for some $\varepsilon>0$ there is a sequence $m_k\to 0$ such that all sequences $(q_k)$ of contractible periodic magnetic geodesics with speed $m_k$ have length
\begin{equation}
    \ell(q_k)> \frac{2\pi m_k}{\Vert\beta\Vert_\infty}(1+\varepsilon).
\end{equation}
Applying Theorem \ref{t:D} with $\varepsilon'\in(0,\varepsilon)$, given the sequence $(m_k)$ there is a sequence $(x_k)$ satisfying \eqref{e:blowupbounds} for all $k$ large enough. Then, by the argument above, up to taking a subsequence, the reparameterized subsequence $(q_k)$ satisfies 
\begin{equation}
\ell(q_k)\leq\frac{2\pi m_k}{\Vert\beta\Vert_\infty}(1+\varepsilon)
\end{equation}
for $k$ large enough, which is a contradiction.
\end{proof}
\section{A Non-compact Example without Periodic Magnetic Geodesics at Low Speed}\label{A Non-compact Example without Periodic Magnetic Geodesics at Low Speed}
This section is entirely devoted to the proof of Theorem \ref{t:B}.
\subsection{A Warped-Product Metric on $\mathbb R^2$}Let $(x,y)$ be coordinates on $\mathbb R^2$ and consider a metric
\begin{equation}
\bar g=\mathrm{d}x^2+a(x)^2\mathrm{d}y^2,    
\end{equation}
where $a\colon \mathbb R\to(0,\infty)$ is a smooth function. The Gaussian curvature $K\colon\mathbb R\to\mathbb R$ of $\bar g$ is given by the formula
\begin{equation}\label{e:curvature}
K=-\frac{a''}{a}.
\end{equation}
Let us take as $K$ any smooth function that satisfies
\begin{equation}\label{e:boundcurv}
-4\leq K\leq -1,\qquad 0<K'\leq 1
\end{equation}
and define $a$ as the unique solution to equation \eqref{e:curvature} with the initial conditions
\begin{equation}
a(0)=1,\quad a'(0)=0.
\end{equation}
We need to check that such a solution is always positive. Notice, indeed, that since $K<0$, the function $a$ cannot have any positive maximum, and since $a(0)=1$, we deduce that
\begin{equation}\label{e:bounda}
a(x)>1,\qquad \forall\,x\neq0.
\end{equation}
Thus, $\mathrm{d}x^2+a(x)\mathrm{d}y^2\geq \mathrm{d}x^2+\mathrm{d}y^2$ and since the Euclidean metric is complete, so is the metric $\bar g$.

To better study the function $a$ it is convenient to introduce the auxiliary function
\begin{equation}
f\colon\mathbb R\to\mathbb R,\qquad f:=\frac{a'}{a}.
\end{equation}
Then $f(0)=0$ and the Riccati equation
\begin{equation}\label{e:Riccati}
f'+f^2+K=0
\end{equation}
holds. Using \eqref{e:boundcurv}, the comparison theorem for ordinary differential equations yields
\begin{equation}\label{e:boundf}
    \tanh(x)\leq f(x)\leq 2\tanh(2x),\qquad \forall\,x\in\mathbb R.
\end{equation}
In particular, $f$ is bounded. By \eqref{e:boundcurv} and \eqref{e:Riccati} the functions $f'$ and $f''$ are also bounded.

We take as magnetic form the area of $\bar g$, that is, \begin{equation}
\bar\beta=a(x)\mathrm{d}x\wedge \mathrm{d}y.    
\end{equation}
In particular, the magnetic strength in this example is constant equal to $1$. Therefore, the whole $\mathbb R^2$ is a non-empty strict local maximum for the magnetic strength.

\subsection{The Power Series Expansion of the Action}By \cite{AB2021} the magnetic geodesics at speed $m$ for the pair $(\bar g,\bar \beta)$ have the integral of motion
\begin{equation}
I=ma(x)\sin\phi-A(x),
\end{equation}
where $A\colon\mathbb R\to\mathbb R$ is any function such that $A'=a$, and $\phi$ is the angle the tangent vector makes with the $x$-axis. We observe that
\begin{equation}
\frac{\partial I}{\partial x}=-a(x)\big(1-mf(x)\sin\phi\big).
\end{equation}
By the bounds \eqref{e:boundf} on $f$ and the bounds \eqref{e:bounda} on $a$, we see that there is $\varepsilon>0$ such that
\begin{equation}\label{e:boundI}
1-mf(x)\sin\phi\geq \varepsilon
\end{equation}
for all low speeds $m$. By the Implicit Function Theorem, we can express the variable $x$ as a function of $I$ and $\phi$, that is, $x=x(I,\phi,m)$. We have
\begin{equation}
\frac{\partial x}{\partial \phi}=-\frac{\frac{\partial I}{\partial \phi}}{\frac{\partial I}{\partial x}}=\frac{m\cos\phi}{1-mf(x)\sin\phi},\qquad \frac{\partial x}{\partial m}=-\frac{\frac{\partial I}{\partial m}}{\frac{\partial I}{\partial x}}=\frac{\sin\phi}{1-mf(x)\sin\phi}.
\end{equation}
We denote $x_0(I)=x(I,\phi,0)=(-A)^{-1}(I)$ and observe for later purposes that, expanding the function $m\mapsto f(x(I,\phi,m))$ at $m=0$, we get
\begin{equation}\label{e:fm}
\begin{aligned}
f(x(I,\phi,m))&=f(x_0(I))+mf'(x_0(I))\frac{\partial x}{\partial m}+o(m)\\
&=f(x_0(I))+mf'(x_0(I))\sin\phi+o(m),
\end{aligned}
\end{equation}
where the remainder is an $o(m)$ in the $C^1$-topology in the variable $I$, uniformly in $\phi$, by the $C^2$-boundedness of $f$ and \eqref{e:boundI}.

We consider the action functional
\begin{equation}
S\colon \mathbb R\to\mathbb R,\qquad S(c):=\int_{\{I=c\}}m\cos\phi\, \mathrm{d}x.
\end{equation}
By \cite{AB2021}, critical points of $S$ are in one-to-one correspondence with periodic magnetic geodesics. We compute
\begin{equation*}
\begin{aligned}
S(c)&=\int_0^{2\pi}\frac{m^2\cos^2\phi}{1-mf(x)\sin\phi}\mathrm{d}\phi\\
&=\int_0^{2\pi}m^2\cos^2\phi\Big(1+mf(x)\sin\phi+m^2f(x)^2\sin^2\phi\Big)\mathrm{d}\phi+o(m^4)\\
&=\int_0^{2\pi}m^2\cos^2\phi\Big(1+mf(x_0(c))\sin\phi+m^2\big(f'(x_0(c))+f(x_0(c))^2\big)\sin^2\phi\Big)\mathrm{d}\phi+o(m^4)\\
&=\pi m^2-\tfrac{\pi}{4}m^4K(x_0(c))+o(m^4),
\end{aligned}
\end{equation*}
where we used the geometric series expansion, the formula in \eqref{e:fm}, the Riccati equation \eqref{e:Riccati}, and standard integrals of trigonometric functions. Here, the remainder is an $o(m^4)$ in the $C^1$-topology in the variable $c\in \mathbb R$. By \eqref{e:boundcurv}, for every bounded interval $(c_-,c_+)$, there exists $m_*>0$ such that $S'(c)\neq0$ for all $c\in (c_-,c_+)$ and all $m\in(0,m_*)$. Indeed, $K'$ is bounded away from zero in any such interval. It follows that there are no contractible closed magnetic geodesics with speed in $(0,m_*)$ for the pair $(\bar g,\bar \beta)$ on the strip $(x_-,x_+)\times\mathbb R$, where $-A(x_\pm)=c_\pm$. Hence, taking a smooth orientation-preserving diffeomorphism $\psi: \mathbb R^2 \to (x_-,x_+)\times\mathbb R$ of the type $\psi(x,y)=(h(x),y)$ for all $(x,y)\in \mathbb R^2$, the pullback magnetic system given by 
\begin{equation}
g:=\psi^*\bar g,\qquad \beta:=\mu_g=\psi^*\bar\beta
\end{equation}
satisfies the property of Theorem \ref{t:B}, and is still invariant under vertical translations. Note that the metric $g$ in the theorem is not required to be complete.
\begin{rmk}
Even if it is not needed for our purposes, it is conceivable that we can choose the Gaussian curvature $K$ so that $S'(c)\neq0$ for all $m$ small enough and all $c\in\mathbb R$. This would require to show that there exists a constant $C>0$ such that
\begin{equation}
|f'(x)|\leq CK'(x),\qquad\forall\, x\in\mathbb R.
\end{equation}
Under this assumption there will be no contractible closed magnetic geodesics on the whole $\mathbb R^2$ for all small speeds for the \textit{complete} Riemannian metric $g=\mathrm{d}x^2+a(x)^2\mathrm{d}y^2$ and the magnetic form $\beta=\mu_g=a(x) \mathrm{d}x\wedge \mathrm{d}y$.
\end{rmk}

\bibliography{bib}

@article{Abraham0,
 author = {Abraham, Ralph},
 title = {Transversality in manifolds of mappings},
 fjournal = {Bulletin of the American Mathematical Society},
 journal = {Bull. Am. Math. Soc.},
 issn = {0002-9904},
 volume = {69},
 pages = {470--474},
 year = {1963},
 language = {English},
 doi = {10.1090/S0002-9904-1963-10969-6},
 zbMATH = {3273749},
 Zbl = {0171.44501}
}

@article{AB2021,
 author = {Asselle, Luca and Benedetti, Gabriele},
 title = {Integrable magnetic flows on the two-torus: {Zoll} examples and systolic inequalities},
 fjournal = {The Journal of Geometric Analysis},
 journal = {J. Geom. Anal.},
 issn = {1050-6926},
 volume = {31},
 number = {3},
 pages = {2924--2940},
 year = {2021},
 language = {English},
 doi = {10.1007/s12220-020-00379-1},
 zbMATH = {7328225},
 Zbl = {1480.70024}
}

@misc{Hasselblatt,
 author = {Hasselblatt, Boris and Wang, Jincheng},
 title = {Surfaces with nonpositive magnetic curvature
},
 year = {2026},
 howpublished = {Preprint, {arXiv}:2608.13534 [math.{DS}]},
 url = {https://arxiv.org/abs/2608.13534},
 arXiv = {arXiv:2608.13534}
}

@misc{RTW,
 author = {Rankin, Shane and Terek, Ivo and Weed, David},
 title = {The curvature of left-invariant magnetic systems
},
 year = {2026},
 howpublished = {Preprint, {arXiv}:2607.09107 [math.{DG}]},
 url = {https://arxiv.org/abs/2607.09107},
 arXiv = {arXiv:2607.09107}
}

@article{Gouda,
 author = {Gouda, Norio},
 title = {Magnetic flows of {Anosov} type},
 fjournal = {T{\^o}hoku Mathematical Journal. Second Series},
 journal = {T{\^o}hoku Math. J. (2)},
 issn = {0040-8735},
 volume = {49},
 number = {2},
 pages = {165--183},
 year = {1997},
 language = {English},
 doi = {10.2748/tmj/1178225145},
 zbMATH = {1054461},
 Zbl = {0938.37011}
}

@article{Gro99,
 author = {Grognet, St{\'e}phane},
 title = {Magnetic flow on negative curvature},
 fjournal = {Ergodic Theory and Dynamical Systems},
 journal = {Ergodic Theory Dyn. Syst.},
 issn = {0143-3857},
 volume = {19},
 number = {2},
 pages = {413--436},
 year = {1999},
 language = {French},
 doi = {10.1017/S0143385799126634},
 zbMATH = {1339310},
 Zbl = {0935.53037}
}

@article{Grognet,
 author = {Grognet, St{\'e}phane},
 title = {Entropies of magnetic flows},
 fjournal = {Annales de l'Institut Henri Poincar{\'e}. Physique Th{\'e}orique},
 journal = {Ann. Inst. Henri Poincar{\'e}, Phys. Th{\'e}or.},
 issn = {0246-0211},
 volume = {71},
 number = {4},
 pages = {395--424},
 year = {1999},
 language = {French},
 url = {https://eudml.org/doc/76840},
 zbMATH = {1383917},
 Zbl = {1131.37300}
}

@incollection{PP96,
 author = {Paternain, Gabriel P. and Paternain, Miguel},
 title = {Anosov geodesic flows and twisted symplectic structures},
 booktitle = {1st international conference on dynamical systems, Montevideo, Uruguay, 1995 - a tribute to Ricardo Ma\~n\'e. Proceedings},
 isbn = {0-582-30296-X},
 pages = {132--145},
 year = {1996},
 publisher = {Harlow: Longman},
 language = {English},
 zbMATH = {978569},
 Zbl = {0868.58062}
}

@article{Pat06,
 author = {Paternain, Gabriel P.},
 title = {Magnetic rigidity of horocycle flows},
 fjournal = {Pacific Journal of Mathematics},
 journal = {Pac. J. Math.},
 issn = {1945-5844},
 volume = {225},
 number = {2},
 pages = {301--323},
 year = {2006},
 language = {English},
 doi = {10.2140/pjm.2006.225.301},
 zbMATH = {5170778},
 Zbl = {1116.37020}
}

@article{Pat09,
 author = {Paternain, Gabriel P.},
 title = {Helicity and the {Ma{\~n}{\'e}} critical value},
 fjournal = {Algebraic \& Geometric Topology},
 journal = {Algebr. Geom. Topol.},
 issn = {1472-2747},
 volume = {9},
 number = {3},
 pages = {1413--1422},
 year = {2009},
 language = {English},
 doi = {10.2140/agt.2009.9.1413},
 zbMATH = {5597150},
 Zbl = {1181.53035}
}

@article{Wojtkowski,
 author = {Wojtkowski, Maciej P.},
 title = {Magnetic flows and {Gaussian} thermostats on manifolds of negative curvature},
 fjournal = {Fundamenta Mathematicae},
 journal = {Fundam. Math.},
 issn = {0016-2736},
 volume = {163},
 number = {2},
 pages = {177--191},
 year = {2000},
 language = {English},
 doi = {10.4064/fm-163-2-177-191},
 url = {https://eudml.org/doc/212437},
 zbMATH = {1445844},
 Zbl = {0997.37011}
}

@book{Milnor,
 author = {Milnor, John W.},
 title = {Morse theory. {Based} on lecture notes by {M}. {Spivak} and {R}. {Wells}},
 fseries = {Annals of Mathematics Studies},
 series = {Ann. Math. Stud.},
 volume = {51},
 year = {1963},
 publisher = {Princeton University Press, Princeton, NJ},
 language = {English},
 doi = {10.1515/9781400881802},
 zbMATH = {3176330},
 Zbl = {0108.10401}
}

@misc{Abraham,
 author = {Abraham, Ralph and Marsden, Jerrold E.},
 title = {Foundations of mechanics. 2nd ed., rev., enl., and reset. {With} the assistance of {Tudor} {Ratiu} and {Richard} {Cushman}},
 year = {1978},
 language = {English},
 howpublished = {Reading, {Massachusetts}: {The} {Benjamin}/{Cummings} {Publishing} {Company}, {Inc}., {Advanced} {Book} {Program}. m-{XVI}, {XXII}, 806 p. \$ 36.50 (1978).},
 zbMATH = {3610787},
 Zbl = {0393.70001}
}

@article {Har,
    AUTHOR = {Kharlamov, Mikhail P.},
     TITLE = {Some applications of differential geometry in the theory of
              mechanical systems},
   JOURNAL = {Mekh. Tverd. Tela},
  FJOURNAL = {Akademiya Nauk Ukrainsko\u i\ SSR. Institut Prikladno\u i\
              Matematiki i Mekhaniki. Mekhanika Tverdogo Tela},
    NUMBER = {11},
      YEAR = {1979},
     PAGES = {37--49, 118},
      ISSN = {0321-1975},
   MRCLASS = {70G25 (34C35 58Fxx)},
  MRNUMBER = {536269},
MRREVIEWER = {J.\ S.\ Joel},
}

@article{PS,
 author = {Palais, Richard S. and Smale, Steve},
 title = {A generalized {Morse} theory},
 fjournal = {Bulletin of the American Mathematical Society},
 journal = {Bull. Am. Math. Soc.},
 issn = {0002-9904},
 volume = {70},
 pages = {165--172},
 year = {1964},
 language = {English},
 doi = {10.1090/S0002-9904-1964-11062-4},
 zbMATH = {3193984},
 Zbl = {0119.09201}
}

@misc{BP,
 author = {Bohr, Jan and Paternain, Gabriel P.},
 title = {Zoll magnetic structures and ruled surfaces},
 year = {2026},
 howpublished = {Preprint, {arXiv}:2606.25173 [math.{DG}]},
 url = {https://arxiv.org/abs/2606.25173},
 arXiv = {arXiv:2606.25173}
}

@article{Tai2010,
 author = {Taimanov, Iskander A.},
 title = {Periodic magnetic geodesics on almost every energy level via variational methods},
 fjournal = {Regular and Chaotic Dynamics},
 journal = {Regul. Chaotic Dyn.},
 issn = {1560-3547},
 volume = {15},
 number = {4-5},
 pages = {598--605},
 year = {2010},
 language = {English},
 doi = {10.1134/S1560354710040131},
 zbMATH = {5836822},
 Zbl = {1207.58015}
}

@article{Koh,
 author = {Koh, Dennis},
 title = {On the evolution equation for magnetic geodesics},
 fjournal = {Calculus of Variations and Partial Differential Equations},
 journal = {Calc. Var. Partial Differ. Equ.},
 issn = {0944-2669},
 volume = {36},
 number = {3},
 pages = {453--480},
 year = {2009},
 language = {English},
 doi = {10.1007/s00526-009-0237-2},
 zbMATH = {5655365},
 Zbl = {1183.58024}
}

@article{PP,
 author = {Paternain, Gabriel P. and Paternain, Miguel},
 title = {Critical values of autonomous {Lagrangian} systems},
 fjournal = {Commentarii Mathematici Helvetici},
 journal = {Comment. Math. Helv.},
 issn = {0010-2571},
 volume = {72},
 number = {3},
 pages = {481--499},
 year = {1997},
 language = {English},
 doi = {10.1007/s000140050029},
 zbMATH = {1121242},
 Zbl = {0921.58017}
}

@article{Mane0,
 author = {Ma{\~n}{\'e}, Ricardo},
 title = {Lagrangian flows: the dynamics of globally minimizing orbits},
 fjournal = {Boletim da Sociedade Brasileira de Matem{\'a}tica. Nova S{\'e}rie},
 journal = {Bol. Soc. Bras. Mat., Nova S{\'e}r.},
 issn = {0100-3569},
 volume = {28},
 number = {2},
 pages = {141--153},
 year = {1997},
 language = {English},
 doi = {10.1007/BF01233389},
 zbMATH = {1089741},
 Zbl = {0892.58064}
}

@article{Mane1,
 author = {Contreras, Gonzalo and Delgado, Jorge and Iturriaga, Renato},
 title = {Lagrangian flows: the dynamics of globally minimizing orbits. {II}},
 fjournal = {Boletim da Sociedade Brasileira de Matem{\'a}tica. Nova S{\'e}rie},
 journal = {Bol. Soc. Bras. Mat., Nova S{\'e}r.},
 issn = {0100-3569},
 volume = {28},
 number = {2},
 pages = {155--196},
 year = {1997},
 language = {English},
 doi = {10.1007/BF01233390},
 zbMATH = {1089742},
 Zbl = {0892.58065}
}

@book{Sorrentino,
 author = {Sorrentino, Alfonso},
 title = {Action-minimizing methods in {Hamiltonian} dynamics. {An} introduction to {Aubry}-{Mather} theory},
 fseries = {Mathematical Notes (Princeton)},
 series = {Math. Notes (Princeton)},
 volume = {50},
 isbn = {978-0-691-16450-2; 978-1-400-86661-8},
 year = {2015},
 publisher = {Princeton, NJ: Princeton University Press},
 language = {English},
 doi = {10.1515/9781400866618},
 zbMATH = {6443553},
 Zbl = {1373.37002}
}

@article{Zuddas0,
 author = {Musina, Roberta and Zuddas, Fabio},
 title = {Embedded loops in the hyperbolic plane with prescribed, almost constant curvature},
 fjournal = {Annals of Global Analysis and Geometry},
 journal = {Ann. Global Anal. Geom.},
 issn = {0232-704X},
 volume = {55},
 number = {3},
 pages = {509--528},
 year = {2019},
 language = {English},
 doi = {10.1007/s10455-018-9638-9},
 zbMATH = {7050116},
 Zbl = {1412.53059}
}

@incollection{CoraMusina,
 author = {Cora, Gabriele and Musina, Roberta},
 title = {Planar loops with prescribed curvature via {Hardy}'s inequality},
 booktitle = {Friends in partial differential equations. The Nina N. Uraltseva 90th anniversary volume},
 isbn = {978-3-98547-094-5; 978-3-98547-594-0},
 pages = {89--109},
 year = {2025},
 publisher = {Berlin: European Mathematical Society (EMS)},
 language = {English},
 doi = {10.4171/NNU90/4},
 zbMATH = {8132147}
}

@book{CEM,
 author = {Cieliebak, Kai and Eliashberg, Yakov and Mishachev, Nikolai},
 title = {Introduction to the {{\(h\)}}-principle},
 edition = {2nd edition},
 fseries = {Graduate Studies in Mathematics},
 series = {Grad. Stud. Math.},
 issn = {1065-7339},
 volume = {239},
 isbn = {978-1-4704-6105-8; 978-1-4704-7617-5; 978-1-4704-7618-2},
 year = {2024},
 publisher = {Providence, RI: American Mathematical Society (AMS)},
 language = {English},
 doi = {10.1090/gsm/239},
 zbMATH = {7792729},
 Zbl = {1531.58008}
}

@misc{BBS,
 author = {Benedetti, Gabriele and Bimmermann, Johanna and Sanjay, Samanyu},
 title = {On the rigidity of {H}amiltonians which are {Z}oll near a minimum, with an application to magnetic systems and almost-{K}ähler manifolds},
 year = {2026},
 howpublished = {Preprint, {arXiv}:2604.06951 [math.{SG}]},
 url = {https://arxiv.org/abs/2604.06951},
 arXiv = {arXiv:2604.06951}
}

@article{Jean,
 author = {Jeanjean, Louis},
 title = {On the existence of bounded {Palais}-{Smale} sequences and application to a {Landesman}-{Lazer}-type problem set on {{\(\mathbb{R}^N\)}}},
 fjournal = {Proceedings of the Royal Society of Edinburgh. Section A. Mathematics},
 journal = {Proc. R. Soc. Edinb., Sect. A, Math.},
 issn = {0308-2105},
 volume = {129},
 number = {4},
 pages = {787--809},
 year = {1999},
 language = {English},
 doi = {10.1017/S0308210500013147},
 zbMATH = {1348082},
 Zbl = {0935.35044}
}

@article{Struwe88,
 author = {Struwe, Michael},
 title = {The existence of surfaces of constant mean curvature with free boundaries},
 fjournal = {Acta Mathematica},
 journal = {Acta Math.},
 issn = {0001-5962},
 volume = {160},
 number = {1-2},
 pages = {19--64},
 year = {1988},
 language = {English},
 doi = {10.1007/BF02392272},
 zbMATH = {4054471},
 Zbl = {0646.53005}
}

@article{Groman,
 author = {Groman, Yoel and Merry, Will J.},
 title = {The symplectic cohomology of magnetic cotangent bundles},
 fjournal = {Commentarii Mathematici Helvetici},
 journal = {Comment. Math. Helv.},
 issn = {0010-2571},
 volume = {98},
 number = {2},
 pages = {365--424},
 year = {2023},
 language = {English},
 doi = {10.4171/CMH/555},
 zbMATH = {7755954},
 Zbl = {1529.53079}
}

@article{FS,
 author = {Frauenfelder, Urs and Schlenk, Felix},
 title = {Hamiltonian dynamics on convex symplectic manifolds},
 fjournal = {Israel Journal of Mathematics},
 journal = {Isr. J. Math.},
 issn = {0021-2172},
 volume = {159},
 pages = {1--56},
 year = {2007},
 language = {English},
 doi = {10.1007/s11856-007-0037-3},
 url = {opus.bibliothek.uni-augsburg.de/opus4/files/56443/0303282v1.pdf},
 zbMATH = {5186271},
 Zbl = {1126.53056}
}

@article{Polterovich,
 author = {Polterovich, Leonid},
 title = {Geometry on the group of {Hamiltonian} diffeomorphisms},
 fjournal = {Documenta Mathematica},
 journal = {Doc. Math.},
 issn = {1431-0635},
 volume = {Extra Vol.},
 pages = {401--410},
 year = {1998},
 language = {English},
 doi = {10.4171/dms/1-2/39},
 url = {https://eudml.org/doc/226555},
 zbMATH = {1184339},
 Zbl = {0909.58004}
}

@article{AL,
 author = {Asselle, Luca and Lange, Christian},
 title = {On the rigidity of {Zoll} magnetic systems on surfaces},
 fjournal = {Nonlinearity},
 journal = {Nonlinearity},
 issn = {0951-7715},
 volume = {33},
 number = {7},
 pages = {3173--3194},
 year = {2020},
 language = {English},
 doi = {10.1088/1361-6544/ab839c},
 zbMATH = {7208085},
 Zbl = {1444.37050}
}

@article{Ben16,
 author = {Benedetti, Gabriele},
 title = {The contact property for symplectic magnetic fields on {{\(S^{2}\)}}},
 fjournal = {Ergodic Theory and Dynamical Systems},
 journal = {Ergodic Theory Dyn. Syst.},
 issn = {0143-3857},
 volume = {36},
 number = {3},
 pages = {682--713},
 year = {2016},
 language = {English},
 doi = {10.1017/etds.2014.82},
 zbMATH = {6578154},
 Zbl = {1343.37053}
}

@article{Ketover,
 author = {Ketover, Daniel and Liokumovich, Yevgeny},
 title = {On the existence of closed {{\(C^{1, 1}\)}} curves of constant curvature},
 fjournal = {Calculus of Variations and Partial Differential Equations},
 journal = {Calc. Var. Partial Differ. Equ.},
 issn = {0944-2669},
 volume = {62},
 number = {9},
 pages = {13},
 note = {Id/No 246},
 year = {2023},
 language = {English},
 doi = {10.1007/s00526-023-02584-6},
 zbMATH = {7748720},
 Zbl = {1527.53002}
}

@incollection{Arnold97,
 author = {Arnol'd, Vladimir I.},
 title = {Remarks concerning the {Morse} theory of a divergence-free vector field, the averaging method, and the motion of a charged particle in a magnetic field},
 booktitle = {Dinamicheskie sistemy i smezhnye voprosy. Sbornik statej. K 60-letiyu so dnya rozhdeniya akademika D. V. Anosova},
 isbn = {5-02-003713-3},
 pages = {9--19},
 year = {1997},
 publisher = {Moskva: Nauka. MAIK Nauka},
 language = {English},
 zbMATH = {1193132},
 Zbl = {0923.58010}
}

@article{San0,
 author = {Raymond, Nicolas and Ng{\d{o}}c, San V{\~u}},
 title = {Geometry and spectrum in {2D} magnetic wells},
 fjournal = {Annales de l'Institut Fourier},
 journal = {Ann. Inst. Fourier},
 issn = {0373-0956},
 volume = {65},
 number = {1},
 pages = {137--169},
 year = {2015},
 language = {English},
 doi = {10.5802/aif.2927},
 zbMATH = {6496536},
 Zbl = {1327.81207}
}

@article{Tai15,
 author = {Taimanov, Iskander A.},
 title = {On an integrable magnetic geodesic flow on the two-torus},
 fjournal = {Regular and Chaotic Dynamics},
 journal = {Regul. Chaotic Dyn.},
 issn = {1560-3547},
 volume = {20},
 number = {6},
 pages = {667--678},
 year = {2015},
 language = {English},
 doi = {10.1134/S1560354715060039},
 zbMATH = {6580505},
 Zbl = {1342.53108}
}

@article{Braun,
title = {Particle motions in a magnetic field},
journal = {Journal of Differential Equations},
volume = {8},
number = {2},
pages = {294-332},
year = {1970},
issn = {0022-0396},
doi = {https://doi.org/10.1016/0022-0396(70)90009-4},
url = {https://www.sciencedirect.com/science/article/pii/0022039670900094},
author = {Martin Braun}
}

@article{Castilho,
 author = {Castilho, C{\'e}sar},
 title = {The motion of a charged particle on a {Riemannian} surface under a non-zero magnetic field},
 fjournal = {Journal of Differential Equations},
 journal = {J. Differ. Equations},
 issn = {0022-0396},
 volume = {171},
 number = {1},
 pages = {110--131},
 year = {2001},
 language = {English},
 doi = {10.1006/jdeq.2000.3836},
 zbMATH = {1599083},
 Zbl = {0984.35156}
}

@article{San1,
 author = {Nguyen, Duc Th{\d{o}} and Raymond, Nicolas and Ng{\d{o}}c, San V{\~u}},
 title = {Boundary effects on the magnetic {Hamiltonian} dynamics in two dimensions},
 fjournal = {L'Enseignement Math{\'e}matique. 2e S{\'e}rie},
 journal = {Enseign. Math. (2)},
 issn = {0013-8584},
 volume = {64},
 number = {3-4},
 pages = {353--369},
 year = {2018},
 language = {English},
 doi = {10.4171/LEM/64-3/4-7},
 zbMATH = {7117763},
 Zbl = {1435.70034}
}

@article{Arnold63,
 author = {Arnol'd, Vladimir I.},
 title = {Small denominators and problems of stability of motion in classical and celestial mechanics},
 fjournal = {Russian Mathematical Surveys},
 journal = {Russ. Math. Surv.},
 issn = {0036-0279},
 volume = {18},
 number = {6},
 pages = {85--191},
 year = {1963},
 language = {English},
 doi = {10.1070/rm1963v018n06ABEH001143},
 zbMATH = {3220205},
 Zbl = {0135.42701}
}

@article{Sternberg,
 author = {Sternberg, Shlomo},
 title = {Minimal coupling and the symplectic mechanics of a classical particle in the presence of a {Yang}-{Mills} field},
 fjournal = {Proceedings of the National Academy of Sciences of the United States of America},
 journal = {Proc. Natl. Acad. Sci. USA},
 issn = {0027-8424},
 volume = {74},
 number = {12},
 pages = {5253--5254},
 year = {1977},
 language = {English},
 doi = {10.1073/pnas.74.12.5253},
 zbMATH = {196513},
 Zbl = {0765.58010}
}

@misc{Souriau,
 author = {Souriau, Jean-Marie},
 title = {Structure des syst{\`e}mes dynamiques.},
 year = {1970},
 language = {French},
 howpublished = {Ma{\^{\i}}trises de math{\'e}matiques. {Dunod}-{Universit{\'e}}. {Paris}: {Dunod}. xxxii, 414 p. (1970).},
 zbMATH = {3297923},
 Zbl = {0186.58001}
}

@article{Siburg,
 author = {Peyerimhoff, Norbert and Siburg, Karl F.},
 title = {The dynamics of magnetic flows for energies above {Ma{\~n}{\'e}}'s critical value},
 fjournal = {Israel Journal of Mathematics},
 journal = {Isr. J. Math.},
 issn = {0021-2172},
 volume = {135},
 pages = {269--298},
 year = {2003},
 language = {English},
 doi = {10.1007/BF02776061},
 zbMATH = {2053068},
 Zbl = {1051.37034}
}

@incollection{Givental,
 author = {Arnol'd, Vladimir I. and Givental', Alexander B.},
 title = {Symplectic geometry},
 booktitle = {Dynamical systems. IV. Symplectic geometry and its applications. Transl. from the Russian by G. Wassermann},
 isbn = {3-540-17003-0},
 pages = {1},
 year = {1985},
 publisher = {Berlin etc.: Springer-Verlag},
 language = {English},
 zbMATH = {486689},
 Zbl = {0780.58016}
}

@article{Anosov,
 author = {Anosov, Dmitri V. and Sina{\u{\i}}, Yakov G.},
 title = {Some smooth ergodic systems. {With} an appendix by {G}. {A}. {Margulis}},
 fjournal = {Uspekhi Matematicheskikh Nauk [N. S.]},
 journal = {Usp. Mat. Nauk},
 issn = {0042-1316},
 volume = {22},
 number = {5(137)},
 pages = {107--172},
 year = {1967},
 language = {Russian},
 doi = {10.1070/RM1967v022n05ABEH001228},
 zbMATH = {3283739},
 Zbl = {0177.42002}
}

@article{Burns,
 author = {Burns, Keith and Paternain, Gabriel P.},
 title = {Anosov magnetic flows, critical values and topological entropy.},
 fjournal = {Nonlinearity},
 journal = {Nonlinearity},
 issn = {0951-7715},
 volume = {15},
 number = {2},
 pages = {281--314},
 year = {2002},
 language = {English},
 doi = {10.1088/0951-7715/15/2/305},
 zbMATH = {1753660},
 Zbl = {1161.37337}
}

@article{Kozlov,
 author = {Kozlov, Valery V.},
 title = {Variational calculus in the large and classical mechanics},
 fjournal = {Uspekhi Matematicheskikh Nauk [N. S.]},
 journal = {Usp. Mat. Nauk},
 issn = {0042-1316},
 volume = {40},
 number = {2(242)},
 pages = {33--60},
 year = {1985},
 language = {Russian},
 zbMATH = {3888976},
 Zbl = {0557.70025}
}

@article{Montgomery,
 author = {Montgomery, Richard},
 title = {Infinitely many syzygies},
 fjournal = {Archive for Rational Mechanics and Analysis},
 journal = {Arch. Ration. Mech. Anal.},
 issn = {0003-9527},
 volume = {164},
 number = {4},
 pages = {311--340},
 year = {2002},
 language = {English},
 doi = {10.1007/s00205-002-0211-z},
 zbMATH = {1882096},
 Zbl = {1024.70005}
}

@article{Morin,
 author = {Morin, L{\'e}o},
 title = {A semiclassical {Birkhoff} normal form for constant-rank magnetic fields},
 fjournal = {Analysis \& PDE},
 journal = {Anal. PDE},
 issn = {2157-5045},
 volume = {17},
 number = {5},
 pages = {1593--1632},
 year = {2024},
 language = {English},
 doi = {10.2140/apde.2024.17.1593},
 zbMATH = {7898587},
 Zbl = {1546.35131}
}

@article{Caponio,
 author = {Caponio, Erasmo},
 title = {Time-like solutions to the {Lorentz} force equation in time-dependent electromagnetic and gravitational fields},
 fjournal = {Journal of Differential Equations},
 journal = {J. Differ. Equations},
 issn = {0022-0396},
 volume = {199},
 number = {1},
 pages = {115--142},
 year = {2004},
 language = {English},
 doi = {10.1016/j.jde.2003.09.001},
 zbMATH = {2105463},
 Zbl = {1054.83010}
}

@article {AbbondandoloMajer2008,
    AUTHOR = {Abbondandolo, Alberto and Majer, Pietro},
     TITLE = {A {M}orse complex for {L}orentzian geodesics},
   JOURNAL = {Asian J. Math.},
  FJOURNAL = {Asian Journal of Mathematics},
    VOLUME = {12},
      YEAR = {2008},
    NUMBER = {3},
     PAGES = {299--319},
      ISSN = {1093-6106,1945-0036},
   MRCLASS = {58E10 (53C22 53C50)},
  MRNUMBER = {2453558},
MRREVIEWER = {Alessandro\ Portaluri},
       DOI = {10.4310/AJM.2008.v12.n3.a3},
       URL = {https://doi.org/10.4310/AJM.2008.v12.n3.a3},
}

@incollection{Ginzburg96b,
 author = {Ginzburg, Viktor L.},
 title = {On closed trajectories of a charge in a magnetic field. {An} application of symplectic geometry},
 booktitle = {Contact and symplectic geometry},
 isbn = {0-521-57086-7},
 pages = {131--148},
 year = {1996},
 publisher = {Cambridge: Cambridge University Press},
 language = {English},
 zbMATH = {956493},
 Zbl = {0873.58034}
}

@book{GeodesicFlows_Paternain,
 author = {Paternain, Gabriel P.},
 title = {Geodesic flows},
 fseries = {Progress in Mathematics},
 series = {Prog. Math.},
 issn = {0743-1643},
 volume = {180},
 isbn = {0-8176-4144-0},
 year = {1999},
 publisher = {Boston, MA: Birkh{\"a}user},
 language = {English},
 zbMATH = {1353524},
 Zbl = {0930.53001}
}

@article {SacksUhlenback,
    AUTHOR = {Sacks, Jonathan and Uhlenbeck, Karen},
     TITLE = {The existence of minimal immersions of two-spheres},
   JOURNAL = {Bull. Amer. Math. Soc.},
  FJOURNAL = {Bulletin of the American Mathematical Society},
    VOLUME = {83},
      YEAR = {1977},
    NUMBER = {5},
     PAGES = {1033--1036},
      ISSN = {0002-9904},
   MRCLASS = {58E15},
  MRNUMBER = {448408},
MRREVIEWER = {A.\ J.\ Tromba},
       DOI = {10.1090/S0002-9904-1977-14366-8},
       URL = {https://doi.org/10.1090/S0002-9904-1977-14366-8},
}

@article {BahriTaimanov,
    AUTHOR = {Bahri, Abbas and Taimanov, Iskander A.},
     TITLE = {Periodic orbits in magnetic fields and {R}icci curvature of
              {L}agrangian systems},
   JOURNAL = {Trans. Amer. Math. Soc.},
  FJOURNAL = {Transactions of the American Mathematical Society},
    VOLUME = {350},
      YEAR = {1998},
    NUMBER = {7},
     PAGES = {2697--2717},
      ISSN = {0002-9947,1088-6850},
   MRCLASS = {58E10 (53C80 58F22 70G05)},
  MRNUMBER = {1458315},
MRREVIEWER = {Jair\ Koiller},
       DOI = {10.1090/S0002-9947-98-02108-4},
       URL = {https://doi.org/10.1090/S0002-9947-98-02108-4},
}

@article {Struwe90,
    AUTHOR = {Struwe, Michael},
     TITLE = {Existence of periodic solutions of {H}amiltonian systems on
              almost every energy surface},
   JOURNAL = {Bol. Soc. Brasil. Mat. (N.S.)},
  FJOURNAL = {Boletim da Sociedade Brasileira de Matem\'atica. Nova S\'erie},
    VOLUME = {20},
      YEAR = {1990},
    NUMBER = {2},
     PAGES = {49--58},
      ISSN = {0100-3569},
   MRCLASS = {58F22 (34C25 58F05 70H05)},
  MRNUMBER = {1143173},
MRREVIEWER = {Gabriella\ Tarantello},
       DOI = {10.1007/BF02585433},
       URL = {https://doi.org/10.1007/BF02585433},
}

@article {AABMT,
    AUTHOR = {Abbondandolo, Alberto and Asselle, Luca and Benedetti,
              Gabriele and Mazzucchelli, Marco and Taimanov, Iskander A.},
     TITLE = {The multiplicity problem for periodic orbits of magnetic flows
              on the 2-sphere},
   JOURNAL = {Adv. Nonlinear Stud.},
  FJOURNAL = {Advanced Nonlinear Studies},
    VOLUME = {17},
      YEAR = {2017},
    NUMBER = {1},
     PAGES = {17--30},
      ISSN = {1536-1365,2169-0375},
   MRCLASS = {37J45 (58E05)},
  MRNUMBER = {3604943},
MRREVIEWER = {Adel\ Daouas},
       DOI = {10.1515/ans-2016-6003},
       URL = {https://doi.org/10.1515/ans-2016-6003},
}

@article {AM2019,
    AUTHOR = {Asselle, Luca and Mazzucchelli, Marco},
     TITLE = {On {T}onelli periodic orbits with low energy on surfaces},
   JOURNAL = {Trans. Amer. Math. Soc.},
  FJOURNAL = {Transactions of the American Mathematical Society},
    VOLUME = {371},
      YEAR = {2019},
    NUMBER = {5},
     PAGES = {3001--3048},
      ISSN = {0002-9947,1088-6850},
   MRCLASS = {37J45 (58E05)},
  MRNUMBER = {3896104},
MRREVIEWER = {Chun-Lei\ Tang},
       DOI = {10.1090/tran/7185},
       URL = {https://doi.org/10.1090/tran/7185},
}

@article {AB2017,
    AUTHOR = {Asselle, Luca and Benedetti, Gabriele},
     TITLE = {On the periodic motions of a charged particle in an
              oscillating magnetic field on the two-torus},
   JOURNAL = {Math. Z.},
  FJOURNAL = {Mathematische Zeitschrift},
    VOLUME = {286},
      YEAR = {2017},
    NUMBER = {3-4},
     PAGES = {843--859},
      ISSN = {0025-5874,1432-1823},
   MRCLASS = {37J45 (58E05)},
  MRNUMBER = {3671563},
MRREVIEWER = {Zhi-Qiang\ Wang},
       DOI = {10.1007/s00209-016-1787-6},
       URL = {https://doi.org/10.1007/s00209-016-1787-6},
}

@article {AB2015,
    AUTHOR = {Asselle, Luca and Benedetti, Gabriele},
     TITLE = {Infinitely many periodic orbits in non-exact oscillating
              magnetic fields on surfaces with genus at least two for almost
              every low energy level},
   JOURNAL = {Calc. Var. Partial Differential Equations},
  FJOURNAL = {Calculus of Variations and Partial Differential Equations},
    VOLUME = {54},
      YEAR = {2015},
    NUMBER = {2},
     PAGES = {1525--1545},
      ISSN = {0944-2669,1432-0835},
   MRCLASS = {37J45 (58E05)},
  MRNUMBER = {3396422},
MRREVIEWER = {Thomas\ Bartsch},
       DOI = {10.1007/s00526-015-0834-1},
       URL = {https://doi.org/10.1007/s00526-015-0834-1},
}

@article {AB2022,
    AUTHOR = {Asselle, Luca and Benedetti, Gabriele},
     TITLE = {Non-resonant circles for strong magnetic fields on surfaces},
   JOURNAL = {Ann. H. Lebesgue},
  FJOURNAL = {Annales Henri Lebesgue},
    VOLUME = {5},
      YEAR = {2022},
     PAGES = {1191--1211},
      ISSN = {2644-9463},
   MRCLASS = {37J40 (58E10)},
  MRNUMBER = {4526250},
       DOI = {10.5802/ahl.147},
       URL = {https://doi.org/10.5802/ahl.147},
}

@article {Gong2023,
    AUTHOR = {Gong, Wenmin},
     TITLE = {Infinitely many noncontractible closed magnetic geodesics on
              non-compact manifolds},
   JOURNAL = {Differential Geom. Appl.},
  FJOURNAL = {Differential Geometry and its Applications},
    VOLUME = {87},
      YEAR = {2023},
     PAGES = {Paper No. 101977, 34},
      ISSN = {0926-2245,1872-6984},
   MRCLASS = {58E10 (53C22)},
  MRNUMBER = {4546043},
MRREVIEWER = {Iskander\ A.\ Taimanov},
       DOI = {10.1016/j.difgeo.2023.101977},
       URL = {https://doi.org/10.1016/j.difgeo.2023.101977},
}

@article {Usher,
    AUTHOR = {Usher, Michael},
     TITLE = {Floer homology in disk bundles and symplectically twisted
              geodesic flows},
   JOURNAL = {J. Mod. Dyn.},
  FJOURNAL = {Journal of Modern Dynamics},
    VOLUME = {3},
      YEAR = {2009},
    NUMBER = {1},
     PAGES = {61--101},
      ISSN = {1930-5311,1930-532X},
   MRCLASS = {53D40 (53D25)},
  MRNUMBER = {2481333},
MRREVIEWER = {Tobias\ Ekholm},
       DOI = {10.3934/jmd.2009.3.61},
       URL = {https://doi.org/10.3934/jmd.2009.3.61},
}

@article {GG2009,
    AUTHOR = {Ginzburg, Viktor L. and G\"urel, Ba\c sak Z.},
     TITLE = {Periodic orbits of twisted geodesic flows and the
              {W}einstein-{M}oser theorem},
   JOURNAL = {Comment. Math. Helv.},
  FJOURNAL = {Commentarii Mathematici Helvetici. A Journal of the Swiss
              Mathematical Society},
    VOLUME = {84},
      YEAR = {2009},
    NUMBER = {4},
     PAGES = {865--907},
      ISSN = {0010-2571,1420-8946},
   MRCLASS = {53D40 (37J05)},
  MRNUMBER = {2534483},
MRREVIEWER = {Sonja\ Hohloch},
       DOI = {10.4171/CMH/184},
       URL = {https://doi.org/10.4171/CMH/184},
}

@article {GG2004,
    AUTHOR = {Ginzburg, Viktor L. and G\"urel, Ba\c sak Z.},
     TITLE = {Relative {H}ofer-{Z}ehnder capacity and periodic orbits in
              twisted cotangent bundles},
   JOURNAL = {Duke Math. J.},
  FJOURNAL = {Duke Mathematical Journal},
    VOLUME = {123},
      YEAR = {2004},
    NUMBER = {1},
     PAGES = {1--47},
      ISSN = {0012-7094,1547-7398},
   MRCLASS = {53D40 (37J05 37J45)},
  MRNUMBER = {2060021},
MRREVIEWER = {Kai\ Cieliebak},
       DOI = {10.1215/S0012-7094-04-12311-5},
       URL = {https://doi.org/10.1215/S0012-7094-04-12311-5},
}

@article {Schlenk,
    AUTHOR = {Schlenk, Felix},
     TITLE = {Applications of {H}ofer's geometry to {H}amiltonian dynamics},
   JOURNAL = {Comment. Math. Helv.},
  FJOURNAL = {Commentarii Mathematici Helvetici. A Journal of the Swiss
              Mathematical Society},
    VOLUME = {81},
      YEAR = {2006},
    NUMBER = {1},
     PAGES = {105--121},
      ISSN = {0010-2571,1420-8946},
   MRCLASS = {53D35 (37J05 37J45)},
  MRNUMBER = {2208800},
MRREVIEWER = {Jelena\ Kati\'c},
       DOI = {10.4171/CMH/45},
       URL = {https://doi.org/10.4171/CMH/45},
}

@article {Osuna2005,
    AUTHOR = {Osuna, Osvaldo},
     TITLE = {Periodic orbits of weakly exact magnetic flows},
     JOURNAL = {Nonl. {A}nalysis and {D}ifferential {E}quations},
     VOL = {1},
     NUMBER = {1-4},
     Pages = {93--100},
     YEAR = {2013},
}

@article {Macarini2004,
    AUTHOR = {Macarini, Leonardo},
     TITLE = {Hofer-{Z}ehnder capacity and {H}amiltonian circle actions},
   JOURNAL = {Commun. Contemp. Math.},
  FJOURNAL = {Communications in Contemporary Mathematics},
    VOLUME = {6},
      YEAR = {2004},
    NUMBER = {6},
     PAGES = {913--945},
      ISSN = {0219-1997,1793-6683},
   MRCLASS = {53D35 (37J05 37J45)},
  MRNUMBER = {2112475},
MRREVIEWER = {Martin\ Pinsonnault},
       DOI = {10.1142/S0219199704001550},
       URL = {https://doi.org/10.1142/S0219199704001550},
}

@article {GGM,
    AUTHOR = {Ginzburg, Viktor L. and G\"urel, Basak Z. and Macarini,
              Leonardo},
     TITLE = {On the {C}onley conjecture for {R}eeb flows},
   JOURNAL = {Internat. J. Math.},
  FJOURNAL = {International Journal of Mathematics},
    VOLUME = {26},
      YEAR = {2015},
    NUMBER = {7},
     PAGES = {1550047, 22},
      ISSN = {0129-167X,1793-6519},
   MRCLASS = {53D40 (37J10 37J55 53D25)},
  MRNUMBER = {3357036},
MRREVIEWER = {Darko\ Milinkovi\'c},
       DOI = {10.1142/S0129167X15500470},
       URL = {https://doi.org/10.1142/S0129167X15500470},
}

@article {AMMP,
    AUTHOR = {Abbondandolo, Alberto and Macarini, Leonardo and Mazzucchelli,
              Marco and Paternain, Gabriel P.},
     TITLE = {Infinitely many periodic orbits of exact magnetic flows on
              surfaces for almost every subcritical energy level},
   JOURNAL = {J. Eur. Math. Soc. (JEMS)},
  FJOURNAL = {Journal of the European Mathematical Society (JEMS)},
    VOLUME = {19},
      YEAR = {2017},
    NUMBER = {2},
     PAGES = {551--579},
      ISSN = {1435-9855,1435-9863},
   MRCLASS = {37J45 (58E05)},
  MRNUMBER = {3605025},
MRREVIEWER = {Maria\ Letizia\ Bertotti},
       DOI = {10.4171/JEMS/674},
       URL = {https://doi.org/10.4171/JEMS/674},
}

@article {AMP,
    AUTHOR = {Abbondandolo, Alberto and Macarini, Leonardo and Paternain,
              Gabriel P.},
     TITLE = {On the existence of three closed magnetic geodesics for
              subcritical energies},
   JOURNAL = {Comment. Math. Helv.},
  FJOURNAL = {Commentarii Mathematici Helvetici. A Journal of the Swiss
              Mathematical Society},
    VOLUME = {90},
      YEAR = {2015},
    NUMBER = {1},
     PAGES = {155--193},
      ISSN = {0010-2571,1420-8946},
   MRCLASS = {58E10 (37D40 53C22)},
  MRNUMBER = {3317337},
MRREVIEWER = {Iskander\ A.\ Taimanov},
       DOI = {10.4171/CMH/350},
       URL = {https://doi.org/10.4171/CMH/350},
}

@article {CMP,
    AUTHOR = {Contreras, Gonzalo and Macarini, Leonardo and Paternain,
              Gabriel P.},
     TITLE = {Periodic orbits for exact magnetic flows on surfaces},
   JOURNAL = {Int. Math. Res. Not.},
  FJOURNAL = {International Mathematics Research Notices},
      YEAR = {2004},
    NUMBER = {8},
     PAGES = {361--387},
      ISSN = {1073-7928,1687-0247},
   MRCLASS = {37J45 (37C27 37J05 53D35)},
  MRNUMBER = {2036336},
MRREVIEWER = {Ely\ Kerman},
       DOI = {10.1155/S1073792804205050},
       URL = {https://doi.org/10.1155/S1073792804205050},
}

@article {CIPP2000,
    AUTHOR = {Contreras, Gonzalo and Iturriaga, Renato and Paternain, Gabriel P. and
              Paternain, Miguel},
     TITLE = {The {P}alais-{S}male condition and {M}a\~n\'e's critical
              values},
   JOURNAL = {Ann. Henri Poincar\'e},
  FJOURNAL = {Annales Henri Poincar\'e. A Journal of Theoretical and
              Mathematical Physics},
    VOLUME = {1},
      YEAR = {2000},
    NUMBER = {4},
     PAGES = {655--684},
      ISSN = {1424-0637,1424-0661},
   MRCLASS = {37J50 (37J45 58E10)},
  MRNUMBER = {1785184},
MRREVIEWER = {Vittorio\ Coti Zelati},
       DOI = {10.1007/PL00001011},
       URL = {https://doi.org/10.1007/PL00001011},
}

@article {Taimanov92B,
    AUTHOR = {Taimanov, Iskander A.},
     TITLE = {Closed non-self-intersecting extremals of multivalued
              functionals},
   JOURNAL = {Sibirsk. Mat. Zh.},
  FJOURNAL = {Rossi\u iskaya Akademiya Nauk. Sibirskoe Otdelenie. Sibirski\u
              i\ Matematicheski\u i\ Zhurnal},
    VOLUME = {33},
      YEAR = {1992},
    NUMBER = {4},
     PAGES = {155--162, 223},
      ISSN = {0037-4474},
   MRCLASS = {49J10 (49L15 49Q20)},
  MRNUMBER = {1185445},
MRREVIEWER = {Claudia\ Simionescu-Badea},
       DOI = {10.1007/BF00971134},
       URL = {https://doi.org/10.1007/BF00971134},
}

@article {Taimanov92A,
    AUTHOR = {Taimanov, Iskander A.},
     TITLE = {Closed extremals on two-dimensional manifolds},
   JOURNAL = {Uspekhi Mat. Nauk},
  FJOURNAL = {Uspekhi Matematicheskikh Nauk},
    VOLUME = {47},
      YEAR = {1992},
    NUMBER = {2(284)},
     PAGES = {143--185, 223},
      ISSN = {0042-1316,2305-2872},
   MRCLASS = {58E10 (53C22)},
  MRNUMBER = {1185286},
MRREVIEWER = {J.\ S.\ Joel},
       DOI = {10.1070/RM1992v047n02ABEH000880},
       URL = {https://doi.org/10.1070/RM1992v047n02ABEH000880},
}

@article {Novikov82,
    AUTHOR = {Novikov, Sergei P.},
     TITLE = {The {H}amiltonian formalism and a multivalued analogue of
              {M}orse theory},
   JOURNAL = {Uspekhi Mat. Nauk},
  FJOURNAL = {Akademiya Nauk SSSR i Moskovskoe Matematicheskoe Obshchestvo.
              Uspekhi Matematicheskikh Nauk},
    VOLUME = {37},
      YEAR = {1982},
    NUMBER = {5(227)},
     PAGES = {3--49, 248},
      ISSN = {0042-1316},
   MRCLASS = {58E05 (58E30 58F05)},
  MRNUMBER = {676612},
MRREVIEWER = {P.\ Ver Eecke},
}

@article {Hofer-Viterbo88,
    AUTHOR = {Hofer, Helmut and Viterbo, Claude},
     TITLE = {The {W}einstein conjecture in cotangent bundles and related
              results},
   JOURNAL = {Ann. Scuola Norm. Sup. Pisa Cl. Sci. (4)},
  FJOURNAL = {Annali della Scuola Normale Superiore di Pisa. Classe di
              Scienze. Serie IV},
    VOLUME = {15},
      YEAR = {1988},
    NUMBER = {3},
     PAGES = {411--445},
      ISSN = {0391-173X,2036-2145},
   MRCLASS = {58F22 (58F05 70H05)},
  MRNUMBER = {1015801},
MRREVIEWER = {Vittorio\ Coti Zelati},
       URL = {http://www.numdam.org/item?id=ASNSP_1988_4_15_3_411_0},
}

@incollection {Arnold86,
    AUTHOR = {Arnol'd, Vladimir I.},
     TITLE = {First steps of symplectic topology},
 BOOKTITLE = {V{III}th international congress on mathematical physics
              ({M}arseille, 1986)},
     PAGES = {1--16},
 PUBLISHER = {World Sci. Publishing, Singapore},
      YEAR = {1987},
      ISBN = {9971-50-208-9},
   MRCLASS = {58F05 (78A10)},
  MRNUMBER = {915560},
       DOI = {10.1070/rm1986v041n06abeh004221},
       URL = {https://doi.org/10.1070/rm1986v041n06abeh004221},
}

@article {Ginzburg87,
    AUTHOR = {Ginzburg, Viktor L.},
     TITLE = {New generalizations of {P}oincar\'e's geometric theorem},
   JOURNAL = {Funktsional. Anal. i Prilozhen.},
  FJOURNAL = {Akademiya Nauk SSSR. Funktsional\cprime ny\u i\ Analiz i ego
              Prilozheniya},
    VOLUME = {21},
      YEAR = {1987},
    NUMBER = {2},
     PAGES = {16--22, 96},
      ISSN = {0374-1990},
   MRCLASS = {58F05 (58A10 78A35)},
  MRNUMBER = {902290},
}

@article{Jakub,
 author = {Jakubczyk, Bronisław and Kry{\'n}ski, Wojciech},
 title = {Conjugate points of dynamic pairs and control systems},
 fjournal = {European Series in Applied and Industrial Mathematics (ESAIM): Control, Optimization and Calculus of Variations},
 journal = {ESAIM, Control Optim. Calc. Var.},
 issn = {1292-8119},
 volume = {31},
 pages = {26},
 note = {Id/No 7},
 year = {2025},
 language = {English},
 doi = {10.1051/cocv/2024079},
 zbMATH = {8008441},
 Zbl = {1569.93051}
}

@article{Agrachev,
 author = {Agrachev, Andrei A.},
 title = {The curvature and hyperbolicity of {Hamiltonian} systems},
 fjournal = {Proceedings of the Steklov Institute of Mathematics},
 journal = {Proc. Steklov Inst. Math.},
 issn = {0081-5438},
 volume = {256},
 pages = {26--46},
 year = {2007},
 language = {English},
 doi = {10.1134/S0081543807010026},
 zbMATH = {5493494},
 Zbl = {1153.37346}
}

@Article{Abbondandolo1,
 Author = {Abbondandolo, Alberto},
 Title = {Lectures on the free period {Lagrangian} action functional},
 FJournal = {Journal of Fixed Point Theory and Applications},
 Journal = {J. Fixed Point Theory Appl.},
 ISSN = {1661-7738},
 Volume = {13},
 Number = {2},
 Pages = {397--430},
 Year = {2013},
 Language = {English},
 DOI = {10.1007/s11784-013-0128-1},
 zbMATH = {6251850},
 Zbl = {1328.37001}
}

@article{AsselleBenedetti2016,
 author = {Asselle, Luca and Benedetti, Gabriele},
 title = {The {Lusternik}-{Fet} theorem for autonomous {Tonelli} {Hamiltonian} systems on twisted cotangent bundles},
 fjournal = {Journal of Topology and Analysis},
 journal = {J. Topol. Anal.},
 issn = {1793-5253},
 volume = {8},
 number = {3},
 pages = {545--570},
 year = {2016},
 language = {English},
 doi = {10.1142/S1793525316500205},
 zbMATH = {6603943},
 Zbl = {1345.37059}
}

@article{Assenza2024,
 author = {Assenza, Valerio},
 title = {Magnetic curvature and existence of a closed magnetic geodesic on low energy levels},
 fjournal = {IMRN. International Mathematics Research Notices},
 journal = {Int. Math. Res. Not.},
 issn = {1073-7928},
 volume = {2024},
 number = {21},
 pages = {13586--13610},
 year = {2024},
 language = {English},
 doi = {10.1093/imrn/rnae209},
 zbMATH = {7985835}
}

@article{Assenza2025,
 author = {Assenza, Valerio and Marshall Reber, James and Terek, Ivo},
 title = {Magnetic flatness and {E}. {Hopf}'s theorem for magnetic systems},
 fjournal = {Communications in Mathematical Physics},
 journal = {Commun. Math. Phys.},
 issn = {0010-3616},
 volume = {406},
 number = {2},
 pages = {20},
 note = {Id/No 24},
 year = {2025},
 language = {English},
 doi = {10.1007/s00220-024-05166-5},
 zbMATH = {7966792}
}

@book{Lecturesonclosedgeodesics,
 author = {Klingenberg, Wilhelm},
 title = {Lectures on closed geodesics},
 fseries = {Grundlehren der Mathematischen Wissenschaften},
 series = {Grundlehren Math. Wiss.},
 issn = {0072-7830},
 volume = {230},
 year = {1978},
 publisher = {Springer, Cham},
 language = {English},
 zbMATH = {3617245},
 Zbl = {0397.58018}
}

@article{Merry2010,
 author = {Merry, Will J.},
 title = {Closed orbits of a charge in a weakly exact magnetic field},
 fjournal = {Pacific Journal of Mathematics},
 journal = {Pac. J. Math.},
 issn = {1945-5844},
 volume = {247},
 number = {1},
 pages = {189--212},
 year = {2010},
 language = {English},
 doi = {10.2140/pjm.2010.247.189},
 zbMATH = {5809035},
 Zbl = {1246.37082}
}

@article{Contreras2006,
 author = {Contreras, Gonzalo},
 title = {The {Palais}-{Smale} condition on contact type energy levels for convex {Lagrangian} systems},
 fjournal = {Calculus of Variations and Partial Differential Equations},
 journal = {Calc. Var. Partial Differ. Equ.},
 issn = {0944-2669},
 volume = {27},
 number = {3},
 pages = {321--395},
 year = {2006},
 language = {English},
 doi = {10.1007/s00526-005-0368-z},
 zbMATH = {5064715},
 Zbl = {1105.37037}
}

@article{AbboSchwarz,
 author = {Abbondandolo, Alberto and Schwarz, Matthias},
 title = {A smooth pseudo-gradient for the {Lagrangian} action functional},
 fjournal = {Advanced Nonlinear Studies},
 journal = {Adv. Nonlinear Stud.},
 issn = {1536-1365},
 volume = {9},
 number = {4},
 pages = {597--623},
 year = {2009},
 language = {English},
 zbMATH = {5635516},
 Zbl = {1185.37145}
}

@article{ARNOLD1,
 author = {Arnol'd, Vladimir I.},
 title = {Some remarks on flows of line elements and frames},
 fjournal = {Soviet Mathematics. Doklady},
 journal = {Sov. Math., Dokl.},
 issn = {0197-6788},
 volume = {2},
 pages = {562--564},
 year = {1961},
 language = {English},
 zbMATH = {3201217},
 Zbl = {0124.14601}
}

@article{CFP2010,
 author = {Cieliebak, Kai and Frauenfelder, Urs and Paternain, Gabriel P.},
 title = {Symplectic topology of {Ma{\~n}{\'e}}'s critical values},
 fjournal = {Geometry \& Topology},
 journal = {Geom. Topol.},
 issn = {1465-3060},
 volume = {14},
 number = {3},
 pages = {1765--1870},
 year = {2010},
 language = {English},
 doi = {10.2140/gt.2010.14.1765},
 zbMATH = {5769287},
 Zbl = {1239.53110}
}

@article{CIPP1998,
 author = {Contreras, Gonzalo and Iturriaga, Renato and Paternain, Gabriel P. and Paternain, Miguel},
 title = {Lagrangian graphs, minimizing measures and {Ma{\~n}{\'e}}'s critical values},
 fjournal = {Geometric and Functional Analysis. GAFA},
 journal = {Geom. Funct. Anal.},
 issn = {1016-443X},
 volume = {8},
 number = {5},
 pages = {788--809},
 year = {1998},
 language = {English},
 doi = {10.1007/s000390050074},
 zbMATH = {1223298},
 Zbl = {0920.58015}
}

@article{Caldiroli2024,
 author = {Caldiroli, Paolo and Capietto, Anna},
 title = {Planar closed curves with prescribed curvature},
 fjournal = {Discrete and Continuous Dynamical Systems},
 journal = {Discrete Contin. Dyn. Syst.},
 issn = {1078-0947},
 volume = {44},
 number = {10},
 pages = {3209--3221},
 year = {2024},
 language = {English},
 doi = {10.3934/dcds.2024056},
 zbMATH = {7901504},
 Zbl = {1552.53003}
}

@article{SarnataroStryker2026,
 author = {Sarnataro, Lorenzo and Stryker, Douglas},
 title = {Existence of closed embedded curves of constant curvature via min-max},
 fjournal = {Journal f{\"u}r die Reine und Angewandte Mathematik},
 journal = {J. Reine Angew. Math.},
 issn = {0075-4102},
 volume = {831},
 pages = {185--211},
 year = {2026},
 language = {English},
 doi = {10.1515/crelle-2025-0090},
 zbMATH = {8154176}
}

@article{ChengZhou1,
 author = {Cheng, Da Rong and Zhou, Xin},
 title = {Existence of curves with constant geodesic curvature in a {Riemannian} 2-sphere},
 fjournal = {Transactions of the American Mathematical Society},
 journal = {Trans. Am. Math. Soc.},
 issn = {0002-9947},
 volume = {374},
 number = {12},
 pages = {9007--9028},
 year = {2021},
 language = {English},
 doi = {10.1090/tran/8510},
 zbMATH = {7618823},
 Zbl = {1512.58008}
}

@article{Schneider1,
 author = {Schneider, Matthias},
 title = {Closed magnetic geodesics on {{\(S^2\)}}},
 fjournal = {Journal of Differential Geometry},
 journal = {J. Differ. Geom.},
 issn = {0022-040X},
 volume = {87},
 number = {2},
 pages = {343--388},
 year = {2011},
 language = {English},
 doi = {10.4310/jdg/1304514976},
 zbMATH = {5917819},
 Zbl = {1232.53006}
}

@article{ResenbergSchneider,
 author = {Rosenberg, Harold and Schneider, Matthias},
 title = {Embedded constant-curvature curves on convex surfaces},
 fjournal = {Pacific Journal of Mathematics},
 journal = {Pac. J. Math.},
 issn = {1945-5844},
 volume = {253},
 number = {1},
 pages = {213--219},
 year = {2011},
 language = {English},
 doi = {10.2140/pjm.2011.253.213},
 url = {msp.berkeley.edu/pjm/2011/253-1/p12.xhtml},
 zbMATH = {5997964},
 Zbl = {1242.53075}
}

@article{Kerman1,
 author = {Kerman, Ely},
 title = {Periodic orbits of {Hamiltonian} flows near symplectic critical submanifolds},
 fjournal = {IMRN. International Mathematics Research Notices},
 journal = {Int. Math. Res. Not.},
 issn = {1073-7928},
 volume = {1999},
 number = {17},
 pages = {953--969},
 year = {1999},
 language = {English},
 doi = {10.1155/S1073792899000501},
 zbMATH = {1385062},
 Zbl = {0958.37041}
}

@article{KermanGinzburg,
 author = {Ginzburg, Viktor L. and Kerman, Ely},
 title = {Periodic orbits of {Hamiltonian} flows near symplectic extrema.},
 fjournal = {Pacific Journal of Mathematics},
 journal = {Pac. J. Math.},
 issn = {1945-5844},
 volume = {206},
 number = {1},
 pages = {69--91},
 year = {2002},
 language = {English},
 doi = {10.2140/pjm.2002.206.69},
 zbMATH = {1868174},
 Zbl = {1055.37065}
}

@article{CGK2004,
 author = {Cieliebak, Kai and Ginzburg, Viktor L. and Kerman, Ely},
 title = {Symplectic homology and periodic orbits near symplectic submanifolds},
 fjournal = {Commentarii Mathematici Helvetici},
 journal = {Comment. Math. Helv.},
 issn = {0010-2571},
 volume = {79},
 number = {3},
 pages = {554--581},
 year = {2004},
 language = {English},
 doi = {10.1007/s00014-004-0814-0},
 zbMATH = {2113905},
 Zbl = {1073.53118}
}

@article{KirschLaurain2010,
 author = {Kirsch, Stephane and Laurain, Paul},
 title = {An obstruction to the existence of immersed curves of prescribed curvature},
 fjournal = {Potential Analysis},
 journal = {Potential Anal.},
 issn = {0926-2601},
 volume = {32},
 number = {1},
 pages = {29--39},
 year = {2010},
 language = {English},
 doi = {10.1007/s11118-009-9142-8},
 zbMATH = {5664591},
 Zbl = {1188.53004}
}

@article{EFP2024,
 author = {Enciso, Alberto and Fern{\'a}ndez, Antonio J. and Peralta-Salas, Daniel},
 title = {Small spheres with prescribed nonconstant mean curvature in {Riemannian} manifolds},
 fjournal = {Journal of Functional Analysis},
 journal = {J. Funct. Anal.},
 issn = {0022-1236},
 volume = {286},
 number = {11},
 pages = {16},
 note = {Id/No 110415},
 year = {2024},
 language = {English},
 doi = {10.1016/j.jfa.2024.110415},
 zbMATH = {7832862},
 Zbl = {1550.53031}
}

@book{Hirsch1976,
  author    = {Hirsch, Morris W.},
  title     = {Differential Topology},
  series    = {Graduate Texts in Mathematics},
  volume    = {33},
  publisher = {Springer-Verlag},
  address   = {New York},
  year      = {1976},
  doi       = {10.1007/978-1-4684-9449-5}
}

@misc{Brahim2026,
 author = {Louis-Brahim Beaufort},
 title = {Tensor tomography and frame flow ergodicity for magnetic flows in higher dimensions},
 year = {2026},
 howpublished = {Preprint, {arXiv}:2604.12495 [math.{DG}]},
 url = {https://arxiv.org/abs/2604.12495},
 arXiv = {arXiv:2604.12495}
}

@misc{Assouline1,
 author = {Rotem Assouline},
 title = {Magnetic {Brunn}-{Minkowski} inequalities},
 year = {2026},
 howpublished = {Preprint, {arXiv}:2606.08626 [math.{DG}]},
 url = {https://arxiv.org/abs/2606.08626},
 arXiv = {arXiv:2606.08626}
}

@misc{Assouline2,
 author = {Rotem Assouline},
 title = {Curvature-{Dimension} for {Autonomous} {Lagrangians}},
 year = {2026},
 howpublished = {Preprint, {arXiv}:2409.08001 [math.{DG}]},
 url = {https://arxiv.org/abs/2409.08001},
 arXiv = {arXiv:2409.08001}
}

@article{AssenzaTestolina2026,
 author = {Assenza, Valerio and Testolina, Giorgia},
 title = {Electromagnetic curvature via {Jacobi}-{Maupertuis} and beyond},
 fjournal = {Journal of Fixed Point Theory and Applications},
 journal = {J. Fixed Point Theory Appl.},
 issn = {1661-7738},
 volume = {28},
 number = {3},
 pages = {25},
 note = {Id/No 63},
 year = {2026},
 language = {English},
 doi = {10.1007/s11784-026-01314-7},
 zbMATH = {8239372}
}

@article{ADMT2026,
 author = {Assenza, Valerio and De Simoi, Jacopo and Marshall Reber, James and Terek, Ivo},
 title = {Marked length spectrum rigidity for {Anosov} magnetic surfaces},
 fjournal = {Advances in Mathematics},
 journal = {Adv. Math.},
 issn = {0001-8708},
 volume = {500},
 pages = {25},
 note = {Id/No 111091},
 year = {2026},
 language = {English},
 doi = {10.1016/j.aim.2026.111091},
 zbMATH = {8227236}
}

@article{ABB2024,
 author = {Asselle, Luca and Benedetti, Gabriele and Berti, Massimiliano},
 title = {Zoll magnetic systems on the two-torus: a {Nash}-{Moser} construction},
 fjournal = {Advances in Mathematics},
 journal = {Adv. Math.},
 issn = {0001-8708},
 volume = {452},
 pages = {39},
 note = {Id/No 109826},
 year = {2024},
 language = {English},
 doi = {10.1016/j.aim.2024.109826},
 zbMATH = {7916612},
 Zbl = {1556.37064}
}

@article{AB2022_normalform,
 author = {Asselle, Luca and Benedetti, Gabriele},
 title = {Normal forms for strong magnetic systems on surfaces: trapping regions and rigidity of {Zoll} systems},
 fjournal = {Ergodic Theory and Dynamical Systems},
 journal = {Ergodic Theory Dyn. Syst.},
 issn = {0143-3857},
 volume = {42},
 number = {6},
 pages = {1871--1897},
 year = {2022},
 language = {English},
 doi = {10.1017/etds.2021.11},
 zbMATH = {7543346},
 Zbl = {1513.70067}
}
\bibliographystyle{plain}
\end{document}